\documentclass{amsart}
\usepackage{amssymb,amsmath,amsthm,amsfonts,mathtools}
\usepackage[utf8]{inputenc} 
\usepackage[T1]{fontenc}    
\usepackage{hyperref}       
\usepackage{url}            
\usepackage{booktabs}       
\usepackage{lipsum}         
\usepackage{graphicx}
\usepackage[eprint=false, url=false, isbn=false]{biblatex}
\usepackage{doi}
\usepackage{tikz}
\usepackage{pgfplots}
\usepgfplotslibrary{patchplots}
\pgfplotsset{compat=1.18} 
\usepackage{caption}
\usepackage{subcaption}
\usepackage{subfiles}
\usepackage{xcolor}
\theoremstyle{plain}
\newtheorem{algo}{Algorithm}[] 
\newtheorem{thm}{Theorem}[] 
\newtheorem{lem}{Lemma}[section]
\newtheorem{prop}[lem]{Proposition}
\newtheorem{cor}[lem]{Corollary}
\theoremstyle{definition}
\newtheorem{defn}[lem]{Definition}

\theoremstyle{remark}
\newtheorem{rem}[lem]{Remark}

\theoremstyle{plain}
\usepackage{comment}
\newcommand{\fat}[1]{{\pmb{ #1}}}
\newcommand{\vv}{\tilde{\fat v}}

\newcommand*\diff{\mathop{}\!\mathrm{d}} 
\newcommand*\discreteDiff{\boldsymbol{d}}

\DeclareMathOperator{\normal}{\eta}
\newcommand{\weakto}{\rightharpoonup}
\newcommand{\weakstarto}{\overset{\ast}{\rightharpoonup}}
\newcommand{\abs}[1]{\left \lvert #1 \right \rvert}
\newcommand{\norm}[1]{\left\lVert#1\right\rVert}

\DeclareMathOperator{\tr}{tr}
\DeclareMathOperator{\BV}{BV}
\DeclareMathOperator{\identitymatrix}{I}

\DeclareMathOperator{\Cont}{\mathcal{C}}
\newcommand{\He}{H^1(\Omega; \mathbb{S}^2)}

\newcommand{\V}{H^1_{0,\sigma}(\Omega)}

\newcommand{\Ha}{{L}^2_{\sigma}(\Omega)}
\newcommand{\W}[2]{W^{#1,#2}\left(\Omega \right)}
\renewcommand{\L}[1]{L^{#1} \left(\Omega \right)}

\renewcommand{\H}[1]{H^{#1} \left(\Omega \right) }
\newcommand{\qspace}[1]{H^{1}_0 (\Omega)}
\newcommand{\velspace}{H^1_{0,\sigma}  (\Omega) \cap  H^2(\Omega)}
\newcommand{\sobnorm}[3]{\left\lVert#1\right\rVert_{\W{#2}{#3}}  }
\newcommand{\Hsobnorm}[2]{\left\lVert#1\right\rVert_{\H{#2}}  }

\newcommand{\lebnorm}[2]{\left\lVert#1\right\rVert_{\L{#2}}  }

\newcommand{\ltwonorm}[1]{\lebnorm{#1}{2}}
\newcommand{\hnorm}[1]{\left\lVert#1\right\rVert_{h}}

\newcommand{\mesh}{\mathcal{T}_h}
\newcommand{\facets}{\mathcal{F}_h}
\newcommand{\nodes}{\mathcal{N}_h}

\newcommand{\adlifting}[2]{{R}_h^{0,#2} (\jump{#1})}
\newcommand{\adliftingl}[3]{{R}_h^{#3,#2} (\jump{#1})}
\newcommand{\mixedfemspacecg}{V_h \times [\cgonezero]^{\dimension}}
\newcommand{\mixedfemspacedg}{V_h \times [\dgzero]^{\dimension}}

\newcommand{\projectcg}{\mathcal{P}^c_h}
\newcommand{\projectcgzero}{\mathcal{P}^c_{h,0}}
\newcommand{\projectdg}{\mathcal{P}^d_h}
\newcommand{\projectV}{\mathcal{Q}_h}
\newcommand{\interpol}{\mathcal{I}_{h} }
\newcommand{\jump}[1]{ \pmb{[} #1 \pmb{]}  }
\newcommand{\avg}[1]{ \pmb{\langle} #1 \pmb{\rangle}  }
\newcommand{\dg}{\mathrm{DG}_h}
\newcommand{\dgzero}{\mathrm{DG}_h^0}
\newcommand{\cg}{\mathrm{CG}_h}
\newcommand{\cgone}{\mathrm{CG}_h^1}
\newcommand{\cgonezero}{\mathrm{CG}_{h,0}^1}
\DeclareMathOperator{\dimension}{N}

\newcommand{\energy}{\mathcal{E}}
\newcommand{\energykin}{\energy_{\text{kin}}}

\newcommand{\energyeladG}{\energy_{\text{ela}}^{h}}
\newcommand{\energystabdG}{\energy_J^h}
 
\newcommand{\velh}{v_h} 
\newcommand{\dir}{{d}} 
\newcommand{\dirh}{{d}_h} 
\newcommand{\Deltah}{\Delta_h} 
\newcommand{\dbc}{d_h^{\Gamma}} 
\DeclareMathOperator{\discretedt}{\boldsymbol{d}_t}
\hypersetup{
pdftitle={Projection Methods Ericksen-Leslie, Reiter},
pdfsubject={35A35, 35Q35, 65M60, 76A15},
pdfauthor={Maximilian E.V.~Reiter},
pdfkeywords={Ericksen--Leslie, nematic liquid crystal flow, unit-norm constraint, geometric constraint, projection, Finite Element method, energy-variational solutions},
}
\begin{document}
\title{Projection Methods in the Context of Nematic Crystal Flow}
\author{Maximilian E.V.~Reiter}
\address{TU Berlin}
\email{m.reiter@tu-berlin.de}
\begin{abstract}
We present a continuous and a discontinuous linear Finite Element method based on a predictor-corrector scheme for the numerical approximation of the Ericksen--Leslie equations, a model for nematic liquid crystal flow including a non-convex unit-sphere constraint.
As predictor step we propose a linear semi-implicit Finite Element discretization which naturally offers a local orthogonality relation between the approximate director field and its time derivative.
Afterwards an explicit discrete projection onto the unit-sphere constraint is applied without increasing the modeled energy.
For the Finite Element approximation of the director field, we compare the usage of a discrete inner product, usually referred to as \textit{mass-lumping}, for a globally continuous, piecewise linear discretization to a piecewise constant, discontinuous Galerkin approach.
Discrete well-posedness results and energy laws are established. 
Conditional convergence of the approximate solutions to energy-variational solutions of the Ericksen-Leslie equations \cite{lasarzik-reiter23} is shown for a time-step restriction, see Theorems \ref{thm:main-result-subconvergence-cg} and \ref{thm:main-result-subconvergence-dg}.
Computational studies indicate the efficiency of the proposed linearization and the improved accuracy by including a projection step in the algorithm.
\end{abstract}
\maketitle
\section{Introduction}
Liquid crystals combine the properties of liquids and solid crystals. 
In their so called nematic phase, the material behaves like a liquid
and the molecules absent positional order.
However, due to the the rod-like shape of the molecules, they still experience orientational order exhibiting self-alignment along one direction. This allows anisotropic dynamics.
The Ericksen--Leslie model is used to mathematically describe the behavior of instationary nematic liquid crystal flow.
\\
By $\fat v: \overline{\Omega} \times [0,T] \to \mathbb{R}^3$ we denote the fluid's velocity, its pressure by $\fat p: \overline{\Omega} \times [0,T] \to \mathbb{R}$ and its director field, which has to be understood as the rod-like molecules' local average, by $\fat d: \overline{\Omega} \times [0,T] \to \mathbb{R}^3$.
We consider the Ericksen--Leslie model governed by the equations (cf. \cite{zhang-wang13})
\begin{subequations}\label{system}
\begin{align}
\begin{split}\label{system_a}
    \partial_t \fat v - \mu \Delta \fat v  + (\fat v \cdot \nabla) \fat v  + \nabla \fat p +\nabla \cdot\left ( (\nabla \fat d )^T\nabla \fat d  \right )    &= 0 \, ,
\end{split}
\\
\begin{split}\label{system_b}    
    \nabla \cdot \fat v &= 0 \, ,
\end{split}
\\
\begin{split}
    \label{system_c}
    \partial_t \fat d 
    +
    (\fat v \cdot \nabla) \fat d 
    + 
    \fat d \times (\fat d \times \Delta \fat d )
     &=0 \, ,
\end{split}
\\
\begin{split}\label{system_d}
\vert \fat d \vert & = 1  \, ,
\end{split}
\end{align}
\end{subequations}
where
the Oseen--Frank energy is simplified to the Dirichlet energy.
We assume that $\mu > 0$ in order to ensure the dissipative character of our model.
Further, we employ the initial and constant-in-time boundary conditions
\begin{align*}
    \fat v(0) &= \fat v_0, \quad 
    & \fat  d(0) = \fat d_0 \text{ with } 
    & \abs{\fat d_0} = 1 \text{ in } \Omega, \\
    \fat v &= 0, \quad 
    & \fat d =\fat  d_{0}  
    \text{ with } 
    & 
    \abs{\fat d_{0}  } = 1 
    \text{ on } \partial \Omega .
\end{align*}
Since the unit-sphere constraint is non-convex and challenging, it is often approximated by a Ginzburg--Landau penalization term added to the energy \cite{becker-feng-prohl08,lin-liu00,lin-liu95} such that equations \eqref{system_c} and \eqref{system_d} would be replaced by
\begin{equation}
\label{system_c-gl}
\partial_t \fat d + (\fat v \cdot \nabla) \fat d  -    \Delta \fat d
     + \frac{1}{\varepsilon} (\abs{ \fat d}^2 -1 ) \fat d
     =0 \, ,
\end{equation}
for some $\varepsilon>0$.
For this penalized model (\ref{system_a}, \ref{system_b}, \ref{system_c-gl}), existence of global weak solutions is well established \cite{lin-liu95}. 
Recently, it was shown that in two spatial dimensions the weak solutions to the penalized system (\ref{system_a}, \ref{system_b}, \ref{system_c-gl}) converge to a weak solution of the original model \eqref{system} as $\varepsilon \to 0 $, see \cite{kortum20}.
\\ 
For our system of equations \eqref{system}, the existence of global weak solutions has only been shown in two spatial dimensions and existence of global weak solutions remains an open question in three dimensions. However, there are results on local well-posedness \cite{zhang-wang13} and existence of global weak solutions under further assumptions, such as an upper hemisphere condition for the initial value of the director field \cite{lin_3d}.
Further, the existence of energy-variational solutions --- a superset of weak solutions --- was proven recently \cite{lasarzik-reiter23}. This is the framework we also follow in this work.
\\
The penalization approach \eqref{system_c-gl} also leads to some disadvantages from a numerical perspective. On the one hand, a further parameter $\varepsilon$ is introduced that has to be chosen jointly with the time-step size and the mesh-width. Balancing these simultaneously can be challenging, see \cite{becker-feng-prohl08, prohl01}.
On the other hand, the penalization term only gives a global bound on the fulfillment of the unit-sphere constraint and therewith allows for local deviations which in turn can lead to a higher error magnitude, see \cite{badia-et-al11}.
\\
Conserving the unit-sphere constraint implicitly --- using a (mass-lumped) midpoint discretization \cite{bartels-prohl06,lasarzik-reiter23} or using a saddle point approach \cite{badia-et-al11} --- leads to non-linear equation systems which increases the computational demand of the formulation. 
In the context of the Ericksen--Leslie equations, solving such a non-linear system iteratively using e.g. a fixed point solver can further lead to high restrictions on the time-step size \cite{reiter-ecmi}.
\\
This is what motivates the projection method as a predictor-corrector method. 
Abstaining from an implicit unit-norm conservation allows to linearize the system, increasing the computational efficiency.
The conformity with the unit-sphere constraint can be guaranteed afterwards by applying an explicit discrete projection as a correction step.
\\
To the author's knowledge the first use of a projection method for the numerical treatment of a unit-sphere constraint can be found in \cite{alouges97} as a strategy to find energy minimizers for stable configurations of liquid crystals.
In \cite{wang00} the first examination of such a projection method in the instationary case of the Landau-Lifshitz-Gilbert equations has been conducted in a solely semi-discrete setting keeping the spatial variables continuous.
In \cite{prohl01} the projection schemes \cite{alouges97, wang00} have been investigated after a reformulation to a global penalization approach.
That is: For a sequence of discrete solutions $(\fat d^{j})_j$ for $1\leq j \leq n$ and time step size $k>0$, the equations \eqref{system_c}, \eqref{system_d} are replaced by the semi-discretization in time
\begin{equation}\label{system-c-prohl}
\partial_t \fat d^{n+1} 
+
( \fat v^{n+1} \cdot \nabla) \fat d^{n+1}
+
\fat d^{n+1} \times [ \fat d^{n+1} \times (  \Delta \fat d^{n+1})]
+\frac{1}{k} (1- \frac{1}{\abs{\fat d^n}^{2-\gamma}})\fat d^n
=0 \, ,
\end{equation}
for some $\gamma \in \{0,1\}$. 
\\
The first transfer of a projection method to a \textit{Finite Element} setting has been developed in parallel in two publications.
In \cite{alouges06}, this transfer succeeds in the context of the Landau--Lifshitz equations using a stability criterion of the type $\frac{\Delta t}{(\Delta x)^p}$ without specifying the exact order $p$.
During the same time frame, \cite{bartels05} introduced the assumption of a weakly acute mesh as a sufficient condition for energy-decreasing projections in the context of harmonic maps.
Since then, the projection method has been widely used, e.g. \cite{barret-bartels-feng-prohl07,bartels09,bartels-feng-prohl08}. 
\\
Most of the mentioned publications describing convergent projection methods consider some orthogonality condition a part of their algorithm, see also \cite{bartels2015}.
This stems from a necessary condition for geometrically constrained partial differential equations.
Let a geometric constraint be implicitly defined by the zeros of a differentiable function $ g : \mathbb{R}^{\dimension} \to \mathbb{R}$ and let $\fat d :\mathbb{R}^{\dimension} \times [0,T] \to \mathbb{R}^{\dimension} $ be a smooth solution to some partial differential equation of the form $\fat d_t + F( \fat d) =0$, that also fulfills the constraint $g$ everywhere in time and space. The chain rule leads to the equality
\begin{equation*}
0 = \partial_t g( \fat d) = g'(\fat d) \cdot \fat d_t .
\end{equation*}
For the unit-sphere constraint, $g(.) = \abs{.}^2 -1$, this boils down to the identity
\begin{equation}\label{eq:orthogonality-of-geometrically-constrained-pdes}
0 = \fat d \cdot \fat d_t.
\end{equation}
In the context of the Ericksen--Leslie equations there are only two instances of a projection method known to the author.
The already mentioned \cite[Section 6]{prohl01} reformulates projection schemes to a penalization approach as in Equation \eqref{system-c-prohl}.
Secondly, in \cite[Section 5.3]{badia-et-al11} the projection method has been applied as a formal linearization of a non-linear saddle point constraint, however without a mathematical analysis. 
A discontinuous Galerkin approach on the other hand, has only been applied to the Ericksen--Leslie equations in \cite{ericksen-leslie-dg}. There it is used to achieve an increased local accuracy for the approximation of weak solutions with higher regularity.
\\
In \cite[p. 2]{bartels2015} it is well argued that "the projection step in predictor-corrector approaches can often be omitted".
In this work we will give some evidence that the projection step can nevertheless be advantageous in the context of nematic crystal flow.
The reason for this mainly stems from the fact that the projection yields $L^{\infty}(\Omega)$-bounds for the discrete director field. This regularity helps us to achieve the necessary \textit{a priori} estimates and the convergence of subsequences to energy-variational solutions, see Section \ref{sec:convergence-to-envar-sols-cg}.
\\
The remainder of the paper is structured as followed. In Section \ref{sec:preliminaries}, all necessary preliminaries will be introduced, especially the definition of energy-variational solutions and sufficient mesh conditions for energy-decreasing projections.
In Section \ref{sec:discrete-system-cg}, we introduce a Finite Element scheme and its properties based on a globally continuous, piecewise linear approximation of the director field.
In the following Section \ref{sec:convergence-to-envar-sols-cg}, convergence to energy-variational solutions in the sense of Definition \ref{def:envar} is proven. 
The following two Sections \ref{sec:discrete-system-dg} and \ref{sec:convergence-to-envar-sols-dg} repeat this, but this time for a Finite Element scheme based on an elementwise constant approximation of the director field.
Computational studies of both approaches are presented in the last section, Section \ref{sec:computational-studies}.
\section{Preliminaries and Notation}
\label{sec:preliminaries}
Throughout this work, we assume $\Omega \subset \mathbb{R}^{\dimension}$ for $\dimension=2,3$ to be a convex, polyhedral and bounded domain.
We use standard notation for Lebesgue, Sobolev and Bochner spaces.
By $\tr$ we denote the trace operator on Sobolev spaces.
The spaces $L^p_{\sigma} (\Omega)$ for $1 \leq p \leq \infty$ and $H^{1}_{0,\sigma} (\Omega)$ are defined as closures of the space of smooth, solenoidal and compactly supported functions $\Cont^{\infty}_{c,\sigma}(\Omega; \mathbb{R}^{\dimension})$ with respect to the norms of $L^p (\Omega)$ and $H^{1} (\Omega)$.
By $\Cont_w ([0,T], \mathbb{B})$ we define the space of weakly continuous functions for some Banach space $\mathbb{B}$.
The spaces $\BV ([0,T])$ and $\BV (\mathbb{R}^{\dimension})$ are the spaces of bounded variation in time and space, respectively.
The Euclidean unit sphere in $\dimension$ dimensions is referred to as $\mathbb{S}^{\dimension -1}$. 
By $\identitymatrix$ we denote the identity matrix in dimension $\dimension$.
We define the standard cross product in three spatial dimensions implicitly via a matrix multiplication, i.e.~$v \times w = [v]_{\times} w = [w]_{\times}^T v$ for vectors $v,w \in\mathbb{R}^3$. Coherently, we can define the cross product between a vector $v\in\mathbb{R}^3$ and matrix $A\in\mathbb{R}^{3\times 3}$ by $ v\times A = [v]_{\times} A$.
For a unit vector $d \in  \mathbb{R}^3$, we further note the identity
$$
-d\times (d\times v) = [d]_{\times}^T [d]_{\times} v
=
(\identitymatrix - d\otimes d )v
,
$$
for $v\in\mathbb{R}^3$, where $\otimes$ is the outer product. This identity allows us to interpret the double cross product also in two spatial dimensions in the following work. We will also make use of the triple product rule
\begin{equation}
    a \cdot (b \times c) = b \cdot (c \times a) = c \cdot (a \times b)
    \label{eq:triple-prodcut}
\end{equation}
for vectors $a,b,c\in\mathbb{R}^3$.
For asymptotic behaviour we use standard big \textit{O} notation.
\subsection{Solution Concept}
For the rest of this work, the total energy will be defined as
\begin{equation}
    \label{eq:definition-energy}
    \mathcal{E} (\fat v, \fat d) 
    \coloneqq
    \mathcal{E}_{kin} (\fat v) 
    +
    \mathcal{E}_{ela} (\fat d) 
    =
    \frac{1}{2} \lebnorm{\fat v}{2}^2
    +
    \frac{1}{2} \lebnorm{ \nabla \fat d}{2}^2
    .
\end{equation}
In the penalized case, for Equations (\ref{system_a}, \ref{system_b}, \ref{system_c-gl}), a basic energy law allows to bound the Laplacian of the director field $\Delta \fat d \in L^2(0,T;\L{2})$, see \cite{lin-liu95}. Under sufficient smoothness of the domain's boundary, standard elliptic estimates then yield the regularity of the director field to be $\fat d \in L^2(0,T;\H{2})$, such that compact embeddings deliver the strong convergence of subsequences needed to take the limit of the Ericksen stress tensor $\nabla \fat d ^T \nabla \fat d$. In our case, for Equations \eqref{system}, one only obtains the bound $\fat d \times \Delta \fat d \in L^2(0,T;\L{2})$, see also Equation \eqref{eq:discrete-energy-inequality-projection}.
To make up for the lack of compactness, we introduce an auxiliary variable $E$ that is an upper bound for our total energy $\mathcal{E}(\fat v, \fat d) $. The difference of these two, $E - \mathcal{E}(\fat v, \fat d) $, can be interpreted as a form of defect, measuring the difference of the strong and weak limits of our approximate solutions.
The scaled difference is then added to a variational inequality to convexify the Ericksen stress tensor.
In order to do so, we also have to reformulate the Ericksen stress tensor by
\begin{align*}
\int_{\Omega} (\nabla \fat d^T \nabla \fat d ): \nabla \vv \diff  x 
=
    -  \int_{\Omega} ( \vv \cdot \nabla ) \fat d\cdot  [\fat d ]_{\times} ^T  [\fat d ]_{\times} \Delta \fat d \diff  x
\end{align*} 
for $\fat d \in H^2(\Omega; \mathbb{S} ^2)$ and $\vv \in \V$, see \cite[Remark 1]{lasarzik-reiter23}.
Note that these regularity assumptions would be fulfilled in the case of global weak solutions to the penalized System (\ref{system_a}, \ref{system_b}, \ref{system_c-gl}).
\begin{defn}[Energy-variational solutions]\label{def:envar}
We call ${(\fat v, \fat d , E )}$ an energy-va\-riational solution, if 
\begin{align}
\begin{split}
&\fat v \in \Cont_w([0,T];\Ha) \cap L^2(0,T;\V)\,, \quad E \in \BV([0,T]) \,,\\
&\fat d \in  \Cont_w([0,T];\He) \cap H^1(0,T;L^{3/2}(\Omega;\mathbb{R}^3) ) 
\text{ such that }
\\&
\fat d \times \Delta \fat d \in L^2(0,T;L^2(\Omega;\mathbb{R}^3))  \, ,
\end{split}
\label{reg}
\end{align}
and $E$ is non-increasing as well as an upper bound for the energy, i.e.~$E \geq \mathcal{E} (\fat v,\fat d ) $ on $[0,T]$ with $E (0) = \mathcal{E} (\fat v(0),\fat d (0))  $. Further, the unit-sphere constraint has to be fulfilled, $ \vert \fat d  \vert =1$ a.e.~in $\Omega \times (0,T)$. The term $ \fat d \times \Delta \fat d$ has to be understood as the weak divergence of $\fat d \times \nabla \fat d $.
The triple ${(\fat v, \fat d , E )}$ fulfills the energy-variational inequality 
\begin{multline}
\left
[ 
     E -  \int_{\Omega} \fat v \cdot\vv  
    \diff x 
 \right 
 ] 
 \Big \vert_{s-}^t 
 \\
 + \int_s^t   \int_{\Omega} \fat v \cdot \partial_t \vv 
  -  ( \fat v \cdot \nabla) \fat v \cdot  \vv   
  + ( \fat d \times ( \vv \cdot \nabla ) \fat d) \cdot ( \fat d \times  \Delta \fat d)  \diff x \diff \tau 
\\
+\int_s^t \int_{\Omega} \mu 
\left (  \nabla \fat v 
  \right )
  :
  (\nabla \fat v - \nabla \vv) 
  \diff x \diff \tau 
+ \vert \fat d \times \Delta \fat d \vert^2 
     \diff x \diff \tau 
    \\
     + \int_s^t 
 \mathcal{K}( \vv ) 
 \left [
    \mathcal{E} (\fat v,\fat d ) - E 
 \right ]  
 \diff \tau  \leq 0 
 . 
 \label{envarform}
\end{multline}
 for all $s$, $t\in (0,T)$ and for all $\vv \in \Cont^1([0,T]; (\V)^*) \cap L^2(0,T;\V)$ such that $ \mathcal{K}(\vv) \in L^1(0,T)$. The initial values, $\fat v(0) = \fat v_0 $, $ \fat d(0) = \fat d_0$, are fulfilled in a weak sense and $\fat d$ fulfills the constant-in-time inhomogeneous Dirichlet boundary conditions in the sense of the trace $ \tr(\fat d(t)) = \tr (\fat d_0)$ for a.e.~$t\in(0,T)$. 
Additionally, it holds 
\begin{align}
\partial_t \fat d + ( \fat v \cdot \nabla ) \fat d + (\identitymatrix - \fat d \otimes \fat d)  \Delta \fat d  = 0 \label{weak:d}
\end{align}
a.e. in $\Omega \times (0,T)$. 
In this work, we  choose $ \mathcal{K}(\vv) = \frac{1}{2}\Vert \vv \Vert_{L^{\infty}(\Omega)}^2$.
\end{defn}
\begin{rem}[Strong continuity of initial values, cf. \cite{eiter-lasarzik24,lasarzik21,lasarzik-reiter23}]
Due to the weak lower semi-continuity of the energy $\mathcal{E}$ and the monotony of $E$ we observe that
\begin{align*}
\mathcal{E} (\fat v (0), \fat d (0))
\leq 
\liminf_{t \searrow 0} \mathcal{E} (\fat v (t) , \fat d (t))
\leq
\liminf_{t \searrow 0} E (t)
\leq
E(0)
=
\mathcal{E} (\fat v (0), \fat d (0)) 
.
\end{align*}
All inequalities are therefore in fact equalities.
This implies norm convergence of the map $t \mapsto (\fat v (t) , \fat d (t))$ in $L^2(\Omega) \times H^1_0(\Omega)$. Together with the weak continuity \eqref{reg}, this yields strong convergence to the initial values, i.e.~$\lim_{t \searrow 0}(\fat v (t) , \fat d (t)) = (\fat v_0 , \fat d_0)$ in $L^2(\Omega) \times H^1_0(\Omega)$.
\end{rem}
\begin{rem}[Properties]
Energy-variational solutions possess a lot of desirable properties: Every weak solution is an energy-variational solution with $E(\tau) = \mathcal{E} (\fat v (\tau),\fat d (\tau)) $ for a.e.~$\tau \in (s,t)$. But also vice versa, $E(\tau) = \mathcal{E} (\fat v (\tau),\fat d (\tau))$ for a.e.~$\tau \in (s,t)$ implies that an energy-variational solution is indeed a weak solution, which can be confirmed by a rescaling argument (cf. \cite{lasarzik21,lasarzik-reiter23}).
Every energy-variational solution fulfills an Energy--Dissipation inequality, which follows from testing the variational Inequality \eqref{envarform} with $\vv =0$ (cf. \cite{eiter-lasarzik24}).
Further, every energy-variational solution is a dissipative solution implying that it fulfills the weak-strong uniqueness property (cf. \cite{agosti-etal23,eiter-etal23,eiter-lasarzik24,lasarzik-reiter23}). 
Lastly, energy-variational solutions fulfill the physically relevant semi-flow property: Every restriction of a solution and every concatenation of two solutions with non-increasing $E$ at the concatenation point are again energy-variational solutions (cf. \cite{agosti-etal23,eiter-lasarzik24}).
\end{rem}
\subsection{Discrete Preliminaries}
By $k$ we denote the time-step size and by $h$ the spatial discretization parameter.
The discrete time derivative of a sequence of approximate solutions $(u^j)_{j}$ for $0\leq j \leq N$ is denoted by $\discreteDiff_t u^j \coloneqq k^{-1} (u^j - u^{j-1})$ for $1\leq j \leq N$.
We will frequently make use of a discrete product rule for $(u^j)_{j}, (v^j)_{j}$, i.e.
\begin{equation}\label{eq:discrete-integration-by-parts-in-time}
\discreteDiff_t (u^j \cdot v^j)
= 
\discreteDiff_t (u^j) \cdot v^j
+
u^{j-1} \cdot \discreteDiff_t (v^j).
\end{equation}
By $C$ we denote generic constants independent of the discretization parameters $k$ and $h$. In particular, we use the symbol $\lesssim$ for inequalities up to a generic constant: $A\lesssim B \Leftrightarrow \exists C>0 \text{ s.t.~}A \leq C B$.
\subsection{Mesh}
We assume $\mesh$ to be a \textit{quasi-uniform} (in the sense of \cite[Def. 4.4.13]{scott-brenner08}) subdivision of the domain $\Omega$ into triangles for $N=2$ or tetrahedra for $N=3$ respectively.
All $T \in \mesh$ have at least one node that is not on the boundary of the domain $\partial \Omega$. The set of nodes of a subdivision $\mesh$ is denoted by $\nodes$ with the the inner nodes $\nodes^i$ and the nodes on the boundary $\nodes^D$. The nodes contained in the closure of some cell $T \in \mesh$ are referred to as $\nodes^T$.
The facets of the mesh are denoted by $\facets$, again with $\facets^i$ for interior facets and $\facets^D$ for the facets on the boundary.
For a cell $T \in \mesh$, we denote its facets by $\facets^T$.
The facet normal is denoted by $\normal_F$ where as the outer normal of a cell $T$ at facet $F$ is denoted by $\normal_{T,F}$. Accordingly, they only differ in their sign.
The set of cells sharing facet $F\in\facets$ is denoted by $\mathcal{M}(F)$.
The quasi-uniformity assumption allows us to estimate the asymptotic scaling of the cells and facets, i.e.
\begin{align*}
    C_{\dimension} (\rho h)^{\dimension} &\leq \abs{T} \leq C_{\dimension} h^{\dimension}
    \\
    C_{\dimension-1} (\rho h)^{\dimension-1} &\leq \abs{F} \leq C_{\dimension-1}  h^{\dimension-1}
\end{align*}
for all $T\in \mesh$ and $F\in \facets$, where $\rho>0$ is the according \textit{quasi-uniformity} constant and $C_N$ is some generic constant depending only on $\dimension = 2,3$.
\\
The barycenter of a facet $F\in \facets$ is denoted by $x_F$. By $x_T\in \stackrel{\circ}{T}$ we index the \textit{cell centers} of every cell $T\in\mesh$. The cell centers are not necessarily the barycenters of the cells. Instead, we later assume some properties regarding the relation of the facet barycenters $x_F$ and the cell centers $x_T$. 
\\
We introduce the notions of a \textit{non-obtuse} and \textit{weakly acute} mesh.
\begin{defn}[Non-Obtuse Mesh]
\label{def:non-obtuse}
A domain decomposition $\mesh$ into triangles ($\dimension =2$) or tetrahedra ($\dimension =3$) is called \textit{non-obtuse} if for all distinct facets $F,\hat{F}\in\facets^T$ of a cell $T\in\mesh$, it holds that
\begin{align}\label{eq:property-non-obtuse-mesh}
    \normal_{T,F}
    \cdot
    \normal_{T,\hat{F}}
    \leq 0
    .
\end{align}
\end{defn}
This means that a \textit{non-obtuse} mesh consists of triangles or tetrahedra such that the angle between edges in two spatial dimensions or the dihedral angle between two faces in three dimensions may not exceed $\frac{\pi}{2}$.
\begin{defn}[Weakly Acute Mesh, see \cite{bartels09}]
\label{def:weakly-acute}
Let $\mesh$ be a domain decomposition into triangles ($\dimension =2$) or tetrahedra ($ \dimension =3$).
Let $(\phi_z)_{z\in\nodes}$ be the basis of Lagrangian shape functions for the space of globally continuous, elementwise affine-linear functions on $\mesh$, such that
$$
\phi_z (\hat{z}) = 
\begin{cases}
    1 \quad \text{ if } z = \hat{z},
    \\
    0 \quad \text{ else }
\end{cases} 
\quad \forall z,\hat{z} \in \nodes .
$$
Then, we call the mesh $\mesh$ \textit{weakly acute}, if
\begin{equation*}
    \omega_{z,\hat{z}} \coloneqq \int_{\Omega} \nabla \phi_z \cdot \nabla \phi_{\hat{z}} \diff x
    \leq 0
\end{equation*}
for all $z,\hat{z} \in \nodes$ with $z \neq \hat{z}$.
\end{defn}
Note that every \textit{non-obtuse} mesh is also \textit{weakly acute}, see \cite{bartels05}. The reverse, however, is not true.
\begin{rem}[On the restrictiveness of non-obtuse and weakly acute meshes]
\label{rem:weakly-acute-restrictive}
    In two spatial dimensions, a non-obtuse triangulation in the sense of Definition \ref{def:non-obtuse} exists \cite{baker-grosse-rafferty88} and can be attained using the Delaunay algorithm and refinement techniques, cf. \cite{shewchuk02}.
    In three spatial dimensions, this property is more restrictive since acute tetrahedra do in general not contain their circumspherecenters --- that can be used as a meshing criterion \cite{korotov-krizek01} --- and also tetrahedra containing their own circumspherecenters are in general not acute.
    Nevertheless, refinement techniques allow to obtain non-obtuse meshes of tetrahedra, cf. \cite{eppstein-sullivan04,shewchuk02,korotov-krizek01}.
\end{rem}
Now, returning to the further specification of the cell centers, the next definition follows in the spirit of \cite{eymard-etal00}.
\begin{defn}[Admissible Meshes]
    \label{def:admissible-mesh}
    We call a mesh $\mesh$ with cell centers $(x_T)_{T\in\mesh}$ \textit{admissible} if the facet barycenter is
    the average of the two neighbouring cell centers, i.e.~for $x_F \in \facets^i$
    \begin{equation}
    x_F 
    =
    \frac{x_T + x_{\tilde{T}}}{2}
    ,
    \label{eq:admissible-mesh}
    \end{equation}
    for $T , {\tilde{T}} \in \mathcal{M}(F)$.
\end{defn}
This notion of admissible meshes is crucial for the strong consistency of the reconstructed gradient, that is Lemma \ref{lem:appendix-strong-consistency-gradient}, which will be used for the convergence of discrete Laplacian and for the strong convergence of the initial condition.
\begin{rem}[On the restrictiveness of the admissibility condition]
\label{rem:admissible-mesh-restrictive}
The assumption, that Definition \ref{def:admissible-mesh} holds, is very restrictive. It can mostly be applied to structured triangulations, but cannot be expected to hold for unstructured domain decompositions.
\end{rem}
Finally, we collect a geometric property of our mesh.
\begin{prop}
Let $T$ be a polytope or polyhedron with facets $\facets^T$. Then, it holds that
$$
\sum_{F \in \facets^T} \abs{F} \normal_{T,F} =0 .
$$
\label{prop:appendix-divergence-theorem}
\end{prop}
\begin{proof}
Take an arbitrary but fixed vector $v\in \mathbb{R}^{\dimension}$. Applying the divergence theorem over $T$ leads to
$$
0 = \int_T \nabla \cdot v \diff x 
= v \cdot \int_{\partial T} \normal \diff s
= v \cdot 
\left(
    \sum_{F \in \facets^T} \abs{F} \normal_{T,F}
    \right)
    .
$$
\end{proof}
\subsection{Finite Element Spaces}
By $\mathbb{P}_k(T)$ we denote the set of polynomials of degree $k$ or less on the domain $T$.
For the velocity and pressure we consider \textit{P2-P1 Taylor-Hood} spaces given by
\begin{align*}
    M_h &\coloneqq \{  q \in   C (\overline{\Omega}) :  \int_{\Omega} q \diff x =0, \quad q \lvert_K \in \mathbb{P}_1 (K) \quad\forall K \in  \mesh  \}
    ,
    \\
    X_h &\coloneqq \{  v \in C (\overline{\Omega}) : v= 0 \text{ on }\partial \Omega, \quad v \lvert_K \in \mathbb{P}_2 (K) \quad\forall K \in  \mesh 
    \}
    ,
    \\ 
    V_h &\coloneqq \{  v\in [X_h]^3 : \quad
    (\nabla \cdot v, q ) = 0 \text{ for all } q\in M_h 
    \} \, ,
\end{align*}
as they are well known to fulfill the \textit{inf-sup} condition under our assumptions. 
In Section \ref{sec:discrete-system-cg} we consider globally continuous, piecewise linear functions without and with homogeneous Dirichlet boundary conditions defined by
\begin{align*}
    \cgone &\coloneqq \{  v \in C (\overline{\Omega}) :  \quad v \lvert_K \in \mathbb{P}_1 (K) \quad\forall K \in  \mesh \}
    ,\\
    \cgonezero &\coloneqq \{  v \in C (\overline{\Omega}) :  \quad v \lvert_K \in \mathbb{P}_1 (K) \quad\forall K \in  \mesh , \quad v= 0 \text{ on }\partial \Omega \} .
\end{align*}
Note that any function $f\in\cgone$ can be represented in terms of its nodal values and shape functions,
$$
f = \sum_{z \in \nodes} f(z) \phi_z \eqcolon \sum_{z \in \nodes} f_z \phi_z 
.
$$
In Section \ref{sec:discrete-system-dg} on the other hand, we use discontinuous, piecewise constant functions as the approximation space. We denote that space by
\begin{align*}
    \dgzero &\coloneqq \{  v \in L^2( \Omega ) :  \quad v \lvert_K \in \mathbb{P}_0 (K) \quad\forall K \in  \mesh \}
    .
\end{align*}
The base of $\dgzero$ consists of indicator functions on the elements that we denote by $\chi_T$ for $T\in\mesh$.
We introduce notions for the discontinuities of the spaces $\dg^0$.
The jump $\jump{\cdot}_F$ and average $\avg{\cdot}_F$ over an interior facet $F\in \facets^i$ are defined in a standard fashion (see \cite{dipietro-ern12}). For boundary facets $F \in \facets^D$ the average and jump will be evaluated as $\jump{v_h}_F = v_h\vert_F$, $\avg{v_h}_F = \frac{1}{2} v_h\vert_F$ for $v_h\in L^2 (F)$. In most situations, we will waive the subscript of the facet.
We will make regular use of the standard identities
\begin{align}
\begin{split}
\label{eq:jump-identities}
    \jump{a\cdot b}_F
    &= \avg{a}_F\cdot\jump{b}_F+\jump{a}_F\cdot \avg{b}_F
    \, ,
    \\
    \avg{a\cdot b}_F
    &=\avg{a}_F\cdot \avg{b}_F + \frac{1}{4} \jump{a}_F\cdot \jump{b}_F
    \, ,
\end{split}
\end{align}
that hold for all $a,b \in \dg^0$ and $F\in\facets$.
For any $v_h \in \dg^0$ we define the jump semi-norm with exponent $1\leq p \leq \infty$ by
\begin{equation*}
    \abs{v_h}_{J,p}^p
    \coloneqq
    \sum_{F \in \facets} \frac{1}{h_F^{p-1}} \norm{\jump{v_h}}^p_{L^p (F)}
\end{equation*}
where $h_F$ is the local length scale given by $h_F = \mathrm{diam} (F)$. For $p=2$, we often drop the according subscript and simply write $\abs{v_h}_{J}$.
We also use restrictions of the semin-norm to the interior and the boundary facets, i.e.
\begin{align*}
    \abs{v_h}_{J,p,i}^p
    \coloneqq &
    \sum_{F \in \facets^i} \frac{1}{h_F^{p-1}} \norm{\jump{v_h}}^p_{L^p (F)}
    ,
    \\
    \abs{v_h}_{J,p,D}^p
    \coloneqq &
    \sum_{F \in \facets^D} \frac{1}{h_F^{p-1}} \norm{v_h}^p_{L^p (F)}.
\end{align*}
\subsection{Interpolation}
\subsubsection*[into CG1]{into $\cgone$}
The standard global nodal interpolation operator for the Finite Element space $[\cgone]^{\dimension}$ is denoted by $\interpol^1: \mathcal{C}(\bar{\Omega}) \to [\cgone]^{\dimension}$. We will make frequent use of standard interpolation error estimates \cite[Thm. 4.4.20]{scott-brenner08}. That is for $f\in \Cont (\bar{\Omega})\cap W^{2,p} (\Omega) $, $p\in (3/2,\infty]$ and $k=0,1,2$, there is a constant $C>0$ such that
\begin{align}\label{eq:interpolation-error-estimate}
\sobnorm{(\identitymatrix -\interpol^1) f}{k}{p}
\leq
C h^{2-k}
\lebnorm{\nabla^2 f }{p}.
\end{align}
\subsubsection*[into DG0]{into $\dgzero$}
Interpolation into $\dgzero$ is simply attained by interpolation at the center $x_T$ of a cell $T\in\mesh$ in the sense of Definition \ref{def:admissible-mesh}. Therefore, we define the local interpolation of some continuous function $\phi\in \Cont (\bar{\Omega})$ by
$$
\mathcal{I}^0_{T} \phi = \phi(x_T)
.
$$
The global interpolation into $[\dgzero]^{\dimension}$ is accordingly defined by $\interpol^0: \mathcal{C}(\bar{\Omega}) \to [\dgzero]^{\dimension}$ such that
$$
\interpol^0 \phi = \sum_{T\in\mesh} \mathcal{I}^0_{T} (\phi) \chi_T.
$$
We will also frequently apply the inverse estimate (e.g. \cite[Thm. 4.5.11]{scott-brenner08}): Let $v$ be some Finite Element function with $v\vert_T \in \mathbb{P}_k(T) \cap W^{l,p}(T) \cap W^{m,q}(T)$ for $k\in\mathbb{N}$ and all $T \in \mesh$. Then, there exists a generic constant $C>0$ independent of $h$ such that
\begin{equation*}
\left(
\sum_{T \in \mesh}
\norm{v}_{W^{l,p}(T)}^p
\right)^{1/p}
\leq
Ch^{l-m + \min (0,\dimension /p- \dimension /q)}
\left(
\sum_{T \in \mesh}
\norm{v}_{W^{m,q}(T)}^q
\right)^{1/p}
\end{equation*}
holds for $1\leq p,q \leq \infty$, $0 \leq m \leq l$.
\subsubsection*{Mass Lumping}
By $(.,.)_2$ we denote the standard $L^2(\Omega)$ inner product, where as $(.,.)_h$ refers to a discrete inner product, often called \textit{mass-lumped} inner product, defined by
\begin{align}\label{def:lumped_norm_product}
    (y_1,y_2)_h \coloneqq 
    \int_{\Omega}
    \interpol^1 (y_1 \cdot y_2) \diff x
    =\sum_{z \in \mathcal{N}_h} y_1(z) \cdot y_2(z) \int_{\Omega} \phi_z \diff x
    , \quad
    \norm{y_1}_h^2 \coloneqq (y_1,y_1)_h.
\end{align}
for $y_1,y_2 \in C(\bar{\Omega};\mathbb{R}^{\dimension})$.
The induced norm is thereby equivalent to the $L^2(\Omega)$-norm on the discrete space with constants independent of the spatial discretization parameter $h$, i.e.~there exists a constant $c_L>0$ independent of $h$ such that
\begin{align}\label{eq:norm-equivalence-mass-lumping}
    \ltwonorm{f_h} \leq \hnorm{f_h}
    \leq c_L \ltwonorm{f_h}
\end{align}
for all $f_h \in [\cgone]^{\dimension}$.
On quasi-uniform meshes, this even holds for every $1\leq p \leq \infty$, cf. \cite{metzger-diss,raviart73}, i.e.~there exists a constant $c_{L,p}>0$ independent of $h$ such that
\begin{align}\label{eq:norm-equivalence-mass-lumping-p}
    \lebnorm{f_h}{p}^p \leq \int_{\Omega} \interpol^1 (\abs{f_h}^p ) \diff x
    \leq c_{L,p} \lebnorm{f_h}{p}^p
\end{align}
holds for all $f_h \in [\cgone]^{\dimension}$.
We further collect a preliminary lemma to be able to do local estimates when using the discrete inner product, see Lemma \ref{lem:a-priori-estimates-projection}.
\begin{lem}\label{lem:global-local-norm-vs-Lp-norm}
Let $\mesh$ be quasi-uniform and $f \in [\cgone]^{\dimension}$. For $1\leq p < \infty$ there exists a constant $C>0$ independent of $h$ such that the following estimate holds,
\begin{align*}
\lebnorm{f}{p}^p
\leq
\sum_{T \in \mesh} \lebnorm{\sum_{z \in \nodes^T}  f(z) \phi_z }{p}^p
\leq 
C
\lebnorm{f}{p}^p
 \, ,
\end{align*}
where $(\phi_z)_{z\in\nodes}$ are the Lagrangian shape functions of $\cg^1$, i.e.~$\phi_z \in \cg^1$ with
$$
\phi_z (\hat{z}) = 
\begin{cases}
    1 \quad \text{ if } z = \hat{z},
    \\
    0 \quad \text{ else }
\end{cases} 
\quad \forall z,\hat{z} \in \nodes .
$$
\end{lem}
\begin{proof}
    The first inequality follows simply from a decomposition over all cells $T\in\mesh$ and enlarging the integral domain.
    For the second inequality, one combines Equation \eqref{eq:norm-equivalence-mass-lumping-p} with the fact that the amount of neighbouring cells is bounded by a constant $C_{\rho}$ independent of $h$ (see \cite[Prop. 11.6]{ern-guermond21}) since the subdivision is assumed to be \textit{quasi-uniform} with \textit{quasi-uniformity} constant $\rho>0$ (cf. \cite[Def. 4.4.13]{scott-brenner08}).
\end{proof}
\subsection{Projections and Discrete Operators}
We define the operators $\projectcg$, $\projectdg$ and $\projectV$ via the following linear equation systems,
\begin{align*}
    (\projectcg f, c)_2 &= (f,c)_2 \quad \forall c\in [\cgone]^{\dimension}
    \\
    (\projectdg f, c)_2 &= (f,c)_2 \quad \forall c\in [\dgzero]^{\dimension}
    \\
    (\projectV f, a)_2 &= (f,a)_2 \quad \forall a\in V_h
    .
\end{align*}
Accordingly, we denote the operator as $\projectcgzero$ when $\cgone$ is replaced by $\cgonezero$ in the above equations.
For the projections $\projectcg, \projectcgzero, \projectdg$, standard stability and error estimates are well-known. We will make use of
\begin{align*}
\lebnorm{\projectcg f}{p} +
\lebnorm{\projectcgzero f}{p}+
\lebnorm{\projectdg f}{p} \leq C \lebnorm{f}{p} \, ,
\end{align*}
for $f\in L^2 (\Omega) \cap L^p(\Omega)$ with $p\in [1,\infty]$.
The projection onto the Taylor-Hood space $V_h$ yields the following estimates \cite[Lemma 4.3]{rannacher-heywood82},
\begin{align}\label{eq:approximation-properties-convergence-projection-projectV}
\lebnorm{\projectV \vv - \vv}{2}
+ h 
\lebnorm{\nabla \projectV \vv - \nabla \vv}{2}
+h^2 
\lebnorm{\nabla^2 \projectV \vv }{2}
\leq 
C h^2 
\lebnorm{ \Delta \vv}{2}
\end{align}
for all $\vv \in \velspace$
since the solution to the continuous incompressible Stokes problem with Dirichlet boundary conditions is regular enough under our assumpions on the domain \cite[Corollary 1.2.2]{mitrea-wright12}.
\subsection{Discontinuous Galerkin Preliminaries}
\label{sec:dg-preliminaries}
In this subsection, we collect some standard results following \cite{dipietro-ern12} and \cite[Sec. 5]{eymardetal10}.
Therewith, we will make use of a mix of Finite Element and Finite Volume theory for our approach of using piecewise constant functions.
Proofs that need an adaption to our case will be outlined.
We start by defining a discrete gradient based on the lifting operator in the Finite Element theory.
For a given mesh $\mesh$ we define the local lifting operator on a facet $F \in \facets$ by
$$
r^0_F : L^2(F;\mathbb{R}^{\dimension}) \to [\dg^0]^{\dimension \times \dimension}
$$
such that for all $\phi \in L^2(F;\mathbb{R}^{\dimension})$
\begin{align}
\begin{split}
    \int_{\Omega} r^0_F(\phi) : w_h \diff x
=
\begin{cases}
- \int_F \phi  \cdot \avg{ w_h } \normal_F \diff s
&, \forall w_h \in [\dg^0]^{\dimension \times \dimension}
, \text{ if } F \in \facets^i,
\\
\int_F (g_h - \phi) \cdot w_h  \normal_F \diff s
&, \forall w_h \in [\dg^0]^{\dimension \times \dimension}
, \text{ if } F \in \facets^D \, ,
\end{cases}
\end{split}
\label{eq:def-discrete-lifting-local}
\end{align}
holds for a given discretization $g_h \in L^2 (\partial \Omega; \mathbb{R}^{\dimension})$ of our Dirichlet boundary condition.
The inclusion of the boundary condition $g_h$ is a slight deviation from the standard case, cf. \cite{dipietro-ern12}.
The existence of such a lifting operator follows from the Lax--Milgram theorem.
Accordingly, the global lifting operator for boundary jumps of a vector valued function $v \in [\dg^0]^{\dimension}$ is defined by
\begin{equation}
\adliftingl{v}{g_h}{0}  \coloneqq \sum_{F \in \facets} r^0_F (\jump{v}).
\label{eq:def-discrete-lifting-reconstructed-gradient}
\end{equation}
While lifting operators exist also for functions of higher order, in the case of a piecewise constant function space, evaluating $\int_{T}\adliftingl{v}{g_h}{0} :A \diff x$ for $T\in \mesh$ and an arbitrary constant matrix $A\in \mathbb{R}^{\dimension \times \dimension}$ allows to derive an explicit representation of $\adlifting{v}{g_h}$, here given by (cf. \cite[Eq. 4.43]{dipietro-ern12})
\begin{align}
    \nonumber
    \adlifting{v}{g_h}
    \vert_T
    & \coloneqq
    \sum_{F \in\facets^T} \frac{\abs{F}}{\abs{T}} (v_F -v_T)\otimes \normal_{T,F},
    \\
    \label{eq:def-v_F}
    \text{ with }
    v_F
    &=
    \begin{cases}
    \avg{v}_F, &\text{ if } F\in\facets^i
    ,
    \\
    g_h , &\text{ if } F\in\facets^D
    .
    \end{cases}
\end{align}
Due to the construction of our admissible mesh (see Definition \ref{def:admissible-mesh}), the above formula can also be considered a reconstruction of the gradient in the spirit of Finite Volume theory, cf. \cite{eymardetal10,eymard-etal00}.
We hereby intersect the Finite Element and Finite Volume theory by choosing a mesh and cell centers such that the average operator on a facet --- often used in the Finite Element theory --- and the barycentric interpolator for the facet barycenter --- often used in the Finite Volume theory --- coincide.
\\
For the following Lemmata, assume that $\dbc \in [\dgzero]^{\dimension}$ is an approximation of our initial condition $\fat{d}_0$ on the boundary fulfilling
$$
\norm{\dbc - \fat{d}_0 }_{L^{\infty} (\partial \Omega)}
\lesssim h \norm{\fat{d}_0 }_{C^1 (\partial \Omega)}.
$$
\begin{lem}[Strong consistency of discrete gradient, cf. {\cite[Lem. 4.4]{eymardetal10}}]
    \label{lem:appendix-strong-consistency-gradient}
    Let $\phi, \fat{d}_0 \in C^2(\bar{\Omega})$ with $\phi \vert_{\partial \Omega} = \fat{d}_0 \vert_{\partial \Omega}$. 
    Let $(\mesh)_h$ be a sequence of admissible meshes in the sense of Definition \ref{def:admissible-mesh}.
    Then, it holds
    \begin{align*}
        \ltwonorm{
            \adlifting{\interpol^0 \phi}{\dbc}
            -
            \nabla \phi
        }
        \to 0
    \end{align*}
    as $h \to 0$.
\end{lem}
The proof follows as in \cite[Lem. 4.4]{eymardetal10}.
\begin{lem}[Weak consistency of the discrete gradient]
\label{lem:weak-consistency-gradient}
Let $\fat{d}_0 \in C^1(\bar{\Omega};\mathbb{R}^{\dimension})$ and $v_h: [0,T] \to [\dgzero ]^{\dimension}$ be a sequence of functions indexed by $h$ that fulfills
\begin{align*}
    v_h  \weakto \fat{v}  & \in L^p (0,T;L^p (\Omega; \mathbb{R}^{\dimension}))
    ,
    \\
    \sup_h \abs{v_h}_{J,1} & \in L^1(0,T)
    ,
    \\
    \adlifting{v_h}{\dbc} \weakto \fat{W} & \in L^p (0,T;L^{p}(\Omega, \mathbb{R}^{\dimension \times \dimension}))
     \, ,
\end{align*}
for $1< p < \infty$.
Then, we have
$$
\fat W = \nabla \fat v \text{ a.e. in } (0,T) \times \Omega, \quad \mathrm{tr} (\fat{v}) = \fat{d}_0 \vert_{\partial \Omega} \text{ a.e. in }(0,T) \times \partial \Omega
.
$$
\end{lem}
\begin{proof}
    We follow the proof of \cite[Thm. 5.7]{dipietro-ern12}.
    We consider a smooth and compactly supported function in time and space $\phi \in C_c^\infty ([0,T] \times \Omega; \mathbb{R}^{\dimension \times \dimension})$.
    First, we note that an application of the divergence theorem yields
    \begin{multline}
    \label{eq:weak-consistency-discrete-lhs}
        \int_{\Omega}
        v_h ( \nabla \cdot  \phi ) \diff x
    =
        \sum_{T\in \mesh}
        \sum_{F \in \facets^T}
        \int_F v_h \vert_T \cdot \phi \normal_{T,F}\diff s
    \\
    =
    \sum_{F \in \facets^i}
    \int_F \jump{v_h} \cdot {\phi} \normal \diff s
    +
    \sum_{F \in \facets^D}
    \int_F {v_h} \cdot \phi \normal \diff s
    ,
    \end{multline}
    where the last term vanishes due to the zero trace of $\phi$.
    On the other hand, the definition of the discrete gradient yields
    \begin{multline}
    \label{eq:weak-consistency-discrete-rhs}
        \int_{\Omega} \adlifting{v_h}{\dbc} : \interpol^0 \phi \diff x
    =
        -\sum_{F\in \facets^i}
        \int_F
        \jump{v_h} \cdot \avg{ \interpol^0 \phi}  \normal \diff s
        \\
        +
        \sum_{F\in \facets^D}
        \int_F
        (\dbc - v_h)  \cdot (\interpol^0 \phi) \normal \diff s.
    \end{multline}
    Adding Equations \eqref{eq:weak-consistency-discrete-lhs} and \eqref{eq:weak-consistency-discrete-rhs} and integrating in time leaves us with 
    \begin{multline*}
        \int_0^T \int_{\Omega}
        v_h ( \nabla \cdot  \phi ) 
        +
        \adlifting{v_h}{\dbc} : \interpol^0 \phi 
        \diff x
        \diff \tau
        \\=
        \int_0^T
        \sum_{F \in \facets^i}
        \int_F \jump{v_h} \cdot (\phi - \avg{\interpol^0 \phi}) \normal \diff s
        +
        \sum_{F\in \facets^D}
        \int_F
        (\dbc - v_h)  \cdot (\interpol^0 \phi) \normal \diff s
        \diff \tau
        .
    \end{multline*}
    For the first term on the right-hand side, we observe that it vanishes as $h \to 0$, since
    \begin{equation*}
        \sum_{F \in \facets^i}
        \int_F \jump{v_h} \cdot (\phi - \avg{\interpol^0 \phi}) \normal \diff s
        \lesssim
        h \norm{\phi }_{C^1 ( \bar{\Omega})}
        \sum_{F\in \facets^D}
            \int_F  \abs{v_h} \diff s
        =
        h \norm{\phi }_{C^1 ( \bar{\Omega})}
        \abs{v_h}_{J,1}
        .
    \end{equation*}   
    The second term vanishes as $h \to 0$ due to the zero trace of $\phi$, i.e.
    \begin{multline*}
    \sum_{F\in \facets^D}
        \int_F
        (\dbc - v_h)  \cdot (\interpol^0 \phi - 0) \normal_F \diff s
    \lesssim
    h \norm{\phi }_{C^1 (\partial \Omega)}
    \sum_{F\in \facets^D}
        \int_F \abs{\dbc} + \abs{v_h} \diff s
    \\
    \lesssim
    h
    \left (
    \int_{\partial \Omega} \abs{\dbc} \diff s+ \abs{v_h}_{J,D,1}
    \right )
    .
    \end{multline*}
    Now, we can infer that $\fat{v}\in H^1(\Omega)$ with the weak gradient $\nabla \fat{v}=\fat{W}$. Accordingly, the Sobolev trace $\tr(\fat{v})$ exists.
    Adding Equation \eqref{eq:weak-consistency-discrete-lhs} and \eqref{eq:weak-consistency-discrete-rhs} again, this time with a test function not vanishing on the boundary, and taking the limit $h\to 0$ yields
    \begin{multline*}
        \int_{\partial \Omega} \tr (\fat{v}) \cdot \phi \normal \diff s
        =
        \int_{\Omega} \nabla v : \phi \diff x
        +
        \int_{\Omega} v (\nabla \cdot \phi) \diff x
    \\
        =
        \lim_{h \to 0}
        \int_{\Omega}
        v_h ( \nabla \cdot  \phi ) \diff x
        +
        \int_{\Omega} \adlifting{v_h}{\dbc} : \interpol^0 \phi \diff x
    \\
        =
        \lim_{h \to 0}
        \sum_{F\in \facets^D}
        \int_F
        \dbc  \cdot \phi \normal \diff s
        =
        \int_{\partial \Omega} \fat{d}_0 \cdot \phi \normal \diff s
        ,
    \end{multline*}
    for all $\phi \in C^\infty ([0,T] \times \Omega; \mathbb{R}^{\dimension \times \dimension})$. This proves the result.
\end{proof}
\begin{lem}[An Aubin--Lions--Simon lemma for $\dgzero$]
\label{lem:dg-aubin-lions-simon}
    Let a sequence of functions $v_h: [0,T] \to [\dgzero ]^{\dimension}$ indexed by $h$ fulfill
    \begin{align*}
        \norm{ v_h }_{L^p (0,T; \L{q})}
        +
        \norm{ \abs{v_h}_{J,1} }_{L^p (0,T)}
        +
        \norm{ \partial_t v_h }_{L^m (0,T; \L{s})}
        \leq M
    \end{align*}
    for $1 \leq p,m \leq \infty$, $1\leq s < q \leq \infty$ and some $M>0$ that does not depend on $h$. If $p<\infty$, there exists a subsequence that we do not relabel such that
    $$
    v_h \to v \in L^p (0,T;\L{r})
    $$
    for $s\leq r<q$.
    If $p= \infty$ and $m>1$, then there exists a subsequence that we do not relabel such that
    $$
    v_h \to v \in C^0([0,T];\L{r})
    $$
    for $s\leq r<q$.
\end{lem}
\begin{proof}
The $\mathrm{BV}(\mathbb{R}^{\dimension})$-norm of functions in $\dgzero$ can be bounded from above by $\abs{.}_{J,1}$, see \cite[Lemma 5.2]{dipietro-ern12}. 
Interpolating $\BV (\Omega) \cap L^q$ between $\BV (\Omega) \cap L^1$ and $L^r$ (cf. \cite[Thm. 5.6]{dipietro-ern12} and \cite[Lemma 5.4]{eymardetal10}), allows to derive the embedding
$$
\BV (\Omega) \cap L^{q}(\Omega)
\stackrel{c}{\hookrightarrow}
L^r(\Omega) 
\hookrightarrow
L^s (\Omega)
$$
for $s\leq r < q$. 
On that embedding triple, we can apply a general version of the Aubin--Lions--Simon lemma, see e.g. \cite{boyer-fabrie13,chen-etal13}, which yields the result.
\end{proof}
\section{A Continuous Finite Element Scheme}
\label{sec:discrete-system-cg}
We can now introduce our Finite Element Approximation based on a globally continuous, piecewise linear approximation of the director field.
Since we work with constant-in-time Dirichlet boundary conditions, we will implicitly use the following decomposition for the director field
$$
 d =  d_\circ +  \interpol^1 \fat d_0 \, ,
$$
such that $d, \interpol^1 \fat d_0 \in [\cgone]^{\dimension}$ and $d_\circ \in [\cgonezero]^{\dimension}$.
\begin{algo}
\label{algo:cg}
Let $(v^{0}, d^{0})= (\projectV \fat v_0, \interpol^1 \fat d_0)$.
For $1 \leq j\leq J$, $(v^{j-1}, d^{j-1}_\circ )\in \mixedfemspacecg$, we want to find $(v^{j}, \tilde{d}^{j}_\circ )\in \mixedfemspacecg$, such that
\begin{subequations}\label{eqs:projection-scheme}
        \begin{multline}
        \label{eq:projection_scheme_a}
            (\discreteDiff_t v^j,a)_2 + \mu (\nabla v^j, \nabla a)_2
           +((v^{j-1}  \cdot \nabla) v^j,a)_2 + \frac{1}{2} ([\nabla \cdot v^{j-1}]v^j,a)_2 
           \\
            -
            \left (
                [\nabla d^{j-1}]^T 
                d^{j-1}\times 
                \interpol^1 ( d^{j-1} \times \Deltah \tilde{d}^{j} )
                ,a
            \right )_2  
            = 0 \,,
        \end{multline}
        \begin{multline}
        \label{eq:projection_scheme_c}
            \frac{1}{k}(\tilde{d}^{j} - d^{j-1},c)_h
             + 
            \left (
                d^{j-1}\times [{\nabla d^{j-1}}v^j], \interpol^1 ( d^{j-1} \times c ) 
            \right )_2
            \\
            {-}  
            \left (
                 d^{j-1} \times \Deltah \tilde{d}^{j} , d^{j-1} \times c 
            \right )_h
            = 0,
        \end{multline}
    \end{subequations}
for all $(a,c)\in \mixedfemspacecg$.
\\
Then, we apply the discrete normalization operator
\begin{align}\label{eq:projection-step-cg}
d^j \coloneq \interpol^1   \frac{\tilde{d}^j}{\vert{\tilde{d}^j}\vert}
 .
\end{align}
\end{algo}
Since the second derivative of a linear function vanishes, 
we introduce the (mass-lumped) discrete Laplacian $\Deltah: [H^1_0(\Omega)]^{\dimension} \to [\cgonezero]^{\dimension}$ which is defined as solution to the equation system
\begin{align*}
    (-\Deltah f, b)_h
    = 
    (\nabla f, \nabla b)
\end{align*}
for all $b \in [\cgonezero]^{\dimension}$.
Choosing $f \in [H^1_0(\Omega)]^{\dimension}$, testing with $b= f$ and $b= \Deltah f$ respectively as well as applying the inverse estimate, yields the estimate
\begin{equation}
    \label{eq:discrete-norm-equivalence-laplacian-gradient}
    c h
    \lebnorm{\Deltah f}{2}
    \leq
    \lebnorm{\nabla f}{2}
    \leq
    C
    \lebnorm{\Deltah f}{2}
\end{equation}
for generic constants $c,C>0$.
Since we only defined the discrete Laplacian for $[H^1_0 (\Omega)]^{\dimension}$, the Laplacian of the approximate director field $ \tilde{d}^j \in [\cgone]^{\dimension}$ for $1 \leq j\leq J$ has to be implicitly understood as
$$
\Deltah \tilde{d}^j  = \Deltah \tilde{d}^j_{\circ} + \projectcg \Delta \fat{d}_0.
$$
\begin{rem}
In Algorithm \ref{algo:cg}, we adapted the convection term for the director equation (cf. Equations \eqref{system_c} and \eqref{eq:projection_scheme_c}) as in \cite{lasarzik-reiter23}. This can be justified if the unit-norm constraint $\abs{ \fat d}=1$ is fulfilled almost everywhere. Then, it formally holds
$$ 
- \fat d \times ( \fat d\times [(\fat v \cdot  \nabla) \fat d])
=
(\identitymatrix - \fat d \otimes \fat d) \nabla \fat d \fat v
=
(\fat v \cdot  \nabla) \fat d - (\fat d \otimes \frac{1}{2} \nabla \abs{\fat{d}}^2) \fat v
=
(\fat v \cdot  \nabla) \fat d
.
$$
\end{rem}
\begin{rem}
    Note that Algorithm \ref{algo:cg} can also be understood as a linearization of the fully implicit discretization in \cite{lasarzik-reiter23}. Due to the linearization, the implicit norm conservation is lost. The projection step corrects that effect.
\end{rem}
Therewith, we can state our main result that will be proven in Section \ref{sec:convergence-to-envar-sols-cg}.
\begin{thm}[Convergence of a subsequence to energy-variational solutions]\label{thm:main-result-subconvergence-cg}
    Let $\Omega \subset \mathbb{R}^{\dimension}$ for $\dimension = 2,3$ be a bounded convex polyhedral domain.
    Let $\left(\fat v_0,\fat d_0\right)
     \in \V \times  H^{2}(\Omega;\mathbb{R}^{\dimension}) $ with $\vert \fat d_0  \vert = 1$ a.e.~in $\Omega$.
     Let $(k_m, h_m)_m $, $m\in \mathbb{N}$ be a decreasing sequence of our positive temporal and spatial discretization parameters converging to zero as $m \to \infty$.
    We assume this sequence fulfills:
    \begin{enumerate}
        \item There exists a \textit{weakly acute} subdivision $\mathcal{T}_{h_m}$ of $\Omega$ (in the sense of Definition \ref{def:weakly-acute}) for each $m \in \mathbb{N}$,
        \item $ k_m \in o (h_m^{1+\dimension /6})$.
    \end{enumerate}
    Then, there exists a subsequence that we do not relabel $(k_m, h_m)_m$ such that a solution to Algorithm \ref{algo:cg} exists for every $(k_m, h_m)$ and $1\leq j \leq J(k_m)$ and its linear temporal interpolate converges to an \textit{energy-variational solution} in the sense of Definition~\ref{def:envar} as $(k_m, h_m)\to 0$. 
\end{thm}
\subsection[Well-Posedness]{Discrete Well-Posedness}
We start this section by showing that the projection step, Equation \eqref{eq:projection-step-cg}, is well-posed and energy-decreasing.
Testing Equation \eqref{eq:projection_scheme_c} with $d^{j-1}_z \phi_z$ yields a nodal orthogonality relation (cf. \cite[p. 22]{lasarzik-reiter23}) subsummed by the following corollary.
\begin{cor}[Orthogonality]\label{cor:orthogonality-cg}
Let $(v^j,  \tilde{d}^{j}_\circ )$ be a solution to Algorithm \ref{algo:cg}. Then, the following nodal orthogonality relation holds:
\begin{align*}
    (\tilde{d}^j_z - d^{j-1}_z) \cdot d^{j-1}_z & = 0 
    & \forall z \in \mathcal{N}_h 
    .
\end{align*}
\end{cor}
This corollary guarantees that $\abs{\dir_h^{j-1} (z)}=1$ implies $\abs{\tilde{\dir}_h^j (z)} \geq 1$ for all $z\in \nodes$, which is a requirement for an energy-decreasing projection as the next lemma will show.
The following lemma stems from \cite[Lemma 2.2]{bartels2015}. For the reader's convenience, we also reiterate the proof.
\begin{lem}\label{lem:energy-decreasing-projection-cg}
    Let $\mesh$ be \textit{weakly acute} in the sense of Definition \ref{def:weakly-acute} and $d^h \in [\cgone]^{\dimension}$ such that $\abs{d^h (z)} \geq 1$ for all $z \in \nodes$.
    Then, the nodal normalization of $d^h$ is energy-decreasing with respect to the Dirichlet energy, i.e.
    \begin{equation}
        \ltwonorm{\nabla \interpol^1 \frac{d^h}{\abs{d^h}}} 
    \leq  
        \ltwonorm{\nabla d^h}.
        \label{eq:energy-decreasing-projection}
    \end{equation}
\end{lem}
\begin{proof}
Let us denote the abbreviation
$$
\omega_{z,\hat{z}} \coloneqq \int_{\Omega} \nabla \phi_z \cdot \nabla \phi_{\hat{z}} \diff x
$$
for the $L^2 (\Omega)$ scalar product of two shape functions.
Then, the identity $a \cdot b = \frac{1}{2} \abs{a}^2 + \frac{1}{2} \abs{b}^2 - \frac{1}{2}\abs{a-b}^2$ allows us to reformulate the left-hand side of Inequality \eqref{eq:energy-decreasing-projection} to
\begin{multline}\label{eq:reformulation-energy-decreasing-projection}
    \ltwonorm{\nabla \interpol^1 \frac{d^h}{\abs{d^h}}}^2
    =
    \sum_{z,\hat{z} \in \nodes} \omega_{z,\hat{z}} 
        \frac{d^h_z}{\abs{d^h_z}} \cdot \frac{d^h_{\hat{z}}}{\abs{d^h_{\hat{z}}}} 
    \\=
    -\frac{1}{2} \sum_{z,\hat{z} \in \nodes} \omega_{z,\hat{z}} \abs{\frac{d^h_z}{\abs{d^h_z}} - \frac{d^h_{\hat{z}}}{\abs{d^h_{\hat{z}}}} }^2
    + \frac{1}{2} \sum_{z,\hat{z} \in \nodes} \omega_{z,\hat{z}} 
    \left ( 
        \abs{\frac{d^h_{\hat{z}}}{\abs{d^h_{\hat{z}}}} }^2 + \abs{\frac{d^h_{z}}{\abs{d^h_{z}}} }^2 
    \right )
    .
\end{multline}
The last sum hereby vanishes since the shape functions are a partition of one, i.e.
$$
\sum_{z \in \nodes} \omega_{z,\hat{z}}
=
\int_{\Omega} (\nabla (\sum_{z \in \nodes} \phi_z ), \nabla \phi_{\hat{z}}) \diff x
= \int_{\Omega} (\nabla 1, \nabla \phi_{\hat{z}}) \diff x
=0.
$$
Due to the weakly acute mesh, all remaining summands in Equation \eqref{eq:reformulation-energy-decreasing-projection} have positive prefactors. It remains to estimate the summands from above.
It is well known that projections from outside a convex set onto its boundary do not increase the Euclidean distance between two points, such that 
$$
\abs{\left[ \interpol^1   \frac{d^h}{\vert d^h \vert} \right]_z - \left[ \interpol^1   \frac{d^h}{\vert d^h \vert} \right]_{\hat{z}} }^2 
=
\abs{\frac{d^h_z}{\abs{d^h_z}} - \frac{d^h_{\hat{z}}}{\abs{d^h_{\hat{z}}}} }^2 
\leq
\abs{d^h_z - d^h_{\hat{z}} }^2
$$
holds. This yields the result.
\end{proof}
Next, we show the existence of discrete solutions.
\begin{lem}[Unconditional existence]
    \label{lem:unconditional-existence-projection}
Let $k,h>0$ and $j \in \mathbb{N}$, such that $1\leq j \leq J = \lfloor T/k \rfloor$ and $u^{j-1}=(v^{j-1},d^{j-1}_\circ )\in \mixedfemspacecg$. Then, there exists a unique solution $u^j=(v^j,\tilde{d}^j_\circ )\in \mixedfemspacecg$ solving Equation System \eqref{eqs:projection-scheme}.
\end{lem}
\begin{proof}
The proof follows by applying the Lax--Milgram theorem. 
We therefore decompose Equations \eqref{eqs:projection-scheme} into a bilinear map $\mathcal{B}: (\mixedfemspacecg)^2 \to \mathbb{R}$ defined by
\begin{multline*}
\mathcal{B}
    \left (
    (v^j, - \Deltah \tilde{d}^j_{\circ})^T), (a,c)^T
    \right )
\\
    \coloneqq
        \frac{1}{k}( v^j,a)_2 + \mu (\nabla v^j, \nabla a)_2
        +
        ((v^{j-1}  \cdot \nabla) v^j,a)_2 
        +
        \frac{1}{2} ([\nabla \cdot v^{j-1}]v^j,a)_2 
    \\
        -
        ([{\nabla d^{j-1}}]^T [d^{j-1}\times \interpol^1 (d^{j-1} \times \Deltah \tilde{d}^{j}_{\circ} )],a)_2  
    \\
        +
        \frac{1}{k}(\tilde{d}^{j}_{\circ},c)_h
        +
        (d^{j-1}\times [{\nabla d^{j-1}}v^j], \interpol^1 [d^{j-1} \times c])_2
    \\
        -
        (d^{j-1} \times \Deltah \tilde{d}^{j}_{\circ}, d^{j-1} \times c)_h
\end{multline*}
and a right hand-side $f: ( \mixedfemspacecg )^* \to \mathbb{R}$,
\begin{multline*}
    \langle f , (a,c)^T \rangle
= 
    \frac{1}{k} 
    \left [
    (v^{j-1}, a) + (d^{j-1} - \interpol^1 \fat d_0,c)
    \right ]
\\
    +
    ([{\nabla d^{j-1}}]^T [d^{j-1}\times \interpol^1(d^{j-1} \times \projectcg \Delta \fat d_0 )],a)_2
    +
    (d^{j-1} \times \projectcg \Delta \fat d_0 , d^{j-1} \times c)_h
\end{multline*}
such that a solution $u^* \in \mixedfemspacecg$ to Equations \eqref{eqs:projection-scheme} solves
$$
\mathcal{B} (u^*, \bar{u}) = \langle f, \bar{u} \rangle
$$
for all $\bar{u} \in \mixedfemspacecg$ and vice versa.
Thereby, the term $\tilde{d}^j_{\circ} = \mathcal{G} (\Deltah \tilde{d}^j_{\circ})$ has to be understood in terms of a Green's function $\mathcal{G} : [\cgonezero]^{\dimension} \to [\cgonezero]^{\dimension}$ of the discrete Laplacian with homogeneous Dirichlet boundary conditions, i.e.~the solution to the equation system
$$
(\nabla \mathcal{G} (w), \nabla b) = ( w,b)_h \text{ for all } b \in [\cgonezero]^{\dimension}.
$$
The well-posedness of the map $\mathcal{G}$ follows again from applying the Lax-Milgram theorem.
For the coercivity of the map $\mathcal{B}$, we first observe that
\begin{align}\label{eq:proof-discrete-ED-mechanism-projection-free}
\begin{split}
    \mathcal{B}\left (
    (v^j, - \Deltah \tilde{d}^j_{\circ})^T), (v^j, - \Deltah \tilde{d}^j_{\circ})^T 
    \right )
    =&
    \frac{1}{k} \left (
    \ltwonorm{v^j}^2  + \ltwonorm{\nabla \tilde{d}^j_{\circ}}^2
    \right )
\\&
+   \mu \lebnorm{\nabla v^{j}}{2}^2 
    +   \hnorm{ d^{j-1} \times \Deltah \tilde{d}^{j}_{\circ} }^2 
    .
\end{split}
\end{align}
Note that the discrete convection term vanished due to their skew-symmetry and that two of the remaining terms cancelled each other out since
$$
-([{\nabla d^{j-1}}]^T [d^{j-1}\times \interpol^1 (d^{j-1} \times \Deltah \tilde{d}^{j}_{\circ} )], v^j)_2 
=
(\interpol^1 [d^{j-1} \times \Deltah \tilde{d}^{j}_{\circ}] , 
d^{j-1}\times[{\nabla d^{j-1}}] v^j)_2 
.
$$
Applying Inequality \eqref{eq:discrete-norm-equivalence-laplacian-gradient} onto the right-hand side of Equation \eqref{eq:proof-discrete-ED-mechanism-projection-free} yields the coercivity,
$$
\mathcal{B}\left (
    (v^j, - \Deltah \tilde{d}^j_{\circ})^T), (v^j, - \Deltah \tilde{d}^j_{\circ})^T 
    \right )
\geq
    \frac{1}{k} \left (
    \ltwonorm{v^j}^2  + c^2 h^2 \norm{\Deltah \tilde{d}^j_{\circ}}_h^2
    \right ) .
$$
The boundedness follows analogously, again using the inverse and Poincaré inequality.
\end{proof}
\subsection{\textit{A priori} Estimates}
\begin{lem}[\textit{A priori} estimates]\label{lem:a-priori-estimates-projection}
Let the assumptions of Theorem \ref{thm:main-result-subconvergence-cg} be fulfilled and $u^j=(v^j,\tilde{d}^j_\circ)\in \mixedfemspacecg$ be a solution of Algorithm \ref{algo:cg} for all $1\leq j\leq n \leq J$. Then, the following discrete Energy-Dissipation inequality holds
\begin{align}
\begin{split}\label{eq:discrete-energy-inequality-projection}
\frac{1}{2}\ltwonorm{v^n}^2  & + \frac{1}{2}\ltwonorm{\nabla \tilde{d}^n}^2
\\&
+ k \sum_{j=1}^n \mu \lebnorm{\nabla v^{j}}{2}^2 
    + k \sum_{j=1}^n \ltwonorm{ \interpol^1 [d^{j-1} \times \Deltah \tilde{d}^{j}] }^2 
    \\& 
    + \frac{1}{2} \sum_{j=1}^n \ltwonorm{\nabla [\tilde{d}^j -d^{j-1}]}^2
    + \frac{k^2}{2} \sum_{j=1}^n \ltwonorm{ \discreteDiff_t {v}^j}^2
    \\& \quad\quad\quad\quad\quad\quad\quad\quad\quad\quad\quad\quad\quad\quad
    \leq
    \frac{1}{2}  \ltwonorm{v^0}^2
    + \frac{1}{2}\ltwonorm{\nabla d^0}^2 
    \, ,
\end{split}
\end{align}
as well as
\begin{align}\label{eq:non-projected-dirichlet-energy-bound}
    \ltwonorm{\nabla {d}^n}^2
    \leq
    \ltwonorm{\nabla \tilde{d}^n}^2 
    .
\end{align}
The projected solution fulfills the unit-sphere constraint nodally, i.e.
\begin{equation}\label{eq:nodal-unit-norm-constraint-for-projected-d}
    \abs{d^n(z)}=1
\end{equation}
for all $z\in\nodes$.
Additionally, for some generic constant $C>0$, we obtain that
\begin{align}
\begin{split}
k \sum_{l=1}^n \norm{k^{-1} (\tilde{d}^j - d^{j-1})}_{L^{3/2}(\Omega)}^2
&\leq 
C 
,
\\
k \sum_{l=1}^n \norm{k^{-1} ({d}^j - d^{j-1})}_{L^{3/2}(\Omega)}^2
&\leq 
C 
,
\\
k \sum_{l=1}^n \norm{\discreteDiff_t v^j}_{(\velspace )^*}^2
&\leq 
C 
,
\label{bounds_temp_variation_inequ}
\end{split}
\end{align}
where the generic constant on the right-hand side may depend on $\mathcal{E}(v^0,d^0)$.
\end{lem}
Note that, under the assumptions of Theorem \ref{thm:main-result-subconvergence-cg}, the right-hand-side of Inequality \eqref{eq:discrete-energy-inequality-projection} is bounded independently of the spatial discretization parameter $h$ due to the $L^2(\Omega)$-stability of $\projectV$ and Inequality \eqref{eq:interpolation-error-estimate}.
\begin{proof}
The nodal unit-norm constraint \eqref{eq:nodal-unit-norm-constraint-for-projected-d} is fulfilled trivially by the projection step \eqref{eq:projection-step-cg}.
For the rest, we follow \cite[Proposition 3.2]{lasarzik-reiter23}. The discrete Energy-Dissipation mechanism, Inequality \eqref{eq:discrete-energy-inequality-projection}, follows by reiterating the computations prior to Equation \eqref{eq:proof-discrete-ED-mechanism-projection-free}. Inequality \eqref{eq:non-projected-dirichlet-energy-bound} follows by applying the energy-decreasing projection, that is Lemma \ref{lem:energy-decreasing-projection-cg}.
Using the $L^p$ duality we can estimate the approximate time derivative in a standard fashion by
\begin{multline*}
    \lebnorm{k^{-1} (\tilde{d}^j - d^{j-1})}{3/2}
    = 
    \sup_{\phi \in \L{3}}
    \abs{
    \frac{(k^{-1} (\tilde{d}^j - d^{j-1}), \projectcgzero \phi)_2}{\lebnorm{\phi}{3}}
    }
\\
\leq 
    \sup_{\phi \in \L{3}}
    \abs{
    \frac{(k^{-1} (\tilde{d}^j - d^{j-1}), \projectcgzero \phi)_h}{\lebnorm{\phi}{3}}
    }
\\
    +
    \sup_{\phi \in \L{3}}
    \abs{
    \frac{\int_\Omega (\identitymatrix -\interpol) k^{-1} (\tilde{d}^j - d^{j-1}) \cdot \projectcgzero \phi \diff x}{\lebnorm{\phi}{3}}
    }
    .
\end{multline*}
The first summand can easily be estimated by evaluating the discrete evolution Equation \eqref{eq:projection_scheme_c}.
The second summand comprises the error introduced by the mass-lumping and can be estimated as
\begin{multline}
    \label{eq:mass-lumping-error}
\int_\Omega (\identitymatrix -\interpol) ([\tilde{d}^j- d^{j-1}]\cdot \projectcgzero \phi) \diff x
\\ 
\lesssim
\sum_{K\in \mesh}
\norm{\tilde{d}^j- d^{j-1}}_{L^2(K)} \norm{\projectcgzero\phi}_{L^3(K)}
\norm{ 1 }_{L^6(K)}
\\
\lesssim
h^{\dimension /6}
\sum_{K\in \mesh}
\lebnorm{\sum_{z\in \nodes^{K}} [\tilde{d}^j_z- d^{j-1}_z]\Phi_z}{2} \lebnorm{\sum_{z\in \nodes^{K}} [\projectcgzero \phi]_z \Phi_z}{3}
.
\end{multline}
We now need to derive a local estimate. First, we test Equation \eqref{eq:projection_scheme_c} with $c= \sum_{z\in \nodes^{K}} (\tilde{d}^j_z - d^{j-1}_z)\phi_z $ to observe
\begin{multline*}
\norm{\sum_{z\in \nodes^{K}} (\tilde{d}^j_z - d^{j-1}_z)\phi_z}_{h}^2
= \left ( \tilde{d}^j - d^{j-1},  \sum_{z\in \nodes^{K}} (\tilde{d}^j_z - d^{j-1}_z)\phi_z \right )_h
\\
\leq
k
\norm{\nabla d^{j-1} v^j + \interpol^1 [ d^{j-1} \times \Deltah \tilde{d}^j ]}_{L^{3/2}(K \cup \mathcal{N}(K))} 
\lebnorm{  \sum_{z\in \nodes^{K}} (\tilde{d}^j_z - d^{j-1}_z)\phi_z}{3}
 \, ,
\end{multline*}
where $\mathcal{N}(K)$ are all neighbouring cells of $K\in\mesh$.
Applying the inverse estimate on the last summand yields
\begin{equation*}\label{eq:local-estimate-time-derivative}
h^{\dimension /6} \lebnorm{  \sum_{z\in \nodes^{K}} (\tilde{d}^j_z - d^{j-1}_z)\phi_z}{2} 
\lesssim k 
\norm{\nabla d^{j-1} v^j + \interpol^1 [ d^{j-1} \times \Deltah \tilde{d}^j ]}_{L^{3/2}(K \cup \mathcal{N}(K))} .
\end{equation*}
Applying first the above estimate, and then Lemma \ref{lem:global-local-norm-vs-Lp-norm} on all factors of the right hand side of Inequality \eqref{eq:mass-lumping-error}, we can find that
\begin{multline*}
\int_\Omega (\identitymatrix -\interpol) ([\tilde{d}^j- d^{j-1}]\cdot \projectcgzero\phi) \diff x
\\
\lesssim
\lebnorm{\nabla d^{j-1} v^j + \interpol^1 [ d^{j-1} \times \Deltah \tilde{d}^j ]}{3/2} \lebnorm{ \projectcgzero\phi}{3}
\\
\lesssim 
\left [
\lebnorm{\nabla d^{j-1} }{2} 
\lebnorm{v^j}{6} 
+ \lebnorm{d^{j-1} \times \Deltah \tilde{d}^j}{2} 
\right ]
\lebnorm{ \projectcgzero\phi}{3}
.
\end{multline*}
Using the Sobolev embedding $L^6(\Omega) \hookrightarrow H^1(\Omega)$ and the \textit{a priori} Estimates \eqref{eq:discrete-energy-inequality-projection} yields the result.
\\
For the projected time derivative $\discreteDiff_t d^j$, we reuse the fact that
$$
\abs{d^j_z - d^{j-1}_z}
=
\abs{\frac{\tilde{d}^{j-1}_z}{\abs{\tilde{d}^{j-1}_z}} - \frac{d^{j-1}_z}{\abs{d^{j-1}_z}}}
\leq
\abs{\tilde{d}^j_z - d^{j-1}_z}
$$
holds for all $z\in\nodes$, which has already been used in the proof of Lemma \ref{lem:energy-decreasing-projection-cg}. Combining this with Equation \eqref{eq:norm-equivalence-mass-lumping-p} yields
\begin{multline*}
\lebnorm{d^j - d^{j-1}}{3/2}^{3/2}
\leq
C
\int_{\Omega } \interpol^1 (\abs{d^j - d^{j-1}}^{3/2}) \diff x
\\
\leq 
C
\int_{\Omega } \interpol^1 (\abs{\tilde{d}^j - d^{j-1}}^{3/2}) \diff x
\leq C
\lebnorm{\tilde{d}^j - d^{j-1}}{3/2}^{3/2}.
\end{multline*}
This proves the result.
The discrete time derivative of the velocity can also be estimated using a duality argument, the previous \textit{a priori} estimates, the stability of the projection $\projectV$ and the Sobolev embedding $\H{2} \hookrightarrow  \L{\infty}$, since
\begin{align*}
(\discreteDiff_t v^j, \projectV \vv)_2 
\lesssim
\left [
\lebnorm{\nabla v^j}{2} 
+
\lebnorm{\interpol^1 [d^{j-1} \times \Deltah\tilde{d}^j ]}{2}
\right ]
\norm{\projectV \vv}_{\H{2}}
\end{align*}
for all $\vv\in \velspace$.
\end{proof}
\subsection{Further Properties}
\begin{prop}[Discrete energy inequality]
Let $u^j = (v^j,d^j_{\circ})\in \mixedfemspacecg$ be a solution to Algorithm~\eqref{algo:cg} for $1\leq j \leq J$. Then, the discrete energy-variational inequality
\begin{align}
\begin{split}\label{discrete_envar}
        \discreteDiff_t E^j
        &     
        + \mu \left ( \nabla v^j , \nabla v^j-  \nabla \projectV\vv \right )  
        +  \norm{d^{j-1} \times \Deltah \tilde{d}^{j} }^2_h  
\\&
        - (\discreteDiff_t v^j,\projectV\vv) - \left ( (v ^{j-1}\cdot \nabla)  v ^j , \projectV\vv \right ) - \frac{1}{2} \left ( ( \nabla \cdot  v ^{j-1} )v^j , \projectV\vv\right )
\\ &
        - \left ( ({\nabla d^{j-1}}]^T [d^{j-1}\times \interpol^1 [ d^{j-1} \times \Deltah \tilde{d}^{j} ] ), \projectV\vv \right )
\\&
        + \mathcal{K}(\projectV\vv) \left ( \frac{1}{2}\Vert v^{j-1}\Vert_{L^2(\Omega)}^2 + \frac{1}{2}\Vert \nabla d^{j-1}\Vert_{L^2(\Omega)}^2  -E^{j-1} \right ) 
  \leq 0
   \, ,
\end{split}
\end{align}
with the variable $E^j$ and the regularity weight $\mathcal{K}$ given by 
\begin{align*}
 E^j := \frac{1}{2}\ltwonorm{v^j}^2 + \frac{1}{2}\ltwonorm{
 \nabla d^j}^2 ,
 \qquad
 \mathcal{K}(\vv ) : ={}& \frac{1}{2}
\Vert  \vv \Vert_{L^{\infty} (\Omega)} ^2 
 \, ,
\end{align*}
holds for all $\vv\in \velspace$.
\label{prop:discrete-envar-cg}
\end{prop}
\begin{proof}
Adding Equation \eqref{eq:projection_scheme_a} tested with $- \projectV \vv^j$ to Inequality \eqref{eq:discrete-energy-inequality-projection}, we find the variational-energy inequality
\begin{align*}
\begin{split}
        \frac{1}{2k}\Vert v^j & \Vert_{L^2(\Omega)}^2
        + \frac{1}{2k}\Vert \nabla d^j\Vert_{L^2(\Omega)}^2  
        + \ltwonorm{\interpol^1 [d^{j-1} \times \Deltah \tilde{d}^{j}] }^2 
\\& 
        + \frac{1}{2k}\ltwonorm{\nabla [\tilde{d}^j -d^{j-1}]}^2
        + \frac{k}{2}  \ltwonorm{ \discreteDiff_t {v}^j}^2
        + \mu \left ( \nabla v^j ,  \nabla v^j -  \nabla \projectV\vv  \right )_2 
 \\&  
        - (\discreteDiff_t v^j,\projectV\vv) - \left ( (v ^{j-1}\cdot \nabla)  v ^j ,\projectV \vv \right )_2 - \frac{1}{2} \left ( ( \nabla \cdot  v ^{j-1} )v^j ,\projectV \vv\right )_2
 \\ &
        - \left ( [{\nabla d^{j-1}}]^T [d^{j-1}\times \interpol^1 (d^{j-1} \times \Deltah \tilde{d}^{j}  )], \projectV\vv \right )_2
\\&
        \leq
        \frac{1}{2k}  \ltwonorm{v^{j-1}}^2
        + \frac{1}{2k} \ltwonorm{\nabla d^{j-1}}^2 .
\end{split}
\end{align*}
Adding the constructive zero
$$
\mathcal{K}(\projectV\vv) \left ( \frac{1}{2}\Vert v^{j-1}\Vert_{L^2(\Omega)}^2 + \frac{1}{2}\Vert \nabla d^{j-1}\Vert_{L^2(\Omega)}^2  -E^{j-1} \right )= 0
$$
implies Inequality \eqref{discrete_envar}. 
\end{proof}
Next, we consider the asymptotic behaviour of the \textit{projection error} which we define by
\begin{equation}\label{eq:def-projection-error}
r^j \coloneqq d^j - \tilde{d}^j=
 \interpol^1   \frac{\tilde{d}^j}{\vert{\tilde{d}^j}\vert} 
- 
\tilde{d}^j.
\end{equation}
\begin{lem}\label{lem:convergence-projection-err}
Let $(v^j,\tilde{d}^j_\circ)\in \mixedfemspacecg$ be a solution of Algorithm \eqref{algo:cg} for $1\leq j\leq J$. Then, the following error estimate holds,
\begin{align*}
\lebnorm{\interpol^1   \left ( \frac{\tilde{d}^j}{\vert{\tilde{d}^j}\vert} \right ) - \tilde{d}^j}{1}
=
\lebnorm{r^j}{1} &\lesssim  k^2 \norm{k^{-1} (\tilde{d}^j - d^{j-1})}_h^2.
\end{align*}
\end{lem}
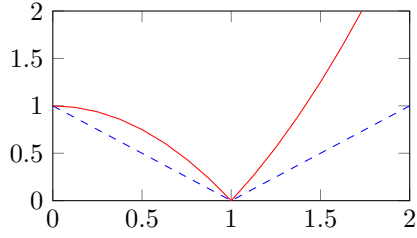
\begin{figure}
\centering
\begin{tikzpicture}
\begin{axis}[
        width=0.5\textwidth,
        height=0.2\textheight,
        xmin=0, xmax=2, 
        ymin=0, ymax=2, 
        domain=0:3   
]
\addplot [blue,dashed]    {abs(abs(x)-1)};
\addplot [red, no marks]    {abs(abs(x)^2-1)};
\end{axis}
\end{tikzpicture}
    \caption{Visualization of the functions $g(x)= \abs{\abs{x}-1}$ (blue, dashed) and $f(x) = \lvert {\abs{x}^2-1} \rvert$ (red).}
    \label{fig:visualization-square-root-abs}
\end{figure}
\begin{proof}
Since $r^j \in [\cgone]^{\dimension}$, it is sufficient to first consider the error on a nodal basis. Let $z\in \nodes$ and recall the definition of the projection, Equation \eqref{eq:projection-step-cg}.
Then, we observe
\begin{equation}
    \abs{r^j_z}
=
    \abs{d^j_z -\tilde{d}^j_z}
=
    \abs{d^j_z (1-  \abs{\tilde{d}^j_z})} 
\leq
    \abs{1- \abs{\tilde{d}^j_z}}
    .  \label{eq:projection-error-nodal-decomposition}
\end{equation}
This shows us that the projection error is dominated by the deviation of our intermediate solution $\tilde{d}^j$ from the unit-sphere constraint.
We apply the inequality (cf. Figure \ref{fig:visualization-square-root-abs})
$$
\lvert \abs{x}-1 \rvert \leq \lvert {\abs{x}^2-1} \rvert
$$
for all $x\in \mathbb{R}$, onto Inequality \eqref{eq:projection-error-nodal-decomposition} in order to obtain
\begin{equation}
\label{eq:application-taylor-sqrt}
    \abs{r^j_z}
\leq
    \abs{\abs{d^{j-1}_z +(\tilde{d}^j_z - d^{j-1}_z)}^2 - 1}
    \leq
    \abs{\tilde{d}^j_z-d^{j-1}_z}^2 
    +
    2
    \abs{d^{j-1}_z \cdot (\tilde{d}^j_z - d^{j-1}_z)}
\end{equation}
where we used that $d^{j-1}$ fulfills the unit-sphere constraint nodally. 
The last summand in Inequality \eqref{eq:application-taylor-sqrt} vanishes due to the nodal orthogonality relation, that is Corollary \ref{cor:orthogonality-cg}. Now, we can simply compute
\begin{align*}
    \lebnorm{r^j}{1}
    &\leq
    \sum_{z \in \nodes} \int_\Omega \phi_z \diff x \abs{r^j_z}
    \leq
    k^2
    \sum_{z \in \nodes} \int_\Omega \phi_z \diff x \abs{k^{-1}(\tilde{d}^j_z-d^{j-1}_z)}^2
\\
    &=
    k^2
    \hnorm{k^{-1} (\tilde{d}^j-d^{j-1})}^2
    \lesssim
    k^2 \ltwonorm{k^{-1}(\tilde{d}^j-d^{j-1})}^2.
\end{align*}
\end{proof}
\section{Convergence of Algorithm \ref{algo:cg}}
\label{sec:convergence-to-envar-sols-cg}
For the convergence analysis, we follow \cite[Section 3]{lasarzik-reiter23} closely.
\subsection{Interpolation in Time}
\label{sec:interpolation-in-time-cg}
For $1\leq j \leq J$ and $t\in ((j-1)k, jk)$, let the piecewise constant interpolates in time be defined by
\begin{align*}
    \overline{d}^k_h (t) &= d^j, & \overline{\tilde{d}}^k_h(t) &= \tilde{d}^j,  \\
    \underline{d}^k_h (t)&= d^{j-1}, & \underline{\tilde{d}}^k_h (t)&= \tilde{d}^{j-1} \, ,
\end{align*}
as well as the continuous, piecewise linear interpolates in time by
\begin{align*}
    d_k^h (t) &= d^{j-1} + \frac{t-(j-1)k}{k} (d^j - d^{j-1}), 
    \\
    \tilde{d}_k^h (t) &= \tilde{d}^{j-1} + \frac{t-(j-1)k}{k} (\tilde{d}^j - \tilde{d}^{j-1}).
\end{align*}
We also denote the discontinuous, piecewise constant interpolate for our \textit{approximate time derivative} by
\begin{equation}
w^k_h (t) = \frac{\tilde{d}^j-d^{j-1}}{k}
\label{eq:def-wkh}
\end{equation}
for $t\in ((j-1)k, jk)$.
For a smooth function in time and space $\phi \in \Cont_c^{\infty} ([0,T] \times \Omega)$ we define the piecewise constant interpolates by
\begin{align*}
    \overline{\phi}^k (t) &= \phi (jk), 
    &
    \overline{\phi}^k_h (t) &= \interpol^1 \phi (jk), 
    \\
    \underline{\phi}^k (t)&= \phi ((j-1)k),
    &
    \underline{\phi}^k_h (t) &= \interpol^1 \phi ((j-1)k) \, ,
\end{align*}
for $t\in ((j-1)k, jk)$.
\subsection{Converging Subsequences}
\label{sec:subsequences-cg}
The \textit{a priori} estimates --- Lemma \ref{lem:a-priori-estimates-projection} --- allow to infer converging subsequences that are not relabeled such that  
\begin{align}
\begin{split}\label{eqs:converging_subsequences}
\overline{v}^k_h,\underline{v}^k_h, v_h^k &\weakstarto \fat{v} \text{ in } L^{\infty} (0,T;\Ha), \\
\overline{v}^k_h,\underline{v}^k_h,v_h^k &\weakto \fat{v} \text{ in }  L^2(0,T;H^1_0(\Omega)), \\
\partial_t v_k^h &\weakto \partial_t \fat{v} \text{ in } L^2(0,T; (\velspace )^*) 
,\\
\overline{d}^k_h,\underline{d}^k_h,d_h^k 
&\weakstarto
\fat{d} \text{ in } L^{\infty} (0,T;H^1(\Omega)) \cap L^{\infty} ( (0,T)\times \Omega)) 
,\\
\overline{\tilde{d}}^k_h,\underline{\tilde{d}}^k_h, \tilde{d}_h^k 
& 
\weakstarto
 \tilde{\fat{d}} \text{ in } L^{\infty} (0,T;H^1(\Omega)) 
,\\
w^k_h &\weakto \fat{ W} \text{ in } L^2(0,T;L^{3/2}(\Omega)),
\\
\partial_t d^k_h &\weakto  \partial_t \fat{d} \text{ in } L^2(0,T;L^{3/2}(\Omega)),
\\
\interpol^1 ( \underline{d}^k_h \times \Deltah \overline{\tilde{d}}^k_h)
&\weakto \fat{ U} \text{ in } L^2(0,T;\L{2}), 
\\
\underline{E}^k_h(t) & \to E(t) \text{ for all } t \in [0,T]
.
\end{split}
\end{align}
The pointwise convergence $\underline{E}^k_h(t) \to E(t)$ follows from an application of Helly's selection principle \cite[Ex. 8.3]{brezis} for monotonic functions since $\underline{E}^k_h(t)\leq \mathcal{E} (\fat v, \fat d) (0)$ is non-increasing due to \eqref{eq:discrete-energy-inequality-projection}.
\\
In a standard fashion, we can also derive strong convergence for the velocity field applying the Aubin-Lions lemma onto the embeddings $(\velspace) \stackrel{c}{\hookrightarrow} \Ha \hookrightarrow (\velspace)^*$. Together with a standard interpolation estimate that yields
\begin{equation}\label{strong-convergence velocity}
    \overline{v}^k_h,\underline{v}^k_h,v_h^k \to \fat v \text{ in }  L^p(0,T;\Ha)
\end{equation}
for all $p\in[2,\infty)$.
The next corollary comprises the asymptotic behaviour of the projection error --- that is a direct consequence of Lemma \ref{lem:convergence-projection-err}, the definition of $w^k_h$ (see Equation \eqref{eq:def-wkh}), and the inverse estimate.
\begin{cor}
\label{cor:identification-rkh}
Under the previous assumptions, we have
$$
\norm{\overline{r}^k_h}_{L^1(0,T;L^1(\Omega))}
\lesssim 
\frac{k^2}{h^{\dimension /3}}
\norm{w^k_h}_{L^2(0,T;L^{3/2}(\Omega))}^2
.
$$
This implies
\begin{align}
    \overline{r}^k_h 
    &\to 0 \in L^1(0,T;L^1(\Omega)) \text{ for } k\in o(h^{\dimension /6})
    ,
    \label{eq:convergence-rkh}
    \\
    \frac{1}{k} \overline{r}^k_h 
    &\to 0 \in L^1(0,T;L^1(\Omega)) \text{ for } k\in o(h^{\dimension /3})
    ,
    \label{eq:convergence-rkh-over-k}
\end{align}
as $k,h\to 0$.
\end{cor}
The preceeding corollary implies that for $k\in o(h^{\dimension /6})$ we can identify
$$
\fat{d} =\tilde{\fat{d}}
$$
in a weak sense.
Further, for $k\in o(h^{\dimension /3})$, we can identify the limit of our approximate time derivative in a weak sense by
$$
\fat{W} = \partial_t \fat{d},
$$
since for $t\in ((j-1)k, jk)$, we can observe the decomposition
\begin{equation}
w^k_h (t) = \frac{\tilde{d}^j - d^{j-1}}{k} = \frac{{d}^j - d^{j-1}}{k} +\frac{\tilde{d}^j - d^{j}}{k} 
= \partial_t d^k_h (t) + \frac{1}{k} \overline{r}^k_h (t)
.
\label{eq:decomposition-wkh}
\end{equation}
The fact that all temporal interpolates have the same limit can be derived in a standard fashion. We here do this exemplary, i.e.
\begin{equation}
\norm{d^k_h - \overline{d}^k_h}_{L^1(0,T;\L{1})}
\leq
\norm{\overline{r}^k_h}_{L^1(0,T;\L{1})}
+
k
\norm{w^k_h}_{L^1(0,T;\L{1})}
.
\label{eq:l1l1limit-temporal-interpolates}
\end{equation}
Corollary \ref{cor:identification-rkh} and the convergence of subsequences \eqref{eqs:converging_subsequences} deliver the result.
Left to identify is the limit $\fat{U}$.
\begin{prop}
\label{prop:identification-of-d-times-deltad}
Under the previous assumptions and additionally assuming
$$
k\in o (h^{1+ \dimension /6}) \, ,
$$
we can identify 
$$
\fat{U} = \nabla \cdot (\fat{d} \times \nabla \fat{d})
$$
in a weak sense.
\end{prop}
\begin{proof}
The definition of the discrete Laplacian yields for a smooth test function $\phi \in C_c^{\infty} (\Omega \times (0,T))$ that
\begin{equation*}
    ( \underline{{{d}}}^k_h \times \Deltah \overline{{\tilde{d}}}^k_h  , \overline{\phi}^k_h  )_h
    =
    -(  \Deltah \overline{{\tilde{d}}}^k_h  , \underline{{{d}}}^k_h \times \overline{\phi}^k_h  )_h
    =
    (  \nabla  \overline{{\tilde{d}}}^k_h  , \nabla \interpol^1[\underline{{{d}}}^k_h \times  \overline{\phi}^k_h  ] )_2  
    .
\end{equation*}
Subtracting the discrete analogue of $\nabla \cdot (\fat{d} \times \nabla \fat{d})$ from the right-hand side yields
\begin{multline*}
\abs{(  \nabla  \overline{{\tilde{d}}}^k_h  , \nabla \interpol^1[\underline{{{d}}}^k_h \times  \overline{\phi}^k_h  ] )_2  
-
(  \nabla  \overline{{\tilde{d}}}^k_h  ,\overline{{\tilde{d}}}^k_h \times  \nabla  \overline{\phi}^k_h  )_2  
}
\\
\leq
\abs{(  \nabla  \overline{{\tilde{d}}}^k_h  , \nabla (\interpol^1- \identitymatrix )[\underline{{{d}}}^k_h \times  \overline{\phi}^k_h  ] )_2  
}
+
\abs{
(  \nabla  \overline{{\tilde{d}}}^k_h  ,
\nabla [(\underline{{{d}}}^k_h  - \overline{{\tilde{d}}}^k_h) \times  \overline{\phi}^k_h]  )_2  
}
\\
\lesssim
\ltwonorm{ \nabla  \overline{{\tilde{d}}}^k_h}
\Big (
   h
   \ltwonorm{\nabla \underline{{{d}}}^k_h} \lebnorm{\nabla \overline{\phi}^k_h  }{\infty}
   \\
   +
   \frac{k}{h^{1+\dimension /6}}
   \lebnorm{w^k_h}{3/2} \sobnorm{ \overline{\phi}^k_h  }{1}{\infty}
\Big )
.
\end{multline*}
Under the assumed time step restriction $k \in o (h^{1+\dimension/6})$, Corollary \ref{cor:identification-rkh} and the converging subsequences \eqref{eqs:converging_subsequences} deliver the result.
\end{proof}
Using the compact embeddings of the Aubin--Lions--Simon lemma \cite{simons87},
\begin{align}
\begin{aligned}\label{d_lions_aubin}
    \{ \fat d \in L^2(0,T;\H{1}) : \partial_t \fat d \in L^{2}(0,T;\L{3/2}) \}
    &\stackrel{c}{\hookrightarrow} L^2(0,T;\L{2}),
    \\
     \{ \fat d \in L^{\infty} (0,T;\H{1}) : \partial_t \fat d \in L^{2}(0,T;\L{3/2}) \}
    &\stackrel{c}{\hookrightarrow} C^0(0,T;\L{2}) \, ,
\end{aligned}
\end{align}
we can even infer strong subconvergence of the director field
\begin{align*}
    \overline{d}^k_h,\underline{d}^k_h,d_k^h 
    \to\fat  d \text{ in } L^2 (0,T;L^2(\Omega)) .
\end{align*}
Inferring strong convergence for all three temporal interpolates can be justified by applying a standard interpolation inequality (\textit{cf.} \cite[p. 192 ff.]{bennett-sharpley88}) on their respective difference, e.g.
\begin{equation*}
\norm{d^k_h - \underline{d}^k_h}_{L^2(0,T;\L{2})}^2
\leq
\norm{d^k_h - \underline{d}^k_h}_{L^1 (0,T;\L{1})}
\norm{d^k_h - \underline{d}^k_h}_{L^\infty (0,T;\L{\infty})}
\, ,
\end{equation*}
where the right hand side vanishes due to Equation \eqref{eq:l1l1limit-temporal-interpolates}.
Then, a Riesz-Thorin interpolation argument and the uniform bound in $ L^{\infty} (0,T;L^{\infty}(\Omega))$  yield strong convergence of
\begin{equation}\label{eq:director_strong_convergence}
    \overline{d}^k_h,\underline{d}^k_h,d_k^h
    \to \fat d \text{ in } L^p (0,T;L^p(\Omega)) \text{ for all }p\in [1,\infty).
\end{equation}
Finally, since $\fat{d} \in C^0(0,T;\L{2})$, we can conclude that $\fat{d} \in C_{w}(0,T;\H{1})$ due to the Lions--Magenes lemma \cite[Lemma II.5.9]{boyer-fabrie13}.
\subsection{Convergence of the Sphere Constraint, Initial and Boundary Conditions}
\label{sec:ics-bcs-convergence-cg}
The convergence of the boundary and initial conditions as well as the divergence-zero condition can be handled using standard Finite Element approximation arguments, see e.g.~\cite{lasarzik-reiter23}.
Further, the unit-norm constraint is fulfilled asymptotically since
$$
\lebnorm{\abs{d^j}^2 - 1}{2}
=
\lebnorm{\abs{d^j}^2 - \interpol^1 \abs{d^j}^2}{2}
\leq
C h \lebnorm{\nabla d^j}{2} \lebnorm{ d^j}{\infty}
$$
vanishes as $h\to 0$.
This leads to the following corollary.
\begin{cor}
Let the assumptions of Theorem \ref{thm:main-result-subconvergence-cg} be fulfilled and let $(v^j, \tilde{d}^j)$ be a solution to Algorithm \ref{algo:cg} for all $1\leq j \leq J$.
Then, the temporal interpolate $d^k_h$ admits the convergence
\begin{align*}
    \lebnorm{ \abs{d^k_h (t)}^2 -1 }{2} \in \mathcal{O}(h)
\end{align*}
a.e.~in $[0,T]$.
\end{cor}
\subsection{Director Equation}
For a smooth and compactly supported test function $\phi \in C_c^{\infty} (\Omega \times [0,T])$, we observe that the trace of its interpolation is also zero, i.e.~$\tr (\overline{\phi}^k_h) = 0$. Therefore, we can use $\overline{\phi}^k_h$ as a test function for the Equation \eqref{eq:projection_scheme_c}.
Integrating in time leads us to
$$
\int_0^T
(w^k_h,\overline{\phi}^k_h)_h
 + (\underline{{d}}^k_h \times [{\nabla \underline{{d}}^k_h}\overline{v}^k_h], \interpol^1 [\underline{{d}}^k_h \times \overline{\phi}^k_h])_2
        {-}(\underline{{d}}^k_h \times \Deltah \overline{\tilde{d}}^k_h,\underline{{d}}^k_h \times \overline{\phi}^k_h)_h
        \diff \tau = 0 .
$$
The only thing left to do is the transition to a $L^2$-scalar product.
The convergence can then be derived in standard fashion from the weak convergences \eqref{eqs:converging_subsequences}, their identified limits and the strong convergences \eqref{strong-convergence velocity} and \eqref{eq:director_strong_convergence}.
We tackle each term independently. The first by applying the interpolation estimate and the inverse estimate, such that
\begin{equation*}
    \lvert 
        \int_0^T
        (w^k_h,\overline{\phi}^k_h)_h - (w^k_h,\overline{\phi}^k_h)_2
        \diff \tau
    \rvert
    \lesssim
    h
    \norm{w^k_h}_{L^2(0,T;\L{3/2})}
    \norm{\nabla \overline{\phi}^k_h}_{L^2(0,T;\L{3})} .
\end{equation*}
Accordingly, we observe for the third term by repeatedly using the interpolation error and inverse estimate 
\begin{multline*}
    \lvert 
        \int_0^T
        (\underline{{d}}^k_h \times \Deltah \overline{\tilde{d}}^k_h,\underline{{d}}^k_h \times \overline{\phi}^k_h)_h
         -
         (\interpol^1 [\underline{{d}}^k_h \times \Deltah \overline{\tilde{d}}^k_h],\underline{{d}}^k_h \times \overline{\phi}^k_h)_2
        \diff \tau
    \rvert
    \\
    \leq
    \lvert 
        \int_0^T
        (\underline{{d}}^k_h \times \Deltah \overline{\tilde{d}}^k_h,\underline{{d}}^k_h \times \overline{\phi}^k_h)_h
         -
         (\interpol^1 [\underline{{d}}^k_h \times \Deltah \overline{\tilde{d}}^k_h],\interpol^1 [\underline{{d}}^k_h \times \overline{\phi}^k_h])_2
        \diff \tau
    \rvert
    \\+
    \lvert 
        \int_0^T
        (\interpol^1 [\underline{{d}}^k_h \times \Deltah \overline{\tilde{d}}^k_h], ( \interpol^1 -\identitymatrix) [\underline{{d}}^k_h \times \overline{\phi}^k_h])_2
        \diff \tau
    \rvert
    \\
    \lesssim
    h
    \norm{\interpol^1 [\underline{{d}}^k_h \times \Deltah \overline{\tilde{d}}^k_h]}_{L^2(0,T;\L{2})}
    \norm{\nabla \overline{\phi}^k_h}_{L^2(0,T;\L{2})} 
    \norm{\underline{{d}}^k_h}_{L^{\infty}(0,T;\H{1})} 
    .
\end{multline*}
The second term follows in the exact same fashion.
\subsection{Variational Inequality}
\label{sec:variational-inequality-cg}
Testing Inequality \eqref{discrete_envar} with the temporal interpolate $\overline{\phi}^k $ of a smooth test function $\phi \in C_c^{\infty} ((0,T))$, $\phi \geq 0$ and integrating in time yields
\begin{align}
\nonumber
    -\int_0^T   
    \underline{E}^k_h \discreteDiff_t\hat{\phi}^k 
    \diff \tau
        &     
        + \overline{\phi}^k 
        \left[ 
        \mu \left ( \nabla \overline{v}^k_h , \nabla \overline{v}^k_h-  \nabla \projectV \overline{\vv}^k \right )  
        +  \ltwonorm{ \interpol^1 [\underline{d}^k_h\times \Deltah \overline{\tilde{d}}^k_h] }^2  
        \right]
        \diff \tau
\\ \nonumber
        &+
        \int_0^T  
        \discreteDiff_t \hat{\phi}^k
        ( \underline{v}^k_h, \projectV \underline{\vv}^k)_2
         +
        \overline{\phi}^k 
        ( \underline{v}^k_h,\discreteDiff_t \projectV \overline{\vv}^k)_2
        \diff \tau
\\& \label{eq:discrete-variational-inequ-integrated-in-time}
        -
        \int_0^T 
        \overline{\phi}^k [
            \left ( (\underline{v}^k_h\cdot \nabla)  \overline{v}^k_h , \projectV\overline{\vv}^k \right )_2 
        +
        \frac{1}{2} 
            \left ( ( \nabla \cdot  \underline{v}^k_h )\overline{v}^k_h , \projectV\overline{\vv}^k\right )_2
        ]
        \diff \tau
\\ & \nonumber
        - \int_0^T   \overline{\phi}^k  \left ( [\nabla \underline{d}^k_h]^T [\underline{d}^k_h\times  \interpol^1 ( \underline{d}^k_h \times \Deltah \overline{\tilde{d}}^k_h )], \projectV\overline{\vv}^k \right )_2
        \diff \tau
\\& \nonumber
        + \int_0^T   \overline{\phi}^k 
        \mathcal{K}(\projectV\overline{\vv}^k) \left ( \frac{1}{2}\Vert \underline{v}^k_h\Vert_{L^2(\Omega)}^2 + \frac{1}{2}\Vert \nabla \underline{d}^k_h\Vert_{L^2(\Omega)}^2  -\underline{E}^k_h \right ) 
        \diff \tau
  \leq 0 \, ,
\end{align}
for all $\vv \in \Cont^1([0,T];(\velspace)^*) \cap C^0 ([0,T];\velspace)$. Hereby, we applied the discrete integration by parts rule, Equation \eqref{eq:discrete-integration-by-parts-in-time}, onto the terms $\int_0^T \overline{\phi}^k \discreteDiff_t \overline{E}^k_h \diff \tau$ and $\int_0^T \overline{\phi}^k [ -  (\discreteDiff_t \overline{v}^k_h,\projectV\vv)] \diff \tau$.
We used that since $\phi$ has compact support on $(0,T)$, terms of the form
$$
\int_0^T   \discreteDiff_t (\overline{\phi}^k   \overline{a}^k) \diff \tau
= \phi^J   a^J - \phi^0 a^0
=0
$$
vanish.
\\
Regarding the convergence, the main challenge lies in the Ericksen stress tensor. We use terms from the discrete Energy-Dissipation mechanism and our potential $\mathcal{K}$ as a regularisation in order to convexify the terms. We collect the following lemma which will be proven later.
\begin{lem}
Let $d^m, G^m, U^m, \tilde{v}^m$ be functions, such that 
\begin{align*}
    d^m &\to \fat{d} \in L^p(0,T; L^p( \Omega;\mathbb{R}^{\dimension } )) \text{ for all }1\leq p < \infty
    ,\\
    G^m &\weakto \nabla \fat{d} \in L^{2}(0,T;L^2( \Omega;\mathbb{R}^{\dimension \times \dimension} ) )
    ,\\
    U^m &\weakto \fat{U} \in L^2(0,T;L^2( \Omega;\mathbb{R}^{\dimension } ))
    ,\\
    \tilde{v}^m &\to \vv \in L^\infty(0,T;L^{\infty}( \Omega;\mathbb{R}^{\dimension } ))
    .
\end{align*}
as $m\to\infty$. Further we assume the uniform bounds
$$
\sup_{m} \norm{d^m}_{L^\infty((0,T) \times \Omega)} 
\leq 1
, \quad
\sup_{m} \norm{\tilde{v}^m}_{L^\infty(0,T;\H{2})} 
\leq C
$$
for a generic constant $C>0$.
Then, we can infer the existence of subsequences that we do not relabel such that
\begin{multline*}
    \int_0^T
    \phi
    [\lebnorm{\fat{U} }{2}^2  
        - \left ( [\nabla \fat{d}]^T [\fat{d} \times \fat{U} ], \vv \right )_{2}
        + \mathcal{K}(\vv) 
        \frac{1}{2}\lebnorm{\nabla \fat{d}}{2}^2 ]
        \diff \tau
\\ 
\leq
\liminf_{m\to \infty}
\int_0^T
\phi
    [\lebnorm{U^m}{2}^2  
    - \left ( [G^m]^T [d^{m}\times U^m ], \tilde{v}^m  \right )_{2}
    + \mathcal{K}(\tilde{v}^m ) \frac{1}{2}\lebnorm{G^m}{2}^2]
    \diff \tau
     \, ,
\end{multline*}
for all $\phi \in L^{\infty} (0,T)$, $\phi \geq 0$.
\label{lem:convergence-ericksen-stress-in-envar}
\end{lem}
In order to apply this lemma, we set $\underline{d}^k_h = d^m$, $G^m = \nabla \underline{d}^k_h$ and $U^m =\interpol^1 [\underline{d}^k_h \times \Deltah \overline{\tilde{d}}^k_h ]$ and refer to the previous argumentation in this section. For the strong convergence of the test function $\projectV\overline{\vv}^k = \tilde{v}^m$, we here make use of the Gagliardo-Nirenberg estimate
\begin{equation*}
    \lebnorm{\projectV\overline{\vv}^k-\overline{\vv}^k}{\infty}
    \leq
    C
    \lebnorm{\projectV\overline{\vv}^k-\overline{\vv}^k}{2}^{1 - \dimension /4}
    \Hsobnorm{\projectV\overline{\vv}^k-\overline{\vv}^k}{2}^{\dimension /4}
    .
\end{equation*}
Then, the convergence follows with the stability properties of the projection $\projectV$ for all $\overline{\vv}^k\in  \V \cap\H{2}$, see Equation \eqref{eq:approximation-properties-convergence-projection-projectV}.
All other terms in the discrete variational inequality \eqref{eq:discrete-variational-inequ-integrated-in-time}  converge in a standard fashion based on the weak convergences \eqref{eqs:converging_subsequences} and the strong convergences \eqref{strong-convergence velocity}, \eqref{eq:director_strong_convergence}.
Finally, we end up with the inequality
\begin{align*}
    -\int_0^T   
    E \phi^{\prime} 
    \diff \tau
        &     
        + {\phi} 
        \left[ 
        \mu \left ( \nabla \fat{v} , \nabla \fat{v}-  \nabla  {\vv} \right )  
        +  \ltwonorm{  \fat{d}\times \Delta \fat{d} }^2  
        \right]
        \diff \tau
\\ 
        &+
        \int_0^T  
        \phi^{\prime} 
        ( \fat{v},  {\vv})_2
         +
        {\phi} 
        ( \fat{v},\partial_t  {\vv})_2
        \diff \tau
        -
        \int_0^T 
        {\phi} 
            \left ( (\fat{v}\cdot \nabla)  \fat{v} , {\vv} \right )_2 
        \diff \tau
\\ & 
        - \int_0^T   {\phi}  \left ( [\nabla \fat{d}]^T [\fat{d}\times  ( \fat{d} \times \Delta {\fat{d}} )], {\vv} \right )_2
        \diff \tau
\\& 
        + \int_0^T   {\phi} 
        \mathcal{K}(\vv) \left ( \frac{1}{2}\Vert \fat{v}\Vert_{L^2(\Omega)}^2 + \frac{1}{2}\Vert \nabla \fat{d}\Vert_{L^2(\Omega)}^2  -{E} \right ) 
        \diff \tau
  \leq 0.
\end{align*}
Applying \cite[Lemma 2.4]{lasarzik22} and density arguments deliver the energy-variational Inequality \eqref{envarform}.
We finish this section with the proof of the preceeding lemma.
\begin{proof}[Proof of Lemma {\ref{lem:convergence-ericksen-stress-in-envar}}]
    We follow \cite[Sec. 3.4]{lasarzik-reiter23}.
    First, we apply a standard \textit{completing the square} approach
    \begin{multline*}
    \lebnorm{U^m }{2}^2  
    - \left ( [G^m]^T [d^m\times U^m ], \tilde{v}^m \right )_{2}
    \\
    =
     \lebnorm{U^m + \frac{1}{2} d^m \times [G^m \tilde{v}^m]}{2}^2
    - \frac{1}{4}\lebnorm{d^m \times [G^m \tilde{v}^m]}{2}^2 
    .
    \end{multline*}
    Then, we decompose the last term on the right hand-side by
    \begin{align*}
         - \lebnorm{d^m \times [G^m \tilde{v}^m]}{2}^2 
        = &
         -\int_\Omega (\tilde{v}^m )^T (G^m)^T (\abs{d^m}^2 \identitymatrix -d^m\otimes d^m) G^m \tilde{v}^m \diff x
    \\
    =&
         -\int_\Omega \abs{d^m}^2  (\tilde{v}^m )^T (G^m)^T  G^m \tilde{v}^m \diff x
         \\&+
         \int_\Omega (\tilde{v}^m )^T (G^m)^T (d^m\otimes d^m) G^m \tilde{v}^m \diff x
    \\
    =&
         -\lebnorm{\abs{d^m} G^m \tilde{v}^m }{2}^2
         +
         \lebnorm{d^m \cdot G^m \tilde{v}^m}{2}^2.
    \end{align*}
    The second term can now again be handled by weak lower semicontinuity. The sum of the first term and our discrete potential can be written using the following Tensor identities
    \begin{align*}
        \nabla d : \nabla d
        =
        I: ([\nabla d]^T \nabla d)
        , \quad 
         ([\nabla d] v)^2
         =
         v\otimes v : ([\nabla d]^T \nabla d)
    \end{align*}
    such that the addition of potential $\mathcal{K}$ makes the resulting matrix positive semi-definite, i.e.
    \begin{multline*}
        2 \mathcal{K}(\tilde{v}^m ) \lebnorm{G^m}{2}^2
            -\lebnorm{\abs{d^m} G^m \tilde{v}^m }{2}^2
    \\
    =
            \lebnorm{ \tilde{v}^m }{\infty}^2 \lebnorm{G^m}{2}^2
            -\lebnorm{\abs{d^m} G^m \tilde{v}^m }{2}^2
    \\ =
    \int_\Omega (\lebnorm{ \tilde{v}^m }{\infty}^2 \identitymatrix - \abs{d^m}^2 \tilde{v}^m \otimes \tilde{v}^m)
    : ([G^m]^T G^m) \diff x
    .
    \end{multline*}
    It is easy to confirm that the matrix $(\lebnorm{ \tilde{v}^m }{\infty}^2 \identitymatrix - \abs{d^m}^2 \tilde{v}^m \otimes \tilde{v}^m)$ is positive semidefinite. Then, there exists an orthogonal decomposition. Even simpler, we observe the explicit decomposition into
    \begin{multline*}
        \lebnorm{\tilde{v}^m }{\infty}^2 \identitymatrix -  \abs{d^m}^2 \tilde{v}^m \otimes \tilde{v}^m
    \\=
        \left (\sqrt{\lebnorm{\tilde{v}^m }{\infty}^2 - \abs{d^m}^2\abs{\tilde{v}^m}^2} \right )^2 \identitymatrix
    +
     \abs{d^m}^2 \left (\abs{\tilde{v}^m}\identitymatrix -\frac{\tilde{v}^m\otimes \tilde{v}^m}{\abs{\tilde{v}^m}} \right )^2
    .
    \end{multline*}
    Using this, we can rewrite everything as
    \begin{multline}
    \lebnorm{U^m}{2}^2  
    - \left ( [G^m]^T [d^m\times U^m ], \tilde{v}^m \right )_{2}
    + \mathcal{K}(\tilde{v}^m) \frac{1}{2}\lebnorm{G^m}{2}^2
    \\ 
    =
     \lebnorm{U^m + \frac{1}{2} d^m \times [G^m \tilde{v}^m]}{2}^2
     + \frac{1}{4}\lebnorm{d^m \cdot G^m \tilde{v}^m}{2}^2
    \\
    +\frac{1}{4}\lebnorm{\left (\sqrt{\lebnorm{\tilde{v}^m }{\infty}^2 - \abs{d^m}^2\abs{\tilde{v}^m}^2} \right )
    G^m }{2}^2 
    \\
    + \frac{1}{4}\lebnorm{\abs{d^m} \left (\abs{\tilde{v}^m}\identitymatrix -\frac{\tilde{v}^m\otimes \tilde{v}^m}{\abs{\tilde{v}^m}} \right ) G^m}{2}^2 
     .
     \label{eq:reformulated-convex-terms}
    \end{multline}
    The norm is known to be weakly lower semicontinuous. It suffices to consider the convergence properties of its contents on the right hand-side of \eqref{eq:reformulated-convex-terms}.
    All of them are bounded and therefore allow to infer converging subsequences. 
    Their limits can be identified by using the assumed strong and weak convergences \eqref{eqs:converging_subsequences}.
\end{proof}
\section{A Discontinuous Finite Element Scheme}\label{sec:discrete-system-dg}
We introduce our Finite Element approximation based on a piecewise constant approximation of the director field.
First, we discretize our constant-in-time Dirichlet boundary conditions prescribed by the initial condition. 
Since we assumed the regularity $\fat{d}_0 \in C^{2} (\bar{\Omega})$ we can interpolate the initial condition.
We define the discrete boundary condition $\dbc \in [\dgzero]^{\dimension}$ in terms of the facet interpolator at the barycenter, i.e.
\begin{equation*}
    \dbc \coloneqq
    \sum_{F\in\facets^D} \interpol^F (\fat{d}_0) \chi_{T(F)}
    =
    \sum_{F\in\facets^D} \fat{d}_0 (x_F) \chi_{T(F)}
\end{equation*}
with $T(F)$ being the cell $T\in\mesh$ bordering the facet $F \in \facets^D$.
Further, the definition of the discrete Laplacian needs to be adapted to the discontinuous case.
\begin{defn}
We define the discrete Laplacian $\Deltah : [\dg^0]^{\dimension} \to [\dg^0]^{\dimension}$ as solution to the equation
\begin{align}
\begin{split} \label{eq:discrete-laplacian-dg}
    (-\Deltah d, b)_2 
    =&
    (\adlifting{d}{\dbc}, \adlifting{b }{0} )_2
    \\&
    + \sum_{F \in \facets^i } \frac{\alpha }{h_F} \int_F \jump{d} \cdot \jump{b} \diff s
    +\sum_{F \in \facets^D} \frac{\alpha}{h_F} \int_F (d - \dbc)\cdot b \diff s
\end{split}
\end{align}
for all $b \in [\dg^0]^{\dimension}$, where $\adlifting{d}{\dbc}$ is the discrete gradient of $d\in [\dg^0]^{\dimension}$ defined in Equation \eqref{eq:def-discrete-lifting-reconstructed-gradient}.
\label{def:discrete-laplacian-dg}
\end{defn}
The solution to Equation \eqref{eq:discrete-laplacian-dg} can also be found using a mixed method. This is advantageous regarding the implementation, since it allows to use a formulation that is realizable in standard Finite Element software.
Instead of solving Equation \eqref{eq:discrete-laplacian-dg}, we solve for $(d_h, q_h, \psi_h) \in [\dg^0]^{\dimension} \times [\dg^0]^{\dimension} \times [\dg^0]^{\dimension \times \dimension}$ such that the equation system
\begin{subequations}
\begin{align}
    (q_h, b)_2 
    &=
    \mathcal{B}(b,\psi_h)
    +
    \sum_{F \in \facets^i}
    \frac{\alpha}{h_F} \int_F\jump{d}\cdot \jump{b} \diff s  
    +
    \sum_{F \in \facets^D}
    \frac{\alpha}{h_F} \int_F (d- \dbc)\cdot b \diff s 
    \, ,
    \label{eq:mixed-method-discrete-laplacian1}
    \\
    \label{eq:mixed-method-discrete-laplacian2}
    (\psi_h, \tau)_2
    &=
    \mathcal{B}(\dirh, \tau)
    +\sum_{F \in \facets^D}
    \int_F 
    \dbc \cdot \tau \normal
    \diff s
    \, ,
\end{align}
\label{eq:mixed-method-discrete-laplacian}
\end{subequations}
is fulfilled for all $(b,\tau)\in [\dg^0]^{\dimension} \times [\dg^0]^{\dimension \times \dimension}$ with 
\begin{equation*}
    \mathcal{B}(b, \tau)
    \coloneqq
    -
    \sum_{F \in \facets^i}
    \abs{F}
    \jump{b}_F\cdot \avg{\tau }_F  \normal_F 
    -
    \sum_{F \in \facets^D}
    \abs{F}
    b\cdot \tau \normal 
    .
\end{equation*}
Note that $\mathcal{B}$ is simply the Finite Element formulation of the global lifting operator, cf.~Equations \eqref{eq:def-discrete-lifting-local} and \eqref{eq:def-discrete-lifting-reconstructed-gradient}.
The existence and uniqueness of a solution $(q_h, \psi_h)$ for a prescribed $d_h \in [\dg^0]^{\dimension}$, as well as of a solution $(d_h, \psi_h)$ for a given $q_h \in [\dg^0]^{\dimension}$ follows from an application of the Lax-Milgram theorem for $\alpha > 0$.
By simple computations it follows that a solution to Equations \eqref{eq:mixed-method-discrete-laplacian} also fulfills Equation \eqref{eq:discrete-laplacian-dg}. We state this as a corollary.
\begin{cor}
\label{lem:identification-mixed-method}
    The triple 
    $$
    {(d_h, q_h, \psi_h) \in [\dg^0]^{\dimension} \times [\dg^0]^{\dimension} \times [\dg^0]^{\dimension \times \dimension} }
    $$
    solves Equations \eqref{eq:mixed-method-discrete-laplacian} if and only if $\psi_h = \adlifting{d_h}{\dbc}$ and $q_h = -\Deltah \dirh $.
\end{cor}
Now, we can define our numerical method using a discontinuous Galerkin approach for the director equation.
\begin{algo}
\label{algo:dg}
Let $(v^{0}, d^{0})= (\projectV  \fat{v}_0, \interpol^0 \fat{d}_0)$.
For $1 \leq j\leq J$, $(v^{j-1}, d^{j-1}_\circ )\in \mixedfemspacedg$, we want to find $(v^{j}, \tilde{d}^{j}_\circ )\in \mixedfemspacedg$, such that
\begin{subequations}
    \label{eqs:projection-scheme-dg}
    \begin{multline}        
        \label{eq:projection_scheme-dg_a}
            \frac{1}{k}(\velh^j - \velh^{j-1},a)_2 
            + 
            \mu (\nabla v^j, \nabla a)_2
            +
            ((v^{j-1}  \cdot \nabla) v^j,a)_2 
        \\
            + 
            \frac{1}{2} ([\nabla \cdot v^{j-1}] v^j,a)_2 
            -
            T_h^E(\dir_h^{j-1}, a, \Deltah \tilde{d}^{j})
            = 0 \,,
    \end{multline}
    \begin{equation}
        \label{eq:projection_scheme-dg_c}
            \frac{1}{k}(\tilde{\dirh}^{j} - \dirh^{j-1},c)_2 
            + 
            T_h^E(\dir_h^{j-1}, v^j, c)
            -
            (d^{j-1} \times \Deltah \tilde{d}^{j},d^{j-1} \times c)_2
            = 0,
    \end{equation}
\end{subequations}
for all $(a,c) \in \mixedfemspacedg$ with the discrete Ericksen stress tensor defined by
\begin{align}
    \label{eq:def-ThE-2}
    T_{h}^E (\dirh^{j-1}, v, q)
\coloneq &
    \left (
        \dirh^{j-1} \times \adlifting{\dirh^{j-1}}{\dbc} v, \dirh^{j-1} \times q
    \right )_2
    .
\end{align}
Then, we apply the discrete normalization operator
\begin{align}\label{eq:projection-step-dg}
     d^j = \interpol^0 \frac{\tilde{d}^j}{ \abs{\tilde{d}^j} } .
\end{align}
\end{algo} 
\begin{rem}
    One could also discretize the Ericksen stress tensor in a discontinuous Galerkin fashion using higher order liftings by
    \begin{align}
        \label{eq:def-ThE-1}
        T_{h}^E (\dirh^{j-1}, v, q)
    \coloneq &
        \left (
            \dirh^{j-1} \times \bar{R}^{2,\dbc}_h (\jump{\dirh^{j-1}}) v, \dirh^{j-1} \times q
        \right )_2
    \\=&
    -   \sum_{F \in \facets^i}
                \int_F 
                \jump{\dirh^{j-1}} \cdot \avg{\dirh^{j-1} \times [\dirh^{j-1} \times q]} (v \cdot \normal)
                \diff s .
        \nonumber
    \end{align}
    The difference between these two discretizations can be shown to vanish asymptotically.
\end{rem}
In addition to the kinetic energy, see Equation \eqref{eq:definition-energy}, we now define the discrete energies of our system by
\begin{align*}
    \energyeladG{ ({\dir})} \coloneq & \frac{1}{2}
    \ltwonorm{\adlifting{\dir}{\dbc}}^2
    ,
    \quad
    &
    \energystabdG{ ({\dir})} 
    \coloneq 
    \frac{1}{2}
    \abs{\dir}_{J,i}^2
    +
    \frac{1}{2}
    \abs{\dir -\dbc}_{J,D}^2
    ,
    \\
    \energy^h (v, d) \coloneq & 
    \frac{1}{2}
    \ltwonorm{v}^2
    +
    \energyeladG{ ({\dir})}
    +
    \alpha \energystabdG{ ({\dir})} 
    .&
\end{align*}
We collect the second main result of this work that will be proven in this section and Section \ref{sec:convergence-to-envar-sols-dg}.
\begin{thm}[Convergence of a subsequence to energy-variational solutions]\label{thm:main-result-subconvergence-dg}
    Let $\Omega \subset \mathbb{R}^{\dimension}$ for $\dimension=2,3$ be a bounded convex polyhedral domain.
    Let $\left(\fat v_0,\fat d_0\right)
     \in \V \times  C^{2}(\bar{\Omega}) $, with $\vert \fat d_0  \vert = 1$ a.e.~in $\Omega$.
     Let $(k_m, h_m)_m $, $m\in \mathbb{N}$ be a decreasing sequence of our positive temporal and spatial discretization parameters converging to zero as $m \to \infty$.
    We assume this sequence fulfills:
    \begin{enumerate}
        \item There exists a \textit{non-obtuse} and admissible subdivision $\mathcal{T}_{h_m}$ of $\Omega$ (in the sense of Definitions \ref{def:non-obtuse} and \ref{def:admissible-mesh}) for each $m \in \mathbb{N}$,
        \item $ k_m \in o (h_m^{1+\dimension /6})$.
    \end{enumerate}
    Then, there exists a subsequence that we do not relabel $(k_m, h_m)_m$ such that a solution to Algorithm \ref{algo:dg} exists for every $(k_m, h_m)$ and $1\leq j \leq J(k_m)$ and its linear temporal interpolate converges to an \textit{energy-variational solution} in the sense of Definition~\ref{def:envar} as $(k_m, h_m)\to 0$ and as $\alpha \to 0$. 
\end{thm}
\begin{rem}[Regularity of Initial Conditions]
The higher regularity assumptions for the initial director field in comparison to Theorem \ref{thm:main-result-subconvergence-cg} are needed for the consistent approximation of the gradient, cf. Lemma \ref{lem:appendix-strong-consistency-gradient} and Section \ref{sec:ics-bcs-convergence-dg}.
\end{rem}
\subsection[Well-Posedness]{Discrete Well-Posedness}
As in Section \ref{sec:discrete-system-cg}, we can obtain an orthogonality relationship, but this time cellwise, by testing Equation \eqref{eq:projection_scheme-dg_c} with $c = \dir_h^{j-1} \vert_{T} \chi_T$ for some $T\in\mesh$. The result is captured by the following corollary.
\begin{cor}[Orthogonality]\label{cor:orthogonality-dg}
    Let $(v^j,  \tilde{d}^{j} )$ be a solution to Algorithm \ref{algo:dg}. Then, the following cellwise orthogonality relation holds:
    \begin{align*}
        (\tilde{d}^j - d^{j-1} ) \vert_T \cdot d^{j-1} \vert_T & = 0 
        & \forall T \in \mesh
        .
    \end{align*}
\end{cor}
In particular this yields
\begin{equation}
\abs{\tilde{d}^j \vert_T }\geq 1 \quad \forall T \in \mesh.
\end{equation}
\begin{lem}
    \label{lem:energy-decreasing-projection-dg}
    Let $\mesh$ be non-obtuse in the sense of Definition \ref{def:non-obtuse} and $d^h\in [\dgzero]^{\dimension}$ such that $\abs{d^h\vert_T}^2 \geq 1$ for all $T\in \mesh$. 
    Then, the cellwise normalization is energy-decreasing, that is
    $$
    \ltwonorm{ \adlifting{ \interpol^0 \frac{d^h}{\abs{d^h}}}{\dbc} }^2
    \leq
    \ltwonorm{ \adlifting{d^h}{\dbc} }^2
    .
    $$
\end{lem}
\begin{proof}
The proof in its essence works analogously to Lemma \ref{lem:energy-decreasing-projection-cg}. For simplification we first restrict ourselves to an arbitrary cell $T\in\mesh$.
Applying the identity $a\cdot b  = \frac{1}{2}a^2 + \frac{1}{2} b^2 - \frac{1}{2} \abs{a-b}^2$ again yields a reformulation of the local norm, i.e.
\begin{multline*}
    \int_T \adlifting{d^h}{\dbc}^2 \diff x
= 
    \sum_{F, \tilde{F}\in\facets^T} \abs{F}{\vert\tilde{F}\vert} (d^h_{F} -d^h_T)\cdot (d^h_{\tilde{F}} -d^h_T)  \normal_{T,F}\cdot \normal_{T,\tilde{F}}
\\= 
    \sum_{F, \tilde{F}\in\facets^T} \abs{F}{\vert\tilde{F}\vert} \normal_{T,F}\cdot \normal_{T,\tilde{F}} 
    \left[
    \frac{1}{2} (d^h_{F} -d^h_T)^2 + \frac{1}{2} (d^h_{\tilde{F}} -d^h_T )^2 
    - \frac{1}{2} \abs{d^h_F - d^h_{\tilde{F}}}^2
    \right]
    .
\end{multline*}
The first two summands depend only on one of the sum indices. They vanish as a consequence of Proposition \ref{prop:appendix-divergence-theorem}.
We end up with the reformulation of the squared and integrated reconstructed gradient
\begin{equation*}
    \int_T \adlifting{d^h}{\dbc}^2 \diff x
= - \frac{1}{2} 
    \sum_{F, \tilde{F}\in\facets^T} \abs{F}{\vert\tilde{F}\vert} \normal_{T,F}\cdot \normal_{T,\tilde{F}} 
    \abs{d^h_F - d^h_{\tilde{F}}}^2
    .
\end{equation*}
Since we assumed the mesh $\mesh$ to be non-obtuse, see Definition \ref{def:non-obtuse}, all summands contribute positively to the sum and therefore we can estimate each one individually by below.
Referring to Equation \eqref{eq:def-v_F}, we evaluate
\begin{equation*}
    \abs{d^h_F - d^h_{\tilde{F}}}=
    \begin{cases}
        \frac{1}{2}
        \abs{ d^h_{T(F)} - d^h_{T(\tilde{F})} }, &
        \text{if } F, \tilde{F} \in \facets^i,
        \\
        \frac{1}{2}
        \abs{\dbc\vert_{{F}} - \dbc \vert_{\tilde{F}}}, &
        \text{if } F, \tilde{F} \in \facets^D,
        \\
        \frac{1}{2}
        \abs{d^h_{T(F)} - \dbc \vert_{\tilde{F}} }, &
        \text{if } F \in \facets^i , \tilde{F}\in\facets^D,
    \end{cases}
\end{equation*}
where $ T(F) \in\mesh$ is the cell neighbouring $T$ that shares facet $F\in\facets$.
Since $\dbc$ fulfills the unit-norm constraint by definition and $\vert d^h_T \vert \geq 1$ by assumption, we can argue as in Lemma \ref{lem:energy-decreasing-projection-cg} to get
\begin{equation}
    \abs{\left[ \interpol^0   \frac{d^h}{\vert d^h \vert} \right]_F - \left[ \interpol^0   \frac{d^h}{\vert d^h \vert} \right]_{\tilde{F}}}^2
    \leq
    \abs{d^h_F - d^h_{\tilde{F}}}^2
    .
\label{eq:convex-estimate-between-facets}
\end{equation}
Summing up over all $T\in \mesh$ delivers the result.
\end{proof}
Reusing the Inequality \eqref{eq:convex-estimate-between-facets} allows us to infer the same for the stabilizing jump terms, which leads to the following Corollary.
\begin{cor}
    \label{cor:energy-decreasing-projection-dg-stab}
    Let $\mesh$ be non-obtuse in the sense of Definition \ref{def:non-obtuse} and $d^h\in [\dgzero]^{\dimension}$ such that $\abs{d^h\vert_T}^2 \geq 1$ for all $T\in \mesh$. 
    Then, the cellwise normalization decreases $\energystabdG$, that is
    $$
    \energystabdG \left ( \interpol^0 \frac{d^h}{\abs{d^h}} \right ) 
    \leq
    \energystabdG \left ( d^h \right )
    .
    $$
\end{cor}
\begin{lem}[Unconditional existence]\label{lem:unconditional-existence-projection-dg}
    Let $k,h>0$ and $j \in \mathbb{N}$, such that $1\leq j \leq J = \lfloor T/k \rfloor$ and $u^{j-1}=(v^{j-1},d^{j-1})\in \mixedfemspacedg$.We assume the regularization parameter of the discrete Laplacian (see Definition \ref{def:discrete-laplacian-dg}) to be positive, $\alpha > 0$. Then, there exists a unique solution $u^j=(v^j,\tilde{d}^j )\in \mixedfemspacedg$ solving Algorithm \eqref{algo:dg}.
\end{lem}
Before we start the proof, we have to decompose the variables into the interior and the prescribed boundary condition again. We do so by
\begin{equation}
    \tilde{d}^j_\circ \coloneqq \tilde{d}^j -\dbc.
\end{equation}
This allows us to split the reconstructed gradient and discrete Laplacian by
\begin{align}
\begin{split}
    \adlifting{\tilde{d}^j_\circ}{0} =& \adlifting{\tilde{d}^j}{\dbc} - \adlifting{\dbc}{\dbc},
    \\
    \Delta_h^0 \tilde{d}^j_\circ =& \Deltah \tilde{d}^j - \Deltah \dbc
     \, ,
\end{split}\label{eqs:decompositions-existence}
\end{align}
where $\Delta_h^0$ is defined as the operator in Equation \eqref{eq:discrete-laplacian-dg} with $\dbc$ replaced by zero.
\begin{proof}
In order to show the claim we apply Lax-Milgram's theorem.
In view of Equations \eqref{eqs:decompositions-existence} we can simply move all dependence on the boundary condition onto the right-hand side.
As in the proof of Lemma \ref{lem:unconditional-existence-projection}, we define the bilinear form $\mathcal{B}$ implicitly by the terms in equation system \eqref{eqs:projection-scheme-dg} depending on $(v^j,\tilde{d}^j_\circ )$ and $(a,c)$.
For the coercivity, we consider 
\begin{multline}
    \label{eq:dg-coercivity-first-step}
    \mathcal{B} ((v^j,- \Deltah\tilde{d}^j_\circ )^T, (v^j,- \Deltah\tilde{d}^j_\circ )^T)
    = 
    \frac{1}{k}
    \norm{v^j}_{L^2(\Omega)}^2
    + \mu \norm{\nabla v^j}_{L^2(\Omega)}^2
    +\frac{\alpha}{k} \abs{\tilde{d}^j_\circ}_J^2
    \\
    +\frac{1}{k}
    \norm{ \adlifting{ \tilde{d}^j_\circ }{0} }_{L^2(\Omega)}^2
    + A \norm{d^{j-1} \times \Deltah\tilde{d}^j_\circ }_{L^2(\Omega)}^2
    .
\end{multline}
Thereby, $\tilde{d}^j_{\circ} = \mathcal{G} (-\Deltah^0 \tilde{d}^j_{\circ})$ has to be understood again in terms of a \textit{Green's function} $\mathcal{G}: [\dgzero]^{\dimension} \to [\dgzero]^{\dimension}$, such that
\begin{multline*}
    (w, \phi)_2 = (\adlifting{ \mathcal{G}(w) }{0}, \adlifting{ \phi }{0})_2
            \\
            + \sum_{F \in \facets^i} \frac{\alpha}{h_F} \int_F \jump{\mathcal{G}(w)}:\jump{\phi} \diff s
            + \sum_{F \in \facets^D} \frac{\alpha}{h_F} \int_F \mathcal{G}(w)  : \phi \diff s
            ,
            \quad \forall \phi \in [\dgzero]^{\dimension}.
\end{multline*}
Note that the assumption $\alpha>0$ is crucial for the well-posedness of the map $\mathcal{G}$ since $\ltwonorm{\adlifting{ . }{0}}$ is not a norm on $[\dgzero]^{\dimension}$ lacking the positive definiteness.
Further the Green's function allows for the estimate
\begin{equation}
\ltwonorm{\adlifting{ \mathcal{G}(w) }{0}} +\alpha \abs{\mathcal{G}(w)}_J
\geq
C^{-1} h
\ltwonorm{w}
\label{eq:discrete-norm-equivalence-dg}
\end{equation}
since
\begin{align*}
    \ltwonorm{w}^2 
    &
    \leq
    \ltwonorm{\adlifting{ \mathcal{G}(w) }{0}}
    \ltwonorm{\adlifting{ w }{0}}
    +
    \alpha \abs{\mathcal{G}(w)}_J \abs{w}_J
    \\
    &
    \leq
    C
    \left (
        \ltwonorm{\adlifting{ \mathcal{G}(w) }{0}} 
    +\alpha \abs{\mathcal{G}(w)}_J 
    \right ) 
        \abs{w}_J
    \\&\leq
    C
    \left (
    \ltwonorm{\adlifting{ \mathcal{G}(w) }{0}} 
    +\alpha \abs{\mathcal{G}(w)}_J 
    \right ) 
    h^{- 1}
    \ltwonorm{w}
     \, ,
\end{align*}
where we used the estimate $\abs{.}_J \leq C h^{-1} \ltwonorm{.}$ for all functions in $[\dgzero]^{\dimension}$. Applying Inequality \eqref{eq:discrete-norm-equivalence-dg} onto Equation \eqref{eq:dg-coercivity-first-step} then yields the coercivity.
\end{proof}
\subsection{\textit{A priori} Estimates}
\label{sec:a-priori-dg}
\begin{lem}[\textit{A priori} estimates]
Let the assumptions of Theorem \ref{thm:main-result-subconvergence-dg} be fulfilled and let $u^j=(v^j,\tilde{d}^j)\in \mixedfemspacedg$ be a solution of Algorithm \eqref{algo:dg} for all $1\leq j\leq n \leq J$. Then, the following discrete Energy-Dissipation inequality holds
\begin{align}
\begin{split}\label{eq:discrete-energy-inequality-projection-dg}
&
\energykin{(v^n)}   + \energyeladG{ (\tilde{d}^n)} + \alpha \energystabdG (\tilde{d}^n )
\\&
+ k \sum_{j=1}^n 
    \left[ \mu \lebnorm{\nabla v^{j}}{2}^2 
    + 
     \norm{d^{j-1} \times \Deltah \tilde{d}^{j} }^2_2 
    \right]
    \\& 
    + \frac{1}{2}\sum_{j=1}^n 
    \left[ 
        \ltwonorm{\adlifting{\tilde{d}^j -d^{j-1}}{0}}^2
        + \alpha  
        \abs{ \tilde{d}^j -d^{j-1}}_J^2
        + k^2 \ltwonorm{ \discretedt {v}^j}^2
    \right]
    \\& \quad\quad\quad\quad\quad\quad\quad\quad\quad\quad
    \leq
    \energykin{(v^0)}  + \energyeladG{ ({d}^0)} + \alpha \energystabdG (d^0)
\end{split}
\end{align}
as well as
\begin{equation*}
    \energyeladG{ ({d}^n)} + \alpha \energystabdG ({d}^n )
    \leq
    \energyeladG{ (\tilde{d}^n)} + \alpha \energystabdG (\tilde{d}^n )
    .
\end{equation*}
The projected solution fulfills the unit-sphere constraint elementwise, i.e.
\begin{equation}\label{eq:elementwise-unit-norm-constraint-for-projected-d}
    \abs{d^n \lvert_K }=1
\end{equation}
for all $K\in\mesh$.
Additionally, for some generic constant $C>0$, we obtain that
\begin{align}
\begin{split}
k \sum_{l=1}^n \norm{k^{-1} (\tilde{d}^j - d^{j-1})}_{L^{3/2}(\Omega)}^2
& \leq 
C 
,
\\
k \sum_{l=1}^n \norm{k^{-1} ({d}^j - d^{j-1})}_{L^{3/2}(\Omega)}^2
& \leq 
C 
,
\\
k \sum_{l=1}^n  \norm{\discretedt v^j}_{(\velspace )^*}^2
& \leq 
C 
,
\label{bounds_temp_variation_inequ-dg}
\end{split}
\end{align}
where the generic constant on the right-hand side may depend on $\mathcal{E}^h(v^0,d^0)$.
\label{lem:a-priori-estimates-projection-dg}
\end{lem}
Note that the right side of Inequality \eqref{eq:discrete-energy-inequality-projection-dg} is bounded uniformily in $h$ since $\fat{d}_0 \in C^2(\bar{\Omega})$, such that the discrete gradient converges strongly in $L^2(\Omega)$, see Lemma \ref{lem:appendix-strong-consistency-gradient}.
For the jump terms we further note that
\begin{equation*}
    \abs{\interpol^0 \fat{d}_0}_J \leq C \norm{\fat{d}_0}_{C^1(\bar{\Omega})}
    .
\end{equation*}
\begin{proof}
The elementwise unit-norm constraint \eqref{eq:elementwise-unit-norm-constraint-for-projected-d} is fulfilled trivially by the projection step \eqref{eq:projection-step-dg}.
The rest follows in essence as in Lemma \ref{lem:a-priori-estimates-projection}, this time applying Lemma \ref{lem:energy-decreasing-projection-dg} for the projection of the discrete director field.
The computation one has to do with care is the usage of boundary conditions in the discrete gradient and laplacian, i.e.
\begin{multline*}
    \frac{1}{k} (\tilde{d}^j  - d^{j-1}, - \Delta_h \tilde{d}^j)_2
=
    \frac{1}{k} (\adlifting{\tilde{d}^j}{\dbc}, \adlifting{\tilde{d}^j  - d^{j-1}}{0})_2
\\
    +
    \frac{\alpha}{k} \sum_{F\in\facets^i} \frac{1}{h_F} \int_F \jump{\tilde{d}^j} \cdot \jump{\tilde{d}^j  - d^{j-1}}
    +
    \frac{\alpha}{k} \sum_{F\in\facets^D} \frac{1}{h_F} \int_F \jump{\tilde{d}^j - \dbc} \cdot \jump{\tilde{d}^j  - d^{j-1}}
    .
\end{multline*}
Applying the identity $a\cdot b = \frac{1}{2} a^2 + \frac{1}{2} b^2 - \frac{1}{2} \abs{a-b}^2$ on the first summand of the right hand side yields
\begin{multline*}
    \frac{1}{k} (\adlifting{\tilde{d}^j}{\dbc}, \adlifting{\tilde{d}^j  - d^{j-1}}{0})
=
    \frac{1}{2k} \ltwonorm{\adlifting{\tilde{d}^j}{\dbc}}^2
\\
    -
    \frac{1}{2k} \ltwonorm{ \adlifting{d^{j-1}}{\dbc}}^2
    +
    \frac{1}{2k} \ltwonorm{\adlifting{\tilde{d}^j  - d^{j-1}}{0}}^2
    \, ,
\end{multline*}
where we used that
\begin{align*}
 \adlifting{\tilde{d}^j  - d^{j-1}}{0}
=&
    \adlifting{\tilde{d}^j  }{\dbc}
    - \adlifting{d^{j-1}}{\dbc}.
\end{align*}
For the second and third term, we proceed in the same fashion, reapplying the $a\cdot b = \frac{1}{2} a^2 + \frac{1}{2} b^2 - \frac{1}{2} \abs{a-b}^2$, to obtain
\begin{multline*}
    \sum_{F\in\facets^i} \frac{1}{h_F} \int_F \jump{\tilde{d}^j} \cdot \jump{\tilde{d}^j  - d^{j-1}}
    +
     \sum_{F\in\facets^D} \frac{1}{h_F} \int_F \jump{\tilde{d}^j  - \dbc} \cdot \jump{\tilde{d}^j - d^{j-1}}
\\=
 \sum_{F\in\facets^i} \frac{1}{2 h_F} (\jump{\tilde{d}^j}^2 - \jump{d^{j-1}}^2 + \jump{\tilde{d}^j- d^{j-1}}^2)
\\+
 \sum_{F\in\facets^D} \frac{1}{2 h_F} (\jump{\tilde{d}^j -\dbc}^2 - \jump{d^{j-1}-\dbc}^2 + \jump{\tilde{d}^j- d^{j-1}}^2).
\end{multline*}
For the remaining \textit{a priori} estimates, we reiterate the $L^p (\Omega)$ duality arguments from Lemma \ref{lem:a-priori-estimates-projection}. For the director field that is
\begin{multline*}
    \lebnorm{k^{-1} (\tilde{d}^j - d^{j-1})}{3/2}
= 
    \sup_{\phi \in \L{3}}
    \abs{
    \frac{(k^{-1} (\tilde{d}^j - d^{j-1}), \projectdg \phi)_2}{\lebnorm{\phi}{3}}
    }
\\\lesssim
    \lebnorm{\dirh^{j-1}}{\infty}^2
    \lebnorm{\adlifting{\dirh^{j-1}}{\dbc}}{2} \lebnorm{ v^j}{6} 
\\
    +
    \lebnorm{\dirh^{j-1}}{\infty}
    \lebnorm{d^{j-1} \times \Delta_h \tilde{d}^j}{2}
     \, ,
\end{multline*}
where we used the $\L{p}$ stability of the projection $\projectdg$.
The estimation of the projected discrete time derivative $\discreteDiff_t d^j $ follows immediately as in Lemma \ref{lem:a-priori-estimates-projection} due to the function being piecewise constant.
The same holds for the estimate of the discrete time derivative of the velocity.
\end{proof}
\subsection{Further Properties}
\begin{prop}[Discrete energy inequality]
Let $u^j = (v^j,d^j)\in \mixedfemspacedg$ be a solution to Algorithm~\eqref{algo:dg} for $1\leq j \leq J$. Then, the discrete energy-variational inequality
\begin{align}
\begin{split}\label{eq:discrete_envar-dg}
    & \discreteDiff_t E^j       
        + \mu \left ( \nabla v^j , \nabla v^j-  \nabla \projectV\vv \right )  
        +  \ltwonorm{d^{j-1} \times \Deltah \tilde{d}^{j} }^2 
\\&
        - (\discreteDiff_t v^j,\projectV\vv) - \left ( (v ^{j-1}\cdot \nabla)  v ^j , \projectV\vv \right ) - \frac{1}{2} \left ( ( \nabla \cdot  v ^{j-1} )v^j , \projectV\vv\right )
\\ &
        - T^E_h (d^{j-1}, \projectV \vv,\Deltah \tilde{d}^{j})
\\&
        + \mathcal{K}(\projectV\vv) \left ( \frac{1}{2}\Vert v^{j-1}\Vert_{L^2(\Omega)}^2 + \frac{1}{2}\Vert \adlifting{d^{j-1}}{\dbc} \Vert_{L^2(\Omega)}^2  -E^{j-1} \right ) 
    \leq  0
\end{split}
\end{align}
holds for all $\vv\in \velspace$
with the energy $E^j$ and the  regularity weight $\mathcal{K}$ given by 
\begin{align*}
    E^j := \frac{1}{2}\ltwonorm{v^j}^2 + \frac{1}{2}\ltwonorm{
    \adlifting{d^j}{\dbc} }^2 
    +\alpha \energystabdG (d^j)
    \qquad \mathcal{K}(\vv ) : ={}& \frac{1}{2}
\Vert  \vv \Vert_{L^{\infty} (\Omega)} ^2 \, 
.
\end{align*} 
\label{prop:discrete-envar-dg}
\end{prop}
\begin{proof}
    The proof works completely analogously to the one of Proposition \ref{prop:discrete-envar-cg}, this time adding $ \frac{1}{2}\ltwonorm{v^{j-1} }^2 + \frac{1}{2}\ltwonorm{
        \adlifting{d^{j-1} }{\dbc} }^2 - E^{j-1} \leq 0$.
\end{proof}
As in Section \ref{sec:discrete-system-cg}, we define the projection error by
\begin{align}\label{eq:def-projection-error-dg}
r^j \coloneqq d^j - \tilde{d}^j= \interpol^0 \left (\frac{\tilde{d}^j}{ \abs{\tilde{d}^j} } \right ) - \tilde{d}^j \, ,
\end{align}
for which we obtain analogous asymptotic behaviour.
\begin{lem}\label{lem:convergence-projection-err-dg}
Let $(v^j,\tilde{d}^j)\in \mixedfemspacedg$ be a solution to Algorithm \eqref{algo:dg}. Then, the error estimate
\begin{align*}
\lebnorm{\interpol^0 \left (\frac{\tilde{d}^j}{ \abs{\tilde{d}^j} } \right ) - \tilde{d}^j}{1}
=
\lebnorm{r^j}{1} 
\leq
C k^2 \ltwonorm{k^{-1}(\tilde{d}^j-d^{j-1})}^2
\end{align*}
holds.
\end{lem}
\begin{proof}
Since $r^j \in [\dgzero]^{\dimension}$ we can simply decompose the integral by
$$
    \lebnorm{r^j}{1}
=
    \sum_{T\in\mesh}  \abs{r^j \vert_T} \abs{T}.
$$
Applying the cellwise orthogonality, that is Corollary \ref{cor:orthogonality-dg}, and reiterating Equations \eqref{eq:projection-error-nodal-decomposition} and \eqref{eq:application-taylor-sqrt}, we obtain
$$
\abs{r^j \vert_T}
\leq
C k^2 \abs{k^{-1}(\tilde{d}^j - d^{j-1}) \vert_T }^2.
$$
Plugging that back into the summation over the cells yields the result.
\end{proof}
Lastly, we collect a result that is later needed to identify the limit of the discrete Laplacian in the convergence analysis.
\begin{lem}[Product rule for the discrete gradient]
    \label{lem:discrete-product-rule-grad-cross-product}
Let $d \in [\dgzero]^{\dimension}$ and let $\phi \in C^{\infty}_c (\Omega;\mathbb{R}^{\dimension})$ arbitrary with $\phi_h \coloneqq \interpol^0 \phi \in [\dgzero]^{\dimension}$. Then, the inequality
\begin{multline}
    \left (
    d \times \adlifting{d}{\dbc}, \adlifting{ \phi_h }{0}
    \right )_2 
    +
    \left (
    \adlifting{ d }{\dbc}, \adlifting{ d \times \phi_h }{0}
    \right )_2 
\\
    \lesssim
    h
    \norm{\nabla \phi}_\infty
    \sqrt{\energystabdG(d)}
    \ltwonorm{\adlifting{d}{\dbc}}
\label{eq:discrete-product-rule-grad-cross-product}
\end{multline}
holds.
\end{lem}
\begin{proof}
    First of all, we note that the following term vanishes
    $$
    \left (
    \phi \times \adlifting{ d }{\dbc},  
    \adlifting{ d }{\dbc}
    \right )_2 
    =0 ,
    $$
    as this is a consequence of the cellwise algebraic computation
    $$
    (w \times M) : M = -M: (w \times M) = - (w \times M) : M
    $$
    for matrices $M$ and a vector $w$.
    Using Equations \eqref{eq:triple-prodcut}, \eqref{eq:jump-identities} and \eqref{eq:def-discrete-lifting-local} implies
    \begin{align}
    \begin{split}
        0 =&-
            \left (
            \phi \times \adlifting{d}{\dbc}, \adlifting{d}{\dbc}
            \right )_2 
        \\ =&
            \sum_{F\in \facets^i } 
            \int_F 
                \jump{d} \cdot \avg{\phi \times \adlifting{d}{\dbc} }   \normal
            \diff s 
            \\&
            +
            \sum_{F\in \facets^D} 
            \int_F 
            (d-\dbc) \cdot
                (\phi \times \adlifting{d}{\dbc} \normal) 
            \diff s
        \\ =&
            \sum_{F\in \facets^i} \int_F 
            \jump{d} 
            \cdot
            \left (
            \avg{\phi} \times \avg{ \adlifting{d}{\dbc}} \normal
            + \frac{1}{4} \jump{\phi} \times \jump{ \adlifting{d}{\dbc}}
            \normal
            \right )
            \diff s
            \\&
            +
            \sum_{F\in \facets^D} 
            \int_F
                (d-\dbc) \cdot ( \phi \times \adlifting{d}{\dbc}  \normal  )
            \diff s
        \\ =&
            \sum_{F\in \facets^i} \int_F 
            \left (
                \jump{d} \times \avg{\phi} \cdot  \avg{ \adlifting{d}{\dbc} } \normal
            +\frac{1}{4} \jump{d} \times \jump{\phi} \cdot  \jump{ \adlifting{d}{\dbc} \cdot \normal} 
            \right )
            \diff s
            \\&
            +
            \sum_{F\in \facets^D} \int_F            
            ( (d-\dbc) \times\phi )  \cdot \adlifting{d}{\dbc}  \normal
            \diff s
        \\ \eqcolon&
            A +\frac{1}{4}J + B(d-\dbc)
            .
    \label{eq:dg0-constructive-zero}   
    \end{split}
    \end{align}
    Now, we can start working on the first term in Inequality \eqref{eq:discrete-product-rule-grad-cross-product} by observing
    \begin{multline*}
       - \left (
            d \times \adlifting{d}{\dbc}, \adlifting{ \phi_h}{0} 
        \right )_2    
    \\=
        \sum_{F\in \facets^i } \int_F \jump{\phi_h } \cdot \avg{d \times \adlifting{d}{\dbc}} \normal 
        \diff s
        +
        \sum_{F\in \facets^D} \int_F
        \phi_h 
        \cdot ( d \times \adlifting{d}{\dbc} \normal )
        \diff s
    \\=
        \sum_{F\in \facets^i } \int_F
            \jump{\phi_h } \cdot  (\avg{d} \times \avg{\adlifting{d}{\dbc} } \normal )
        + \frac{1}{4}
            \jump{\phi_h }  \cdot ( \jump{d} \times \jump{\adlifting{d}{\dbc} \normal} )
        \diff s
        \\
        +
        \sum_{F\in \facets^D} \int_F
            \phi_h \cdot (d \times \adlifting{d}{\dbc} \normal )
        \diff s
    \\=
        \sum_{F\in \facets^i } \int_F \jump{\phi_h }\times  \avg{d} \cdot  \avg{\adlifting{d}{\dbc} } \normal
        + \frac{1}{4}
          \jump{\phi_h } \times \jump{d} \cdot  \jump{\adlifting{d}{\dbc} \normal} 
         \diff s
         \\
        +
        \sum_{F\in \facets^D} \int_F
            (\phi_h \times d) \cdot   \adlifting{d}{\dbc} \normal 
        \diff s
    \\=
        - \sum_{F\in \facets^i } \int_F 
            (\avg{d} \times  \jump{\phi_h })  \cdot  \avg{\adlifting{d}{\dbc}} \normal 
        + \frac{1}{4}
            (\jump{d} \times \jump{\phi_h })  \cdot  \jump{\adlifting{d}{\dbc}\normal} 
        \diff s
        \\
        -
        \sum_{F\in \facets^D} \int_F
            (d \times \phi_h) \cdot   \adlifting{d}{\dbc} \normal 
        \diff s
    \\ \eqcolon
        -C - \frac{1}{4}J - B(d)
        .
    \end{multline*}
    Now, we decompose the second term of \eqref{eq:discrete-product-rule-grad-cross-product} to
    \begin{multline*}
        -\left (
            \adlifting{d}{\dbc}, \adlifting{ d \times \phi_h }{0}
        \right ) 
    \\=
        \sum_{F\in \facets^i} \int_F 
        \jump{ d \times \phi_h }
        \cdot
        \avg{\adlifting{d}{\dbc} } \normal  
        \diff s
        +
        \sum_{F\in \facets^D} \int_F 
        (d \times \phi_h )
        \cdot
        \adlifting{d}{\dbc}  \normal    
        \diff s
    \\=
        \sum_{F\in \facets^i} 
        \int_F 
            (\avg{ d} \times \jump{\phi_h })
            \cdot \avg{\adlifting{d}{\dbc} } \normal
        +
            (\jump{ d} \times \avg{\phi_h } )
            \cdot \avg{\adlifting{d}{d}} \normal
            \diff s
        \\
        +
        \sum_{F\in \facets^D} \int_F 
        (d \times \phi_h )
        \cdot
        \adlifting{d}{\dbc} \normal    
        \diff s
    =
        C + A + B(d)
        .
    \end{multline*}
    Adding the two,  we end up with
    \begin{align*}
        \left (
            d \times \adlifting{d}{\dbc}, \adlifting{\phi_h }{0}
        \right )  
        &+  
        \left (
            \adlifting{d}{\dbc}, \adlifting{ d \times \phi_h }{0}
        \right ) 
    \\ &=
        +C + \frac{1}{4}J + B(d) 
        - C - A - B(d)
    \\ &=
          -A + \frac{1}{4}J
     =
        \frac{1}{2}J + B(d-\dbc)
        ,
    \end{align*}
    where we made use of Equation \eqref{eq:dg0-constructive-zero} in the last step. 
    However, since $\abs{F} \propto h^{\dimension -1}$ we can estimate the terms denoted by $J$ by
    \begin{multline*}
        J
    =
        \sum_{F\in \facets^i} \abs{F} 
        (\jump{d} \times  \jump{\phi_h }) \cdot  \jump{\adlifting{d}{\dbc} } \normal
    \lesssim
    h
    \norm{\nabla \phi}_\infty
            \abs{d}_{J,i}
            \norm{\adlifting{d}{\dbc} }_{L^2(\Omega)}
            .
    \end{multline*}
    Thereby, we used that the amount of facets and neighbours of each element is bounded and that
    $$
    \abs{\jump{\phi_h }_F}
    =
    \abs{\jump{\interpol^0 \phi }_F}
    =
    \abs{\phi (x_{T_1}) - \phi (x_{T_2})}
    \lesssim
    h
    \norm{\nabla \phi}_\infty
    ,
    $$
    where $x_{T_1},x_{T_2}$ are the cell centers of the cells sharing facet $F$.
    The boundary term $B$ can be estimated by 
    \begin{multline*}
        B(d - \dbc) = 
        \sum_{F\in \facets^D} \int_F 
        ( [d - \dbc] \times [\phi_h - 0] )
        \cdot
        \adlifting{d}{\dbc} \normal    
        \diff S
    \\
    \lesssim 
    h
    \norm{\nabla \phi}_\infty
    \abs{d - \dbc}_{J,D}
    \ltwonorm{\adlifting{d}{\dbc}}
    ,
    \end{multline*}
    where we used that the boundary trace of $\phi$ is zero.
    Summing up yields the result.
\end{proof}
\section{Convergence of Algorithm \ref{algo:dg}}
\label{sec:convergence-to-envar-sols-dg}
The interpolation of functions in time and space follows the same notation as in Section \ref{sec:interpolation-in-time-cg}, where the interpolation operator applied on smooth functions is adapted to the space $\dgzero$, i.e.~$\interpol^1$ is replaced by $\interpol^0$.
\subsection{Converging Subsequences}
We continue as in Section \ref{sec:subsequences-cg}. The analogous \textit{a priori} estimates are inferred by the bounds in Lemma \ref{lem:a-priori-estimates-projection-dg}.
Note that everything is indexed by the stabilization parameter $\alpha>0$, which so far has been chosen arbitrarily but fixed.
Regarding the approximation of the velocity field nothing has changed in comparison to the previous algorithm. In spite of the discontinuous Galerkin approach for the director equation, the following quantities are still conforming approximations to their respective function spaces, such that we can infer subsequences that we do not relabel, i.e.~ 
\begin{align}
    \begin{split}
    \label{eqs:converging_subsequences-dg-alpha}
    \overline{v}^k_h,\underline{v}^k_h, v_h^k &\weakstarto {\fat v}_{\alpha} \text{ in } L^{\infty} (0,T;\Ha), 
    \\
    \overline{v}^k_h,\underline{v}^k_h,v_h^k &\weakto \fat{v}_{\alpha} \text{ in }  L^2(0,T;H^1_0(\Omega)), 
    \\
    \partial_t v_k^h &\weakto \partial_t \fat{v}_{\alpha} \text{ in } L^2(0,T; (\velspace )^*) ,
    \\
    \overline{v}^k_h,\underline{v}^k_h,v_h^k & \to \fat{v}_{\alpha}  \text{ in }  L^p(0,T;\Ha)  \text{ for all } p\in[2,\infty),
    \\
    \overline{d}^k_h,\underline{d}^k_h,d_h^k 
    &\weakstarto
    \fat{d}_{\alpha} \text{ in } 
    L^{\infty} ( (0,T) \times \Omega)
    ,\\
    w^k_h &\weakto \fat{ W}_{\alpha}  \text{ in } L^2(0,T;L^{3/2}(\Omega)),
    \\
    \partial_t d^k_h &\weakto \partial_t \fat{ d}_{\alpha}  \text{ in } L^2(0,T;L^{3/2}(\Omega)),
    \\
     \underline{d}^k_h \times \Deltah \overline{\tilde{d}}^k_h
    &\weakto \fat{ U}_{\alpha}  \text{ in } L^2(0,T;\L{2}), 
    \\
    \underline{E}^k_h(t) & \to E_{\alpha} (t) \text{ for all } t \in [0,T]
    \,
    ,
    \end{split}
\end{align}
where $w^k_h$ is again defined by Equation \eqref{eq:def-wkh}.
Next, we consider the non-conforming approximation of the initial gradient, which we from now on denote by 
\begin{align*}
\overline{G}^k_h \coloneqq \adlifting{\overline{d}^k_h}{\dbc},
\quad
\underline{G}^k_h \coloneqq \adlifting{\underline{d}^k_h}{\dbc},
\quad
G^k_h \coloneqq \adlifting{d^k_h}{\dbc},
\\
\overline{\tilde{G}}^k_h \coloneqq \adlifting{\overline{\tilde{d}}^k_h}{\dbc},
\quad
\underline{\tilde{G}}^k_h \coloneqq \adlifting{\underline{\tilde{d}}^k_h}{\dbc},
\quad
\tilde{G}^k_h \coloneqq \adlifting{\tilde{d}^k_h}{\dbc} \, ,
\end{align*}
for better readability. The theorem of Banach--Alaoglu--Bourbaki again allows us to derive converging subsequences, i.e.
\begin{equation}
    \overline{G}^k_h, \underline{G}^k_h,G^k_h \weakstarto \fat{G}_{\alpha} \text{ in } L^{\infty} (0,T;L^{2} (\Omega))
    .
    \label{eq:convergence-dg-discrete-grad}
\end{equation}
In order to identify $\fat{G}_{\alpha}$, we apply Lemma \ref{lem:weak-consistency-gradient} and obtain
$$
\fat{G}_{\alpha} = \nabla \fat{d}_{\alpha}.
$$
We can repeat this identification process for the other temporal interpolates as well as for $\fat{\tilde{G}}_{\alpha} = \nabla \fat{\tilde{d}}_{\alpha}$.
Lemma \ref{lem:dg-aubin-lions-simon} allows us to infer even strong convergence of the director field,
$$
\overline{d}^k_h,\underline{d}^k_h,d_h^k  \to \fat{d}_{\alpha} 
\text{ in }
L^2(0,T;\L{2}) .
$$
Note that the strong convergence holds indeed for all temporal interpolates since their respective difference vanishes as in Equation \eqref{eq:l1l1limit-temporal-interpolates}. 
As prior, a Riesz-Thorin interpolation argument allows us to obtain the strong convergence
\begin{equation}\label{eq:director_strong_convergence-dg}
    \overline{d}^k_h,\underline{d}^k_h,d_k^h
    \to \fat d \text{ in } L^p (0,T;L^p(\Omega)) \text{ for all }p\in [1,\infty).
\end{equation}
Applying the Aubin--Lions--Simon lemma, cf. Equation \eqref{d_lions_aubin}, and the Lions--Magenes lemma \cite[Lemma II.5.9]{boyer-fabrie13} allow us now to derive that the limit function $\fat{d}_{\alpha}$ is even weakly continuous, i.e.~$\fat{d}_{\alpha} \in C_w (0,T;\H{1})$.
\\
Regarding the projection error, we observe that its behaviour follows immediately from Lemma \ref{lem:convergence-projection-err-dg} and the inverse estimate. For completeness, we collect the result in the following corollary.
\begin{cor}
\label{cor:identification-rkh-dg}
Under the previous assumptions, we have
$$
\norm{\overline{r}^k_h}_{L^1(0,T;L^1(\Omega))}
\lesssim 
\frac{k^2}{h^{\dimension /3}}
\norm{w^k_h}_{L^2(0,T;L^{3/2}(\Omega))}^2
.
$$
This implies
\begin{align}
    \overline{r}^k_h 
    &\to 0 \in L^1(0,T;L^1(\Omega)) \text{ for } k\in o(h^{\dimension /6})
    ,
    \label{eq:convergence-rkh-dg}
    \\
    \frac{1}{k} \overline{r}^k_h 
    &\to 0 \in L^1(0,T;L^1(\Omega)) \text{ for } k\in o(h^{\dimension /3})
    \label{eq:convergence-rkh-over-k-dg}
\end{align}
as $k,h\to 0$.
\end{cor}
This implies for $k\in o(h^{\dimension /3})$ that we can identify
$$
\fat{d}_{\alpha} =\tilde{\fat{d}}_{\alpha}, \quad \fat{W}_{\alpha} = \partial_t \fat{d}_{\alpha}
,
$$
see eventually Equation \eqref{eq:decomposition-wkh}. 
\subsection{Convergence of the Sphere Constraint, Energy, Initial and Boundary Conditions}
\label{sec:ics-bcs-convergence-dg}
For the analysis of the velocity field, we refer back to Section \ref{sec:ics-bcs-convergence-cg} since there have been no changes.
The initial discrete gradient of the director field $\adlifting{d^0}{\dbc}$ converges strongly to the gradient of the initial condition $\nabla \fat{d}_0$ due to Lemma \ref{lem:appendix-strong-consistency-gradient}.
On the stabilizing terms $\alpha \energystabdG (\overline{d}^k_h)$ one can again apply  Helly's selection principle \cite[Ex. 8.3]{brezis} to obtain
$$
 \energystabdG (\overline{d}^k_h (t) )  \to E_{J,\alpha} (t) \text{ for all } t \in [0,T].
$$
Due to Lemma \ref{lem:a-priori-estimates-projection-dg} the unit-norm constraint is fulfilled almost everywhere. The fulfillment in the limit then follows from the strong convergence \eqref{eq:director_strong_convergence-dg}. The convergence of the boundary condition of the director field follows from Lemma \ref{lem:weak-consistency-gradient}.
\subsection{Director Equation}
We consider a smooth and compactly supported test function $\phi \in C_c^{\infty} ( [0,T] \times \Omega ; \mathbb{R}^{\dimension} )$.
The tested discrete director Equation \eqref{eq:projection_scheme-dg_c} integrated in time leads us to
\begin{equation*}
\int_0^T
    ( w^k_h , \overline{\phi}^k_h )_2
    + 
    (\underline{{d}}^k_h \times [\underline{{G}^k_h}\overline{v}^k_h] , \underline{{d}}^k_h \times \overline{\phi}^k_h )_2
    -
    (\underline{{d}}^k_h \times \Deltah \overline{\tilde{d}}^k_h,\underline{{d}}^k_h \times \overline{\phi}^k_h)_2
\diff \tau = 0 .
\end{equation*}
The convergence can be derived in standard fashion from the convergences \eqref{eqs:converging_subsequences-dg-alpha}, \eqref{eq:convergence-dg-discrete-grad} and \eqref{eq:director_strong_convergence-dg}. We end up with:
\begin{equation}
\label{eq:director-equation-alpha}
    \int_0^T
        ( \partial_t \fat{d}_{\alpha}, \phi )_2
        + 
        (\fat{d}_{\alpha} \times [(\fat{v}_{\alpha} \cdot \nabla) \fat{d}_{\alpha} ] , \fat{d}_{\alpha} \times \phi )_2
        -
        (\fat{U}_{\alpha}, \fat{d}_{\alpha} \times \phi)_2
    \diff \tau = 0 .
\end{equation}
\subsection{Relative Inequality for the Laplacian}
\label{sec:convergence-dg-laplacian}
We cannot yet identify $\fat{U}_{\alpha}$ as in Section \ref{sec:convergence-to-envar-sols-cg}. However, we can still find a relative inequality depending on $\alpha$. 
\begin{prop}
    \label{cor:identification-of-d-times-deltad}
    Under the previous assumptions and additionally assuming
    $$
    k\in o (h^{1+ \dimension /6}) \, ,
    $$
    the following relative inequality holds
    \begin{equation}
       \int_0^T  (\fat{U}_{\alpha},\phi)_2
        +
        ( \fat{d}_{\alpha} \times \nabla \fat{d}_{\alpha}, \nabla \phi )_2
        \diff \tau
    \lesssim
        \alpha^{1/2} 
        \norm{\nabla \phi}_{L^{2} (0,T; \L{\infty})}
    \label{eq:relative-inequality-laplacian-alpha}
    \end{equation}
    for all $\phi\in \Cont_c^{\infty}([0,T]\times {\Omega})$.
\end{prop}
\begin{proof}
First, we note the following decomposition of our discrete terms:
\begin{align*}
    (\underline{d}^k_h \times \Deltah \overline{\tilde{d}^k_h} , \phi_h)_2 
    +&
    (\underline{d}^k_h  \times \overline{\tilde{G}}^k_h, \adlifting{ \phi_h  }{0} )_2
    \\
    =&
    ([\underline{d}^k_h-\overline{\tilde{d}^k_h}] \times \Deltah \overline{\tilde{d}^k_h} , \phi_h)_2 
    \\
    &+
    (\overline{\tilde{d}^k_h} \times \Deltah \overline{\tilde{d}^k_h} , \phi_h)_2 
    +
    (\overline{\tilde{d}^k_h}  \times \overline{\tilde{G}}^k_h, \adlifting{ \phi_h  }{0} )_2
    \\
    &+
    ([\overline{\tilde{d}^k_h}-\underline{d}^k_h]  \times \overline{\tilde{G}}^k_h, \adlifting{ \phi_h  }{0} )_2
\end{align*}
for the interpolate $\phi_h = \interpol^0 \phi$ of some test function $\phi \in \Cont_c^{\infty}(\Omega;\mathbb{R}^{\dimension})$.
The definition of the discrete Laplacian \eqref{eq:discrete-laplacian-dg} allows to deduce the estimate
\begin{equation*}
    h \ltwonorm{\Deltah f}
    \lesssim
    \ltwonorm{\adlifting{f}{\dbc}}
    + \alpha \sqrt{\energystabdG (f)}
\end{equation*}
for all $f\in [\dgzero]^{\dimension}$, see also Inequality \eqref{eq:discrete-norm-equivalence-dg}. We apply this and obtain
\begin{multline*}
\int^T_0 \lebnorm{(\underline{d}^k_h - \overline{\tilde{d}^k_h}) \times \Deltah \overline{\tilde{d}^k_h}}{1}
\diff \tau
\leq
\int^T_0 
k
\lebnorm{w^k_h }{2} \lebnorm{\Deltah \overline{\tilde{d}^k_h}}{2}
\diff \tau
\\
\lesssim
\frac{k}{h^{\dimension /6}}
\int^T_0 
\lebnorm{w^k_h }{3/2} \lebnorm{\Deltah \overline{\tilde{d}^k_h}}{2}
\diff \tau
\lesssim
(1+\alpha^{1/2})
\frac{k}{h^{1+\dimension /6}}
.
\end{multline*}
Assuming $k\in o (h^{1+\dimension /6})$ is accordingly sufficient for the right-hand side to vanish. Next, we observe the term 
\begin{multline}
    (\overline{\tilde{d}^k_h} \times \Deltah \overline{\tilde{d}^k_h} , \phi_h)_2 
=    
    (-\Deltah \overline{\tilde{d}^k_h} , \overline{\tilde{d}^k_h}  \times \phi_h )_2 
    \\
=
    (\overline{\tilde{G}}^k_h, \adlifting{ \overline{\tilde{d}^k_h}  \times \phi_h  }{0} )_2
    + \sum_{F \in \facets^i} \frac{\alpha }{h_F} \int_F \jump{\overline{\tilde{d}^k_h}  } \cdot \jump{\overline{\tilde{d}^k_h}  \times \phi_h} \diff s
    \\
    +\sum_{F \in \facets^D} \frac{\alpha}{h_F} \int_F (\overline{\tilde{d}^k_h} - \dbc)\cdot (\overline{\tilde{d}^k_h}  \times \phi_h) \diff s
    .
\label{eq:energy-inequality-first-step}
\end{multline}
We can use the product rule for the reconstructed gradient, Lemma \ref{lem:discrete-product-rule-grad-cross-product}, to estimate the first term on the right-hand side by
\begin{equation*}
    (\overline{\tilde{G}}^k_h, \adlifting{ \overline{\tilde{d}^k_h}  \times \phi_h  }{0} )_2
    =
    -
    (\overline{\tilde{d}^k_h}  \times \overline{\tilde{G}}^k_h, \adlifting{ \phi_h  }{0} )_2
    + \frac{1}{\sqrt{\alpha}} \mathcal{O} (h) \lebnorm{\nabla \phi}{\infty}
    .
\end{equation*}
Further, we note that
\begin{multline*}
    ([\overline{\tilde{d}^k_h} -\underline{d}^k_h ]  \times \overline{\tilde{G}}^k_h, \adlifting{ \phi_h  }{0} )_2
    \\
    \lesssim
    k 
    \norm{w_k^h}_{L^2(0,T;\L{2})} 
    \norm{\overline{\tilde{G}}^k_h}_{L^{\infty}(0,T;\L{2})} 
    \norm{\adlifting{ \phi_h  }{0}}_{L^{2}(0,T;\L{\infty})}
    \\
    \lesssim
    \frac{k}{h^{\dimension /6}} 
    \norm{w_k^h}_{L^2(0,T;\L{3/2})} 
    \norm{\overline{\tilde{G}}^k_h}_{L^{\infty}(0,T;\L{2})} 
    \norm{\nabla \phi}_{L^{2}(0,T;\L{\infty})}
    .
\end{multline*}
The remaining terms can be estimated using the identity $\jump{a}\cdot \jump{a\times b} = \jump{a}\cdot (\avg{a} \times \jump{b})$:
\begin{multline*}
    \sum_{F \in \facets^i} \frac{\alpha }{h_F} \int_F \jump{\overline{\tilde{d}^k_h}  } \cdot \jump{\overline{\tilde{d}^k_h}  \times \phi_h} \diff s
    \lesssim
    \alpha
    \abs{\overline{\tilde{d}^k_h}}_{J,i}
    \ltwonorm{\overline{\tilde{d}^k_h}}
    \lebnorm{\nabla \phi}{\infty}
    \\\lesssim
    \alpha^{1/2}
    \ltwonorm{\overline{\tilde{d}^k_h}}
    \lebnorm{\nabla \phi}{\infty}
\end{multline*}
and by
\begin{multline*}
    \sum_{F \in \facets^D} \frac{\alpha}{h_F} \int_F (\overline{\tilde{d}^k_h} - \dbc)\cdot \overline{\tilde{d}^k_h}  \times \phi_h \diff s
    =
    \sum_{F \in \facets^D} \frac{\alpha}{h_F} \int_F (\overline{\tilde{d}^k_h} - \dbc)\cdot \dbc  \times \phi_h \diff s
    \\ \lesssim
    \alpha
    \abs{\overline{\tilde{d}^k_h}- \dbc}_{J,D}
    \lebnorm{\nabla \phi}{\infty}
    \lesssim
    \alpha^{1/2}
    \ltwonorm{\overline{\tilde{d}^k_h}}
    \lebnorm{\nabla \phi}{\infty}
    .
\end{multline*}
Since $\overline{\tilde{d}^k_h} = k w^k_h + \underline{d}^k_h \in L^2 (0,T;L^2(\Omega))$ for $k\in o(h^{\dimension/6})$, adding everything up yields
\begin{multline*}
    \int_0^T (\underline{{d}^k_h} \times \Deltah \overline{\tilde{d}^k_h} , \phi_h)_2
    +
    (\underline{d}^k_h  \times \overline{\tilde{G}}^k_h, \adlifting{ \phi_h  }{0} )_2
    \diff \tau
    \\
    \lesssim
    \left (
        \frac{1}{\sqrt{\alpha}} \mathcal{O} (h)
    +
    (1+\alpha^{1/2})
    \frac{k}{h^{1+\dimension /6}}
    +
    \alpha^{1/2}
    \right )
    \left (
        \int_0^T
    \sobnorm{\phi}{1}{\infty}^2
    \right )^{1/2}
    .
\end{multline*}
Based on the converging subsequences, the time-step restriction $k\in o (h^{1+\dimension /6})$ and Lemma \ref{lem:appendix-strong-consistency-gradient}, we can now take the limit $k,h \to 0$ to obtain the result.
\end{proof}
\subsection{Variational Inequality}
As prior, testing Inequality \eqref{eq:discrete_envar-dg} with the temporal interpolate $\overline{\phi}^k $ of a smooth test function $\phi \in C_c^{\infty} ([0,T])$, $\phi \geq 0$ and integrating in time yields
\begin{align}
\begin{split}
    \label{eq:discrete-variational-inequ-integrated-in-time-dg}
    -\int_0^T   
    \underline{E}^k_h \discreteDiff_t\hat{\phi}^k 
    \diff \tau
        &     
        + \overline{\phi}^k 
        \left[ 
        \mu \left ( \nabla \overline{v}^k_h , \nabla \overline{v}^k_h-  \nabla \projectV \overline{\vv}^k \right )  
        +  \ltwonorm{ \underline{d}^k_h\times \Deltah \overline{\tilde{d}}^k_h }^2  
        \right]
        \diff \tau
        \\
        &+
        \int_0^T  
        \discreteDiff_t \hat{\phi}^k
        ( \underline{v}^k_h, \projectV \underline{\vv}^k)_2
         +
        \overline{\phi}^k 
        ( \underline{v}^k_h,\discreteDiff_t \projectV \overline{\vv}^k)_2
        \diff \tau
        \\&
        -
        \int_0^T 
        \overline{\phi}^k [
            \left ( (\underline{v}^k_h\cdot \nabla)  \overline{v}^k_h , \projectV\overline{\vv}^k \right )_2 
        +
        \frac{1}{2} 
            \left ( ( \nabla \cdot  \underline{v}^k_h )\overline{v}^k_h , \projectV\overline{\vv}^k\right )_2
        ]
        \diff \tau
\\ &
        - \int_0^T   \overline{\phi}^k  \left ( [ \underline{G}^k_h]^T [\underline{d}^k_h\times  ( \underline{d}^k_h \times \Deltah \overline{\tilde{d}}^k_h )], \projectV\overline{\vv}^k \right )_2
        \diff \tau
\\&
        + \int_0^T   \overline{\phi}^k 
        \mathcal{K}(\projectV\vv) \left ( \frac{1}{2}\Vert \underline{v}^k_h\Vert_{L^2(\Omega)}^2 + \frac{1}{2}\Vert \underline{G}^k_h\Vert_{L^2(\Omega)}^2  -\underline{E}^k_h \right ) 
        \diff \tau
  \leq 0
\end{split}
\end{align}
for all $\vv \in \Cont^1([0,T];(\velspace)^*) \cap C^0 ([0,T];\velspace)$.
All terms can be handled as in Section \ref{sec:variational-inequality-cg}, again applying Lemma \ref{lem:convergence-ericksen-stress-in-envar} and \cite[Lemma 2.4]{lasarzik22}.
Then, we end up with the relative energy inequality
\begin{multline}
    \left
    [ 
         E_{\alpha} -  \int_{\Omega} \fat{v}_{\alpha} \cdot\vv  
        \diff x 
     \right 
     ] 
     \Big \vert_{s-}^t 
     \\ + \int_s^t   \int_{\Omega} \fat{v}_{\alpha} \cdot \partial_t \vv 
      -  ( \fat{v}_{\alpha} \cdot \nabla) \fat{v}_{\alpha} \cdot  \vv 
      \diff x \diff \tau  
      + \int_s^t   \int_{\Omega} ( \fat{d}_\alpha \times ( \vv \cdot \nabla ) \fat{d}_\alpha) \cdot ( \fat{U}_{\alpha})  \diff x \diff \tau 
    \\
    +\int_s^t \int_{\Omega} \mu 
    \left (  \nabla \fat{v}_{\alpha} 
      \right )
      :
      (\nabla \fat{v}_{\alpha}- \nabla \vv) 
      \diff x \diff \tau 
    + \vert \fat{U}_{\alpha} \vert^2 
         \diff x \diff \tau 
\\
         + \int_s^t 
     \mathcal{K}( \vv ) 
     \left [
        \mathcal{E} (\fat{v}_{\alpha} , \fat{d}_\alpha ) - E_{\alpha} 
     \right ]  
     \diff \tau  \leq 0 
     ,
     \label{envarform-alpha}
\end{multline}
for all $s,t\in (0,T)$ and $\vv \in \Cont^1([0,T];(\V)^*) \cap L^2(0,T;\V \cap \H{2})$.
\subsection{Convergence for \texorpdfstring{$\alpha \to 0$}{vanishing Regularisation}}
Our first step is to reiterate all \textit{a priori} estimates from Section \ref{sec:a-priori-dg}. Therefore, we test Inequality \eqref{envarform-alpha} with $\vv =0$ and obtain
\begin{align}
\begin{split}
    &
    \left
    [ 
         E_{\alpha} 
     \right 
     ] 
     \Big \vert_{s-}^t 
    +\int_s^t \int_{\Omega} \mu 
     \abs{ \nabla \fat{v}_{\alpha} 
       }^2
      \diff x \diff \tau 
    + \vert \fat{U}_{\alpha} \vert^2 
         \diff x \diff \tau     
     \diff \tau  \leq 0 ,
    \\
    &
    \energy (\fat{v}_{\alpha}, \fat{d}_\alpha) (t) \leq E_{\alpha} (t)
    \text{ a.e. in }(0,T)
     . 
    \end{split}
     \label{a-priori-alpha}
\end{align}
Since the inequality
\begin{equation}
    \energystabdG (d^0) \lesssim \norm{\fat{d}_0}_{C^1 (\bar{\Omega})}
    \label{eq:bound-d0-stab-energy}
\end{equation}
holds, $E_{\alpha} (0)$ can be bounded from above independently of $\alpha>0$: 
\begin{multline}
    \label{eq:bound-E-alpha-independent-of-alpha}
E_{\alpha} (0)
=
\lim_{k,h \to 0} \frac{1}{2}\ltwonorm{v^0}^2 + \frac{1}{2}\ltwonorm{
    \adlifting{d^0}{\dbc} }^2 
    +\alpha \energystabdG (d^0)
\\
    \lesssim 
\energy (\fat{v}_0, \fat{d}_0)
+
\norm{\fat{d}_0}_{C^1 (\bar{\Omega})}
.
\end{multline}
Thereby, we used that the discrete gradient converges strongly to the gradient of the initial condition due to its regularity $\fat{d}_0 \in C^2 (\bar{\Omega})$, cf. Lemma \ref{lem:appendix-strong-consistency-gradient}.
The director Equation \eqref{eq:director-equation-alpha} is fulfilled in a weak sense such that the bounds in \eqref{a-priori-alpha} suffice to derive
\begin{equation*}
\norm{\partial_t \fat{d}_\alpha}_{ L^2(0,T; L^{3/2}(\Omega)) }
\leq  C
 \, ,
\end{equation*}
where the generic constant $C>0$ may depend on $\fat{v}_0, \fat{d}_0 $.
For the temporal derivative of the velocity field, we apply an abstract integration by parts rule \cite[Thm. II.5.12]{boyer-fabrie13} for $\vv \in \Cont^1([0,T];(\V)^*) \cap L^2(0,T;\V \cap \H{2})$ to obtain
\begin{equation*}
     -
     \left
    [ 
        \int_{\Omega} \fat{v}_{\alpha} \cdot\vv  
        \diff x 
     \right 
     ] 
     \Big \vert_{s}^t 
     +
     \int_s^t   \int_{\Omega} \fat{v}_{\alpha} \cdot \partial_t \vv \diff \tau 
     =
     -\int_s^t   \int_{\Omega} (\partial_t  \fat{v}_{\alpha} )\cdot \vv \diff \tau 
     .
\end{equation*}
Plugging this into Inequality \eqref{envarform-alpha}, estimating by above and using density arguments let us derive
\begin{multline*}
    -\int_s^t   \int_{\Omega} (\partial_t  \fat{v}_{\alpha} )\cdot \vv \diff \tau 
    \leq 
    E_{\alpha} (0)
    +
    E_{\alpha} (0) \int_s^t 
     \mathcal{K}( \vv ) 
     \diff \tau
     \\
      + \int_s^t   \int_{\Omega} 
      \abs{ 
        ( \fat{v}_{\alpha} \cdot \nabla) \fat{v}_{\alpha} \cdot  \vv 
      +\mu 
    \nabla \fat{v}_{\alpha} 
      :
      \nabla \vv
      -  ( \fat{d}_\alpha \times ( \vv \cdot \nabla ) \fat{d}_\alpha) \cdot ( \fat{U}_{\alpha})  
      }      
    \diff x \diff \tau 
\\  
    \lesssim 
    E_{\alpha} (0) 
    \left ( 
    1+
    \norm{\vv}_{L^1(0,T;\L{\infty})}
    +    
    \norm{\vv}_{L^2(0,T;\L{\infty})}
    +
    \norm{\nabla \vv}_{L^2(0,T;\L{2})}
    \right )
     \, ,
\end{multline*}
for all $\vv \in L^2(0,T;\velspace)$.
Testing with $\vv =-\tilde{ \fat w}  / \norm{\tilde{ \fat w} }$ for
\\
$ {\tilde{\fat w} \in L^2(0,T;\velspace) \backslash \{ 0 \}}$, multiplying with $\norm{\tilde{ \fat w} }$ and applying the Sobolev embedding $H^2(\Omega) \hookrightarrow \L{\infty}$ lead to the inequality
\begin{equation*}
    \int_0^T   \int_{\Omega} (\partial_t  \fat{v}_{\alpha} )\cdot \tilde{ \fat w}  \diff \tau 
    \lesssim
    E_{\alpha} (0) \norm{\tilde{ \fat w}}_{L^2(0,T; \velspace)}
    .
\end{equation*}
Since $E_{\alpha} (0)$ has an upper bound independent of $\alpha$, see Equation \eqref{eq:bound-E-alpha-independent-of-alpha}, we can even argue that there exists a constant $M>0$ independent of $\alpha$ such that
\begin{equation*}
    \int_0^T   \int_{\Omega} (\partial_t  \fat{v}_{\alpha} )\cdot \tilde{ \fat w}  \diff \tau 
    \leq
    M
    \norm{\tilde{ \fat w}}_{L^2(0,T; \velspace)}
\end{equation*}
for all $\tilde{ \fat w} \in L^2(0,T; \velspace ) $.
Using a duality argument delivers the expected estimate for the time derivative of the velocity field,
$$
\sup_{\alpha}
\norm{\partial_t \fat{v}_\alpha}_{ L^2(0,T; (\velspace )^*) }
\leq M.
$$
Reiterating converging subsequences in the prior fashion, but this time as $\alpha \to 0$, we obtain
\begin{align}
    \begin{split}
    \label{eqs:converging_subsequences-alpha}
    {\fat v}_{\alpha} &\weakstarto {\fat v} \text{ in } L^{\infty} (0,T;\Ha), 
    \\
    {\fat v}_{\alpha} &\weakto {\fat v} \text{ in }  L^2(0,T;H^1_0(\Omega)), 
    \\
    \partial_t {\fat v}_{\alpha} &\weakto \partial_t \fat{v} \text{ in } L^2(0,T; (\velspace )^*) ,
    \\
    {\fat v}_{\alpha} & \to \fat{v}  \text{ in }  L^p(0,T;\Ha)  \text{ for all } p\in[2,\infty),
    \\
    \fat{d}_{\alpha}
        &\weakstarto
            \fat{d} \text{ in } L^{\infty} (0,T;H^1(\Omega)) \cap L^{\infty} ([0,T] \times \Omega) 
            ,
    \\
    \partial_t \fat{d}_{\alpha}  &\weakto \partial_t \fat{d}  \text{ in } L^2(0,T;L^{3/2}(\Omega)),
    \\
    \fat{d}_{\alpha}
        &\to
            \fat{d} \text{ in } L^{p} (0,T;\L{p}) \text{ for all } p \in [1,\infty)
            ,
    \\
    \fat{ U}_{\alpha}
        &\weakto \fat{ U} \text{ in } L^2(0,T;\L{2}), 
    \\
    E_{\alpha} (t) & \to E (t) \text{ for all } t \in [0,T]
    ,
    \\
    E_{J, \alpha} (t) & \to E_J (t) \text{ for all } t \in [0,T]
    .
    \end{split}
\end{align}
Taking the limit in the variational Inequality \eqref{envarform-alpha} works as prior.
The boundary conditions and divergence-zero condition as well as the unit-norm constraint are all preserved.
Due to the regularity of the initial condition $\fat{d}_0 $, the gradient converges strongly and we can further argue that
\begin{multline*}
    \energy (\fat{v}_0, \fat{d}_0)
    \leq
    E_{\alpha} (0)
    =
    \lim_{k,h \to 0} \frac{1}{2}\ltwonorm{v^0}^2 + \frac{1}{2}\ltwonorm{
        \adlifting{d^0}{\dbc} }^2 
        + \alpha \energystabdG (d^0)
    \\
    =
    \frac{1}{2}\ltwonorm{\fat{v}_0}^2 + \frac{1}{2}\ltwonorm{
        \nabla \fat{d}_0 }^2 
        + \alpha E_{J,\alpha} (0)
    \leq 
    \energy (\fat{v}_0, \fat{d}_0)
    +
    \alpha E_{J,\alpha} (0)
    .  
\end{multline*}
Inequality \eqref{eq:bound-d0-stab-energy} and taking the limit $\alpha \to 0$ lead to the desired result,
$$
\energy (\fat{v}_0, \fat{d}_0)
\leq
E (0)
=
\lim_{\alpha \to 0}
E_{\alpha} (0)
\leq 
\energy (\fat{v}_0, \fat{d}_0)
+
\lim_{\alpha \to 0}
\alpha E_{J,\alpha} (0)
=
\energy (\fat{v}_0, \fat{d}_0)
.
$$
Lastly, we want to identify the quantity $\fat{U}$.
Taking the limit $\alpha \to 0$ in Inequality \eqref{eq:relative-inequality-laplacian-alpha} allows us to derive that
$$
\fat{U} = \nabla \cdot (\fat{d} \times \nabla \fat{d})
$$
in the weak sense.
\section{Computational Studies}\label{sec:computational-studies}
We introduce the coupling parameters $A, v_{el}$ as a means of controlling the scaling between the director and velocity field, i.e.~one would replace Equations \eqref{system_a} and \eqref{system_c} by
    \begin{align*}
    \partial_t \fat v - \mu \Delta \fat v  + (\fat v \cdot \nabla) \fat v  + \nabla \fat p 
    + A v_{el} ( \fat d \times \Delta \fat d ) \cdot (\fat d \times \nabla \fat d)  &= 0,
    \\
         \partial_t \fat d 
         + v_{el} (\fat v \cdot \nabla) \fat d
         + A \fat d \times (\fat d \times  \Delta \fat d )
         &=0.
    \end{align*}
The code of the implementation is publicly available \cite{mysoftware}. It is based on the python API of the Finite Element package FEniCSx \cite{fenicsx}.
For this section, we set the physical parameters if not otherwise mentioned to $\mu, v_{el}, A=1$. The regularisation parameter for Algorithm \ref{algo:dg} is set to $\alpha =0.005$ if not otherwise mentioned.
After the projection step, the nodal or elementwise unit-sphere constraint is observed to be fulfilled almost up to machine precision.
\subsection{Magical Spiral\texorpdfstring{, e.g.~{\cite{badia-et-al11,de_gennes90}}}{}}
\label{sec:spiral}
In this numerical experiment, we consider a centered disk with a hole as domain $\Omega = \{x \in \mathbb{R}^2 :  1 < \abs{x} < 2\}$. We define the inner boundary by $\Gamma_i = \{x \in \mathbb{R}^2 :  \abs{x} =1 \}$ and the outer boundary by $\Gamma_o = \{x \in \mathbb{R}^2 :  \abs{x} =2 \}$. We apply \textit{no-slip} (homogeneous) Dirichlet boundary conditions on the velocity field. The constant-in-time Dirichlet boundary conditions for the director field are given by
$$
\fat{d} (x) = \frac{(x_1,x_2)^T}{\abs{(x_1,x_2)^T}}\text{ for all } x\in \Gamma_i
, \quad
\fat{d} (x) =  \frac{(x_2,-x_1)^T}{\abs{(x_1,x_2)^T}}\text{ for all } x\in \Gamma_o.
$$
The initial values are set to $\fat{v}_0 (x) = 0$, $\fat{d}_0 (x) = {x}/{\abs{x}}$ for all $x\in\Omega$.
In this setting an exact stationary solution is known \cite[p. 158 ff.]{de_gennes90} and can be described in terms of the angle $\psi$ between the local director field and the radial direction, given for $x \in \Omega$ by 
\begin{equation}
    \label{def:psi}
    \psi (x) = \frac{\pi}{2} \frac{\log (\abs{x})}{\log (2)}.
\end{equation}
The evolution of the director and velocity field can be seen for both algorithms in Figure \ref{fig:spiral}. We observe that the director field finds a stable, spiral like configuration in order to agree with the boundary conditions. This induces two similar velocity fields.
\\
The convergence of both algorithms can be confirmed in Table \ref{tab:error-magical-spiral} and Figure \ref{fig:error-magical-spiral}.
In Figure \ref{fig:spiral-comparison}, we observe the effect of omitting the projection step and the choice of the regularization parameter $\alpha$. For both algorithms, we can see that excluding the projection step diminishes the accuracy of the method significantly. 
In Figure \ref{fig:spiral-comparison-dg}, we see that the projection step in Algorithm \ref{algo:dg} allows to choose a much smaller, but not arbitrarily small regularisation parameter $\alpha$ while keeping the same accuracy: Both, the choice $\alpha=0.1$ and $\alpha=0.01$ deliver the highest accuracy, when the projection step is present. In contrast, when the projection is omitted, a higher regularization leads to better results.
In general, the outcome is less sensitive with respect to the parameter $\alpha$, when the projection step is included.
Figure \ref{fig:spiral-comparison-cg} shows that Algorithm \ref{algo:cg} reaches the same accuracy as the fixed-point solver in \cite{lasarzik-reiter23,reiter-ecmi} at a much coarser temporal discretization, see also Table \ref{tab:error-magical-spiral}. 
Overall, the continuous approach of Algorithm \ref{algo:cg} seems to be more efficient than Algorithm \ref{algo:dg}.
\subsection{Annihilation of two Defects\texorpdfstring{, e.g.~{\cite{becker-feng-prohl08,lasarzik-reiter23}}}{}}
\label{sec:defects}
As second benchmark, we consider the unit cube $\Omega = (-0.5, 0.5)^3$.
We model two so called defects --- discontinuities of the director field --- along the lines $\{x\in \mathbb{R}^3 : \abs{x_1}=0.25, x_2 = 0\}$. Since the interpolation points of the spaces $\dgzero$ and $\cgone$ do not coincide we enlarge the location of the discontinuities by an $h$-tube. This makes sure that the discontinuities are present in both discrete initial conditions.
The initial director field then reads as
\begin{align*}
\fat d_0 
=
\begin{cases}
    (0,0,1)^T
    &\text{ if } (x_1,x_2)^T \in B_h ((0.25,0)^T) \cup B_h ((-0.25,0)^T) \\
    \frac{\tilde{\fat d}_0}{\abs{\tilde{\fat d}_0}}, 
    &\text{ otherwise with }
    \tilde{ \fat d}_0 (x) = (4 x_1^2+4 x_2^2 - 0.25, 2 x_2, 0)^T \, ,
\end{cases}
\end{align*}
where $B_h (x)$ is the open circle with radius $h$ and midpoint $x$ in two spatial dimensions.
We accordingly employ constant-in-time inhomogeneous Dirichlet boundary conditions for the director field prescribed by $\fat d_0$.
The initial velocity field is chosen to be $\fat v_0 = 0$ with \textit{no-slip} homogeneous Dirichlet boundary conditions.
We further set the parameter $v_{el} = 0.25$. 
\\
The results can be found in Figure \ref{fig:defects}. 
Both algorithms show the same qualitative behaviour: The defects in the director field slowly vanish. The director field takes a stable long-term configuration. In both cases, the induced velocity field shows four swirls, see Figure \ref{fig:defects-velocity}.
This is in agreement with the mentioned references.
The quantitative difference between the two Algorithms can be explained however: The interpolation points used for the initial condition do not coincide such that the defects have a different support (see Figure \ref{fig:defects-1}) and therewith a different influence strength on the dynamics of the evolution. 
\begin{figure}
    \centering
    \begin{subfigure}[b]{.49\textwidth}
        \centering
        \includegraphics*[width = \textwidth]{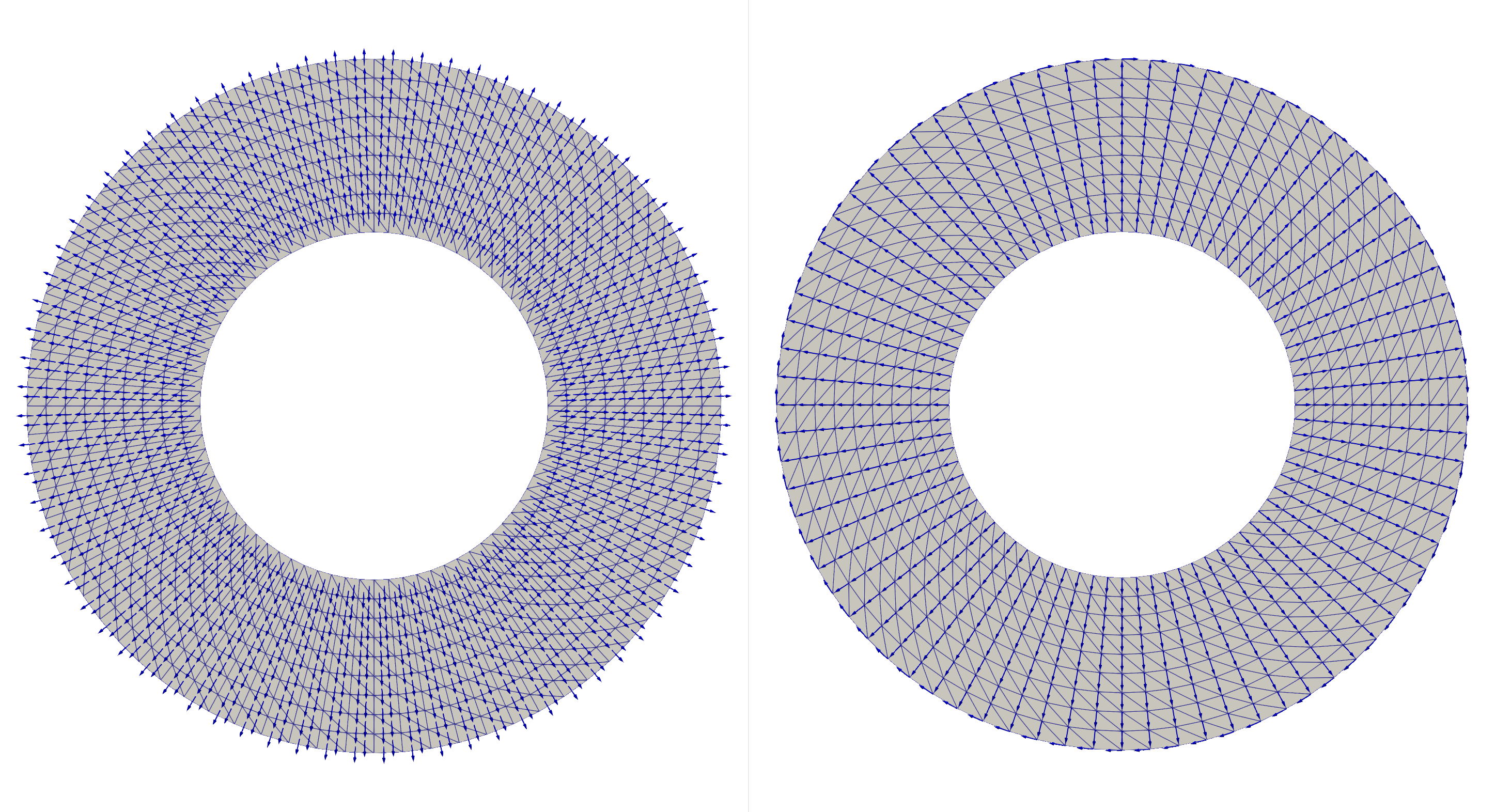} 
        \caption{Director field, $t=0$.}       
    \end{subfigure}
    \begin{subfigure}[b]{.49\textwidth}
        \centering
        \includegraphics*[width = \textwidth]{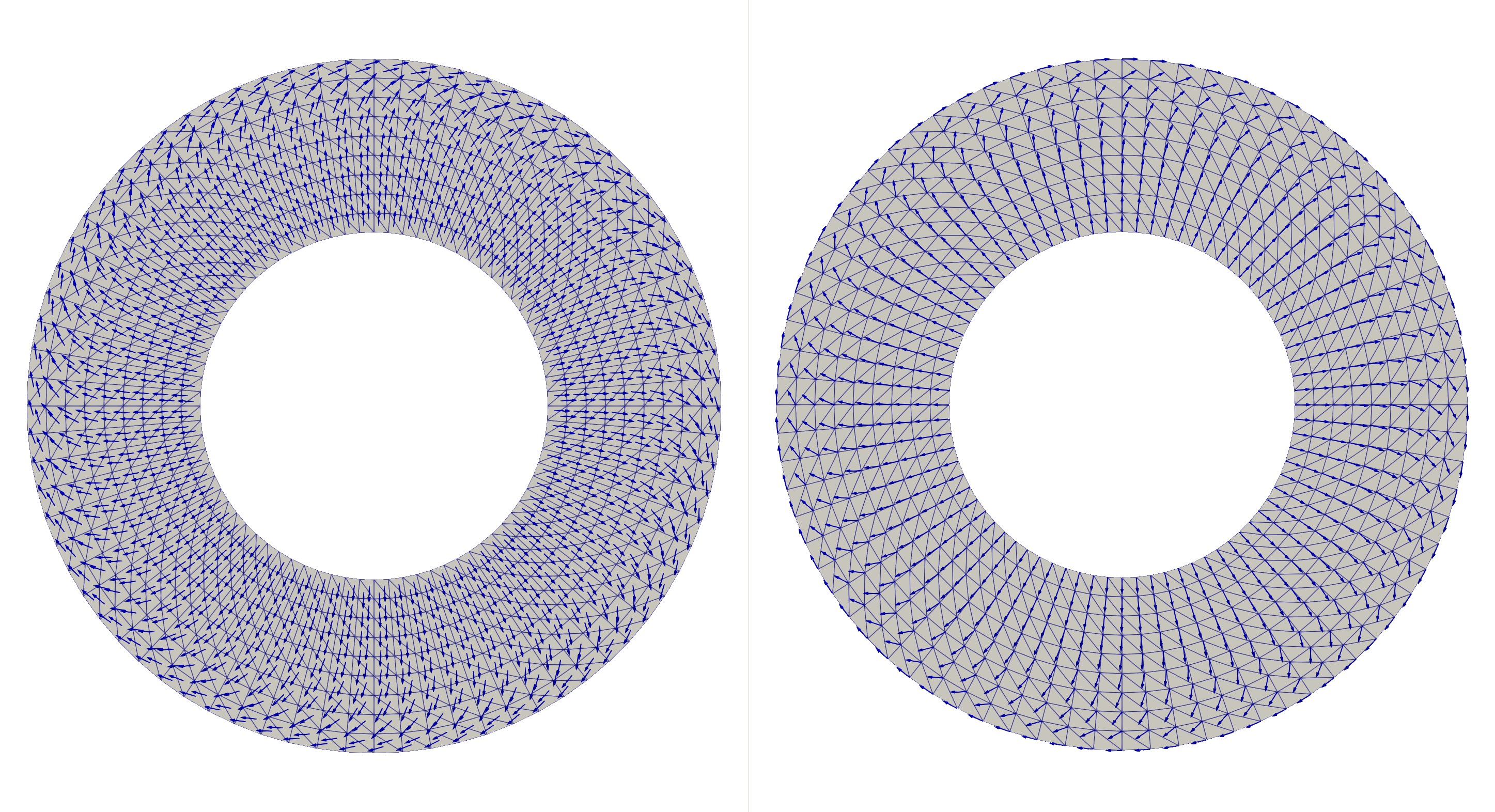}    
        \caption{Director field, $t=0.05$.}       
    \end{subfigure}
    \begin{subfigure}[b]{.49\textwidth}
        \centering
        \includegraphics*[width = \textwidth]{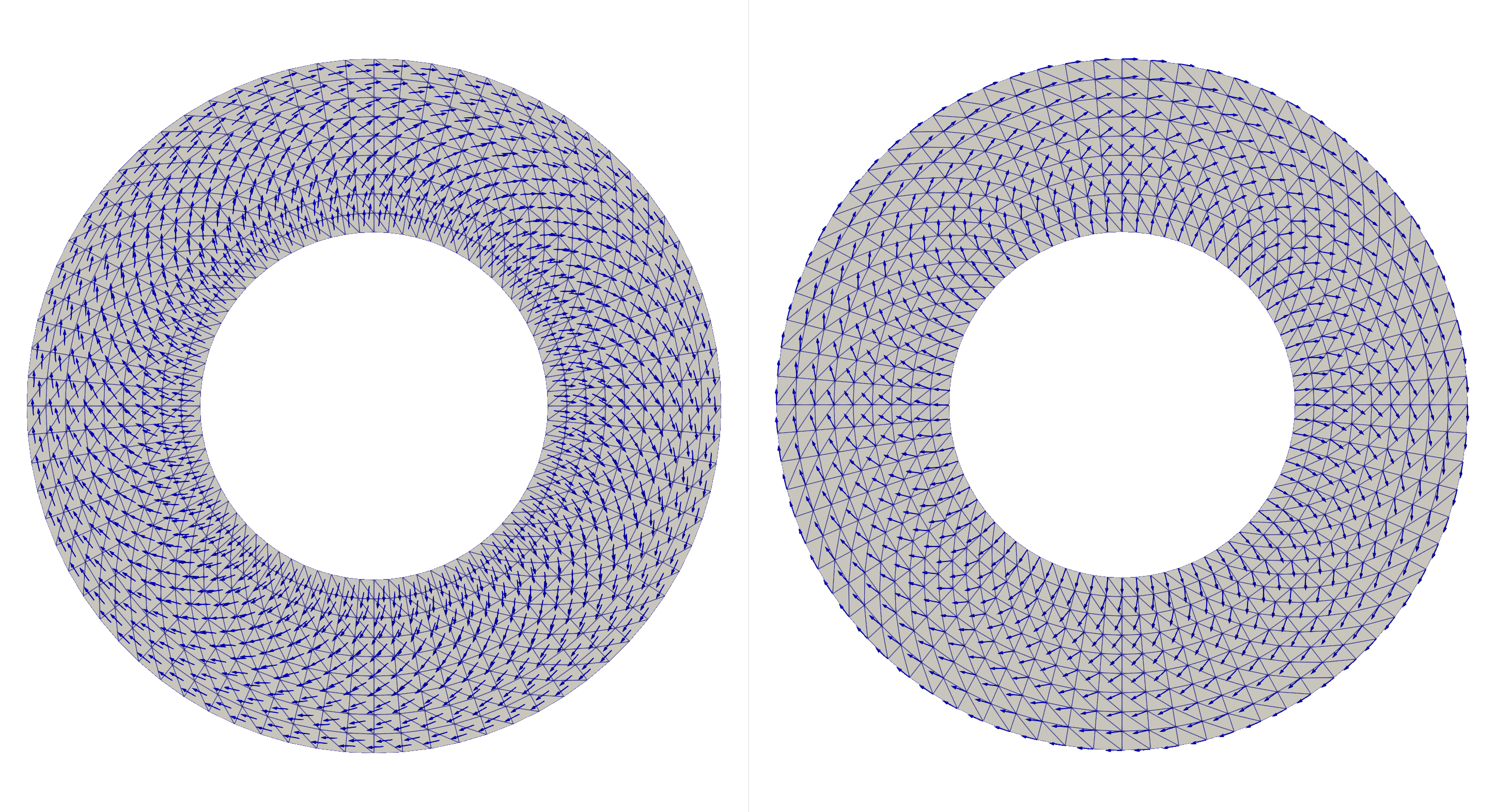}        
        \caption{Director field, $t=1.5$.}   
    \end{subfigure}
    \begin{subfigure}[b]{.49\textwidth}
        \centering
        \includegraphics*[width = \textwidth]{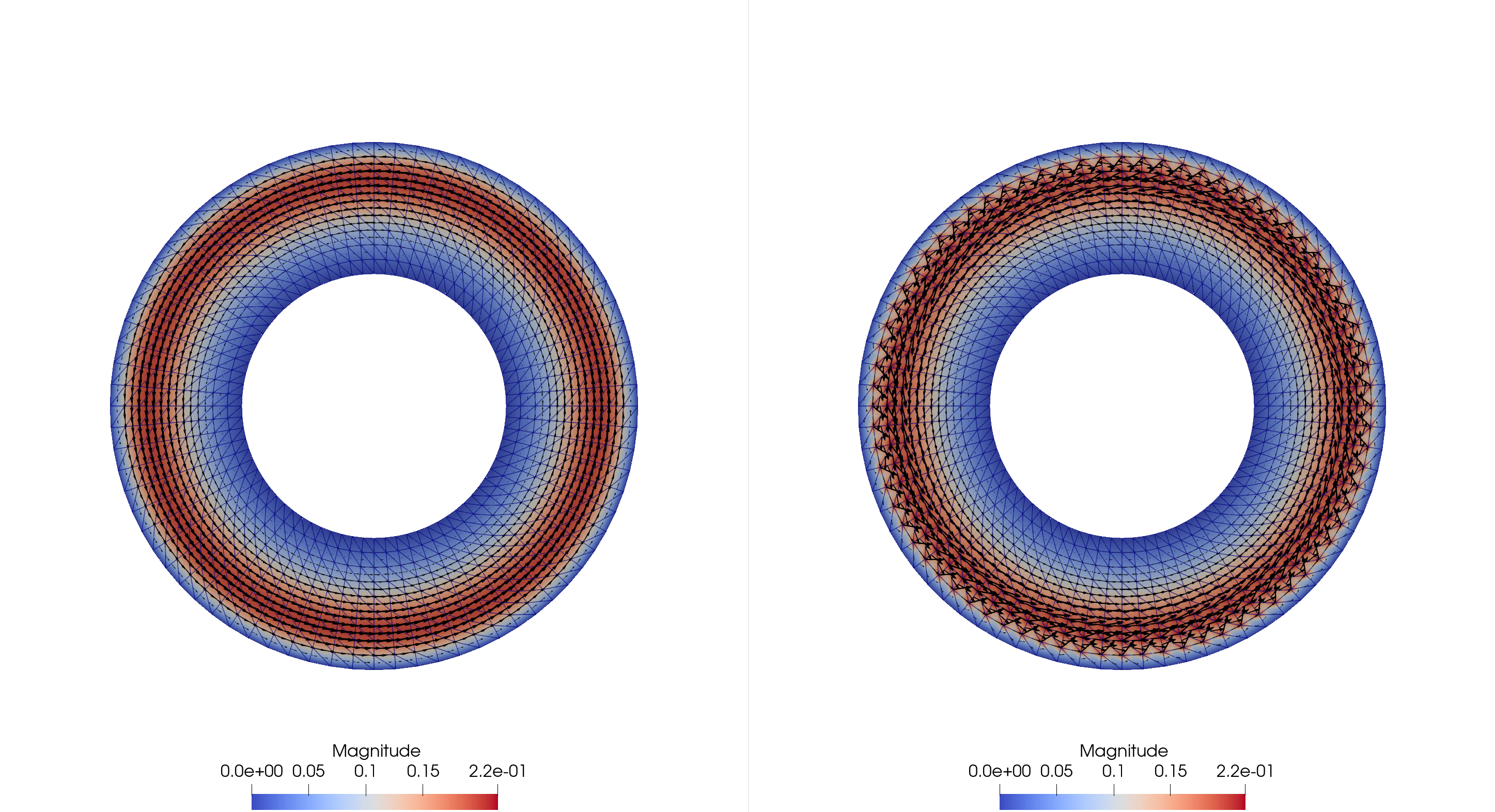}        
        \caption{Velocity field, $t=0.05$.}   
    \end{subfigure}
    \caption{Experiment \ref{sec:spiral}: Evolution of Algorithm \ref{algo:dg} (left) and Algorithm \ref{algo:cg} (right) for $h=0.1$ and $k=0.01$}
    \label{fig:spiral}
\end{figure}
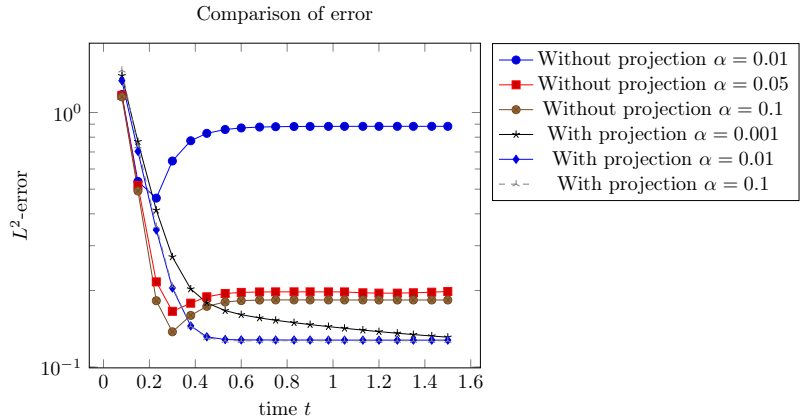
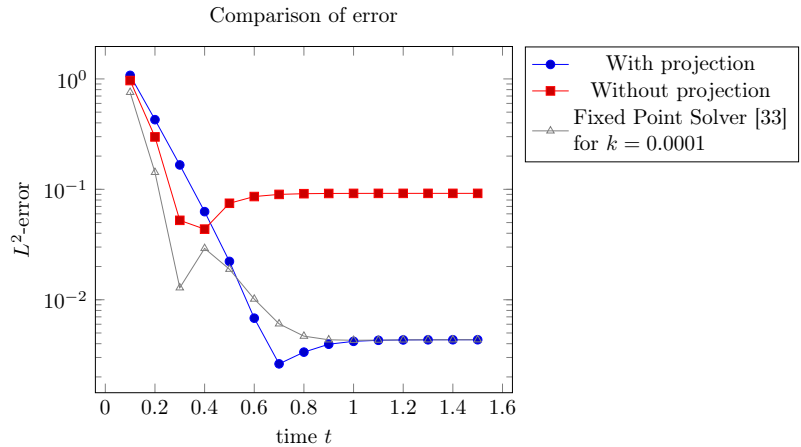
\begin{figure}
    \centering
    \begin{subfigure}[b]{\textwidth}
        \resizebox{.85 \textwidth}{!}{
        \begin{tikzpicture}
        \begin{axis}[
                title={Comparison of error},
                enlargelimits=true, 
                xlabel =time $t$,
                ylabel =$L^2$-error,
                ymode = log,
                legend pos=outer north east,
                legend style={cells={align=left}},
            ]
            \addplot 
            table
            [x=time, y=ldga01, col sep=comma]{spiral-dg-comp.csv};
            \addlegendentry{Without projection $\alpha = 0.01$}
            \addplot 
            table
            [x=time, y=ldga05, col sep=comma]{spiral-dg-comp.csv};
            \addlegendentry{Without projection $\alpha = 0.05$} 
            \addplot 
            table
            [x=time, y=ldga1, col sep=comma]{spiral-dg-comp.csv};
            \addlegendentry{Without projection $\alpha = 0.1$}           
            \addplot 
            table
            [x=time, y=lpdga001, col sep=comma]{spiral-dg-comp.csv};
            \addlegendentry{With projection $\alpha = 0.001$} 
            \addplot 
            table
            [x=time, y=lpdga01, col sep=comma, mark=o]{spiral-dg-comp.csv};
            \addlegendentry{With projection $\alpha = 0.01$} 
            \addplot[mark=Mercedes star, dashed, color=gray] 
            table
            [x=time, y=lpdga1, col sep=comma]{spiral-dg-comp.csv};
            \addlegendentry{With projection $\alpha = 0.1$} 
        \end{axis}
        \end{tikzpicture}
        }
        \caption{Algorithm \ref{algo:dg} with and without projection step for different choices of parameter $\alpha$.}        
        \label{fig:spiral-comparison-dg}
    \end{subfigure}
    \begin{subfigure}[b]{\textwidth}
    \resizebox{.85 \textwidth}{!}{
        \begin{tikzpicture}
        \begin{axis}[
                title={Comparison of error},
                enlargelimits=true, 
                xlabel =time $t$,
                ylabel =$L^2$-error,
                ymode = log,
                legend pos=outer north east,
                legend style={cells={align=left}},
            ]
            \addplot 
            table
            [x=time, y=LhP, col sep=comma]{spiral-comparison.csv};
            \addlegendentry{With projection} 
            \addplot 
            table
            [x=time, y=Lh, col sep=comma]{spiral-comparison.csv};
            \addlegendentry{Without projection} 
            \addplot[color=gray, mark = triangle] 
            table
            [x=time, y=FP, col sep=comma]{spiral-comparison.csv};
            \addlegendentry{Fixed Point Solver \cite{lasarzik-reiter23}\\for $k= 0.0001$} 
        \end{axis}
        \end{tikzpicture}
        }
        \caption{Algorithm \ref{algo:cg} with and without projection step compared to a non-linear method \cite{lasarzik-reiter23}.}
        \label{fig:spiral-comparison-cg}
    \end{subfigure}    
        \caption{Error comparison for Experiment \ref{sec:spiral}. The error is computed as difference of the angles of the director field in the $L^2(\Omega)$-norm. We set $h=0.1$ and $k=0.01$ if not mentioned otherwise.}
        \label{fig:spiral-comparison}
\end{figure}
\begin{table}
    \centering
    \begin{tabular}{c|ccccc}
        \toprule
        k &    0.0005 &    0.0010 &    0.0050 &    0.0100 &    0.0500 \\
        h      &           &           &           &           &           \\
        \midrule
        0.1000 
        &  0.004342
        &  0.004335
        &  0.004342
        &  0.004344
        &  0.004327
        \\
        0.0625 
        &  0.001621
        &  0.001624
        &  0.001623
        &  0.001623
        &  0.001609
        \\
        0.0500 
        &  0.001029
        &  0.001014
        &  0.001023
        &  0.001024
        &  0.001011
        \\
        0.0250 
        &  0.000249
        &  0.000249
        &  0.000249
        &  0.000250
        &  0.000237
        \\
        \bottomrule
        \end{tabular}
    \bigskip
    \begin{tabular}{c|ccccc}
    \toprule
    k &    0.0005 &    0.0010 &    0.0050 &    0.0100 &    0.0500 \\
    h      &           &           &           &           &           \\
    \midrule
    0.1000 
    & 0.116146
    & 0.114545
    & 0.115187
    & 0.115199
    & 0.115271
    \\
    0.0625 
    & 0.069534
    & 0.069583
    & 0.069783
    & 0.068301
    & 0.069548
    \\
    0.0500 
    & 0.055044
    & 0.054834
    & 0.055043
    & 0.055042
    & 0.055047
    \\
    0.0250 
    & 0.027896
    & 0.026979
    & 0.026997
    & 0.026997
    & 0.027001
    \\
    \bottomrule
    \end{tabular}
    \caption{$L^2$-error of the approximate solution compared to the exact stationary solution at $t=1.5$
    for Algorithm \ref{algo:cg} (top) and Algorithm \ref{algo:dg} with $\alpha = 0.005$ (bottom).}    
    \label{tab:error-magical-spiral}
\end{table}
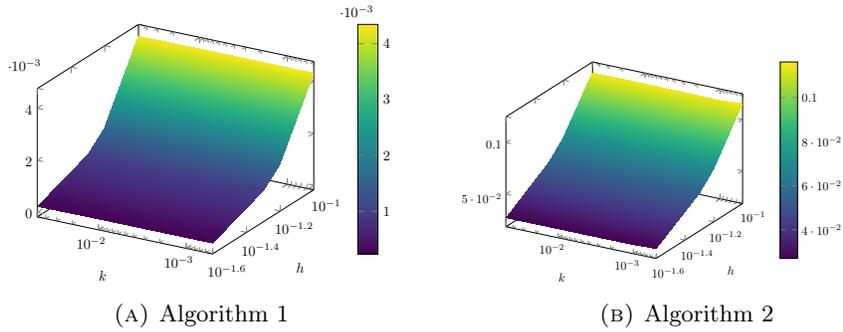
\begin{figure}
    \begin{subfigure}[b]{0.49\textwidth}
        \centering
            \resizebox{.85 \textwidth}{!}{
            \begin{tikzpicture}
                \begin{axis}[
                    mesh/ordering=y varies,
                    colormap/viridis,
                    view={-60}{30},
                    colorbar,
                    ymode=log,
                    xmode=log,
                    xlabel={$h$},
                    ylabel={$k$}
                    ]           
                    \addplot3 [
                    surf,
                    shader=interp, 
                    ] coordinates {
                    (0.1000 ,0.0005, 0.004341979)   (0.1000 ,0.0010, 0.004334837)   (0.1000 ,0.0050, 0.004342155)   (0.1000 ,0.01, 0.004343988) (0.1000 ,0.0500, 0.004326729)

                    (0.0625 ,0.0005, 0.001621184)   (0.0625,0.0010,   0.00162383)  (0.0625,  0.0050, 0.001622806)  (0.0625,  0.01,  0.00162316)   (0.0625, 0.0500,  0.001608931)

                    (0.0500 ,0.0005, 0.001029153)   (0.0500 ,0.0010, 0.001013531)   (0.0500 , 0.0050, 0.001023478)  (0.0500 ,0.01, 0.001024277) (0.0500 ,0.0500, 0.001010666)

                    (0.0250 ,0.0005, 0.000248616)   (0.0250 ,0.0010, 0.000249012)   (0.0250 , 0.0050, 0.000249487)  (0.0250 ,0.01, 0.000249856) (0.0250 ,0.0500, 0.000237118)
                    };
                \end{axis}
                \end{tikzpicture} 
            }
        \caption{Algorithm \ref{algo:cg}}
    \end{subfigure}
    \hspace*{\fill}
    \begin{subfigure}[b]{0.49\textwidth}
        \resizebox{.85 \textwidth}{!}{
        \begin{tikzpicture}
            \begin{axis}[
                mesh/ordering=y varies,
                    colormap/viridis,
                    view={-60}{30},
                    colorbar,
                    ymode=log,
                    xmode=log,
                    xlabel={$h$},
                    ylabel={$k$}
                ]           
                \addplot3 [
                surf,
                shader=interp, 
                ] coordinates {
                (0.1000 ,0.0005, 0.116145603)  (0.1000 ,0.0010, 0.114545099)   (0.1000 ,0.0050, 0.115187124)   (0.1000 ,0.01, 0.115199259) (0.1000 ,0.0500, 0.115271483)

                (0.0625 ,0.0005, 0.069534)  (0.0625,0.0010,   0.069583282)  (0.0625,  0.0050, 0.069783495)  (0.0625,  0.01,  0.068300686)   (0.0625, 0.0500,  0.069547692)

                (0.0500 ,0.0005, 0.055043856)  (0.0500 ,0.0010, 0.054834209)   (0.0500 , 0.0050, 0.055042601)  (0.0500 ,0.01, 0.055042493) (0.0500 ,0.0500, 0.055046717)

                (0.0250 ,0.0005, 0.02789639)   (0.0250 ,0.0010, 0.026978605)   (0.0250 , 0.0050, 0.026996926969009686)  (0.0250 ,0.01, 0.02699668) (0.0250 ,0.0500, 0.027000759)
                };
            \end{axis}
            \end{tikzpicture} 
        }
        \caption{Algorithm \ref{algo:dg}}
    \end{subfigure}
    \caption{$L^2$-error of the approximate solution compared to the exact stationary solution at $t=1.5$. Visualization of Table \ref{tab:error-magical-spiral}.}
    \label{fig:error-magical-spiral}
\end{figure}
\begin{figure}
    \centering
    \begin{subfigure}[b]{0.49\textwidth}
        \centering
        \includegraphics*[width =\textwidth]{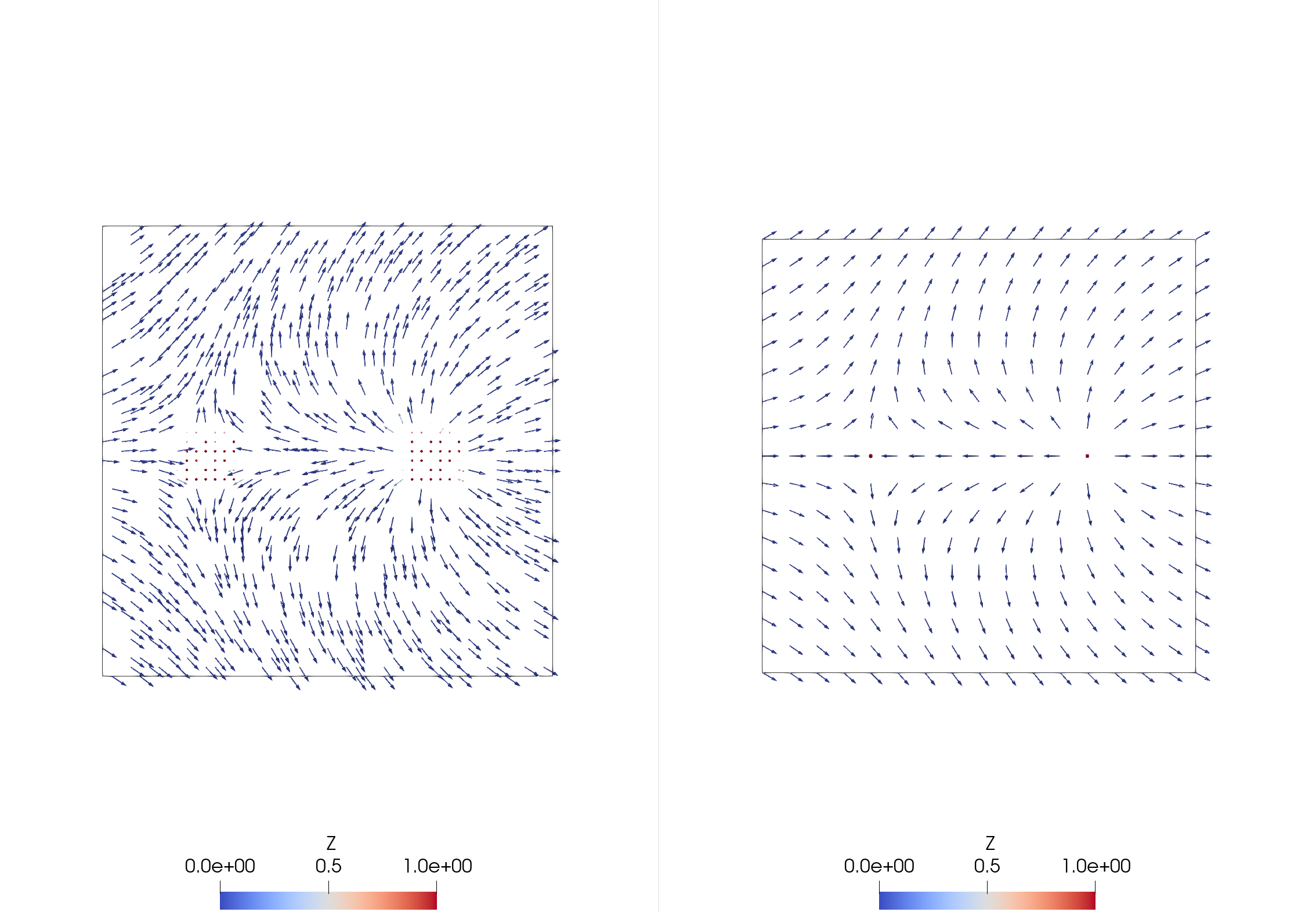}
        \caption{Director field, $t=0$. The color marks the (absolute) value of the z-component.}
        \label{fig:defects-1}
    \end{subfigure}
    \begin{subfigure}[b]{0.49\textwidth}
        \centering
        \includegraphics*[width = \textwidth]{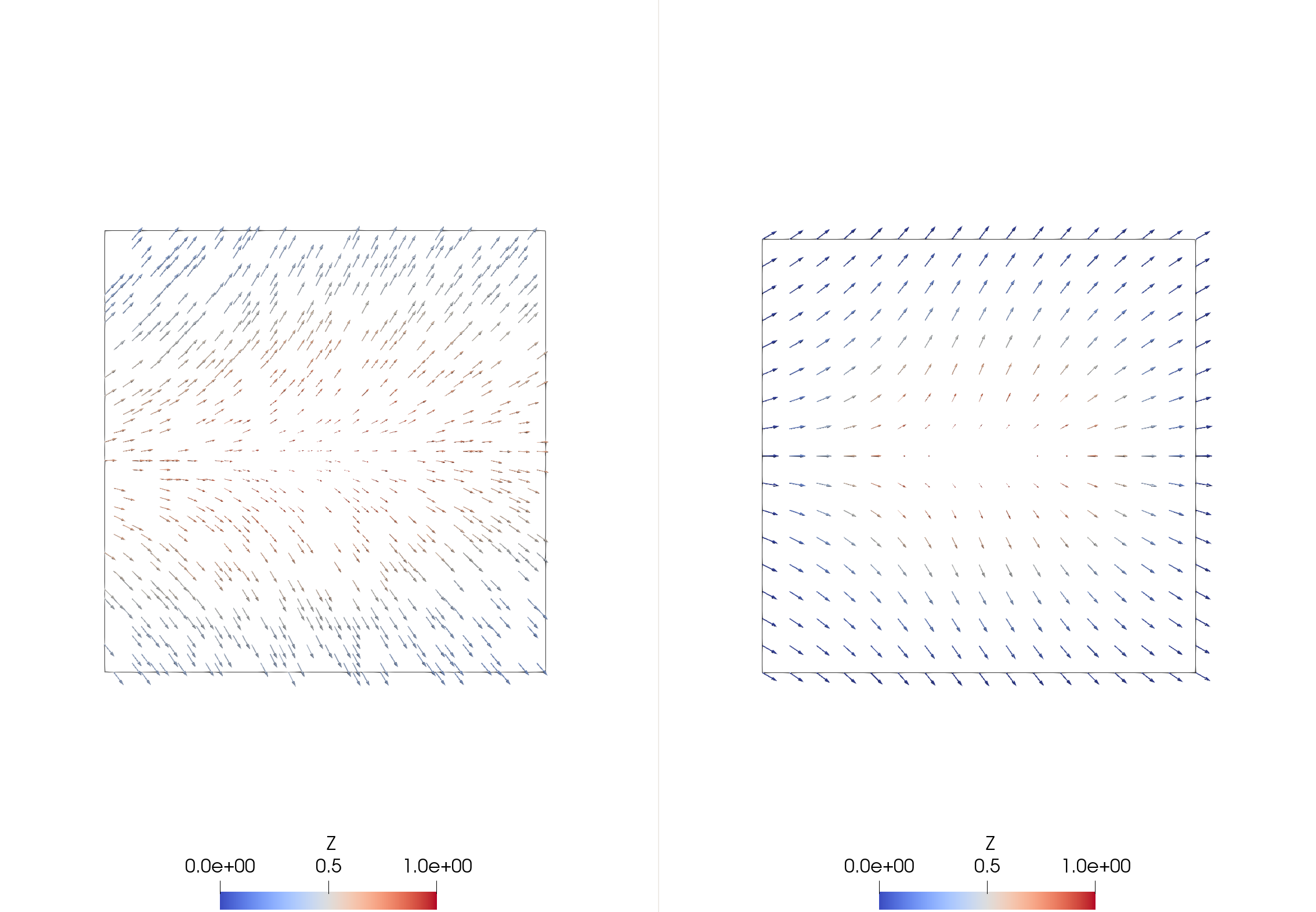}
        \caption{Director field, $t=0.05$. The color marks the (absolute) value of the z-component.}
    \end{subfigure}
    \begin{subfigure}[b]{0.49\textwidth}
        \centering
        \includegraphics*[width =\textwidth]{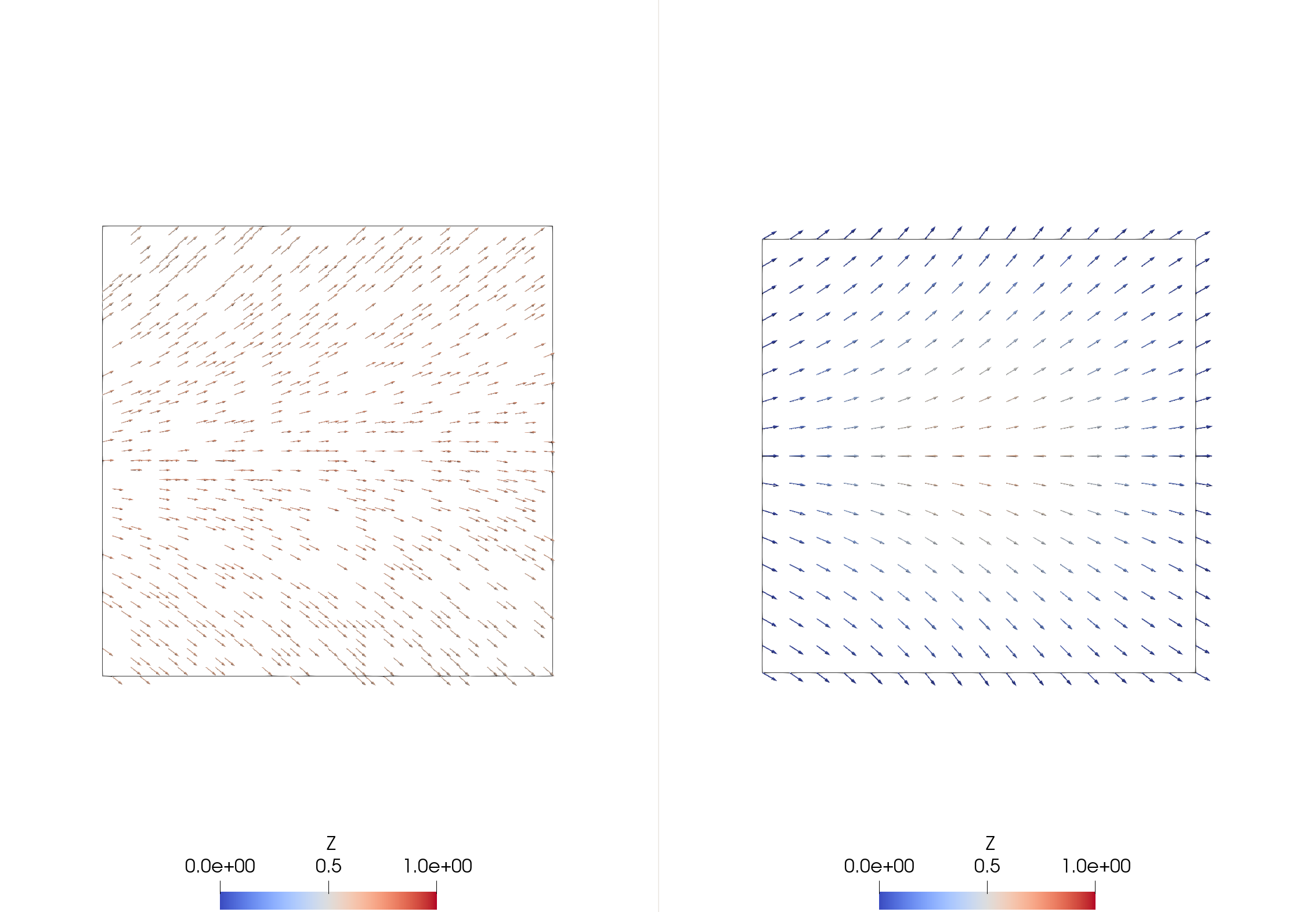}
        \caption{Director field, $t=1.0$. The color marks the (absolute) value of the z-component.}
    \end{subfigure}
    \begin{subfigure}[b]{0.49\textwidth}
        \centering
        \includegraphics*[width = \textwidth]{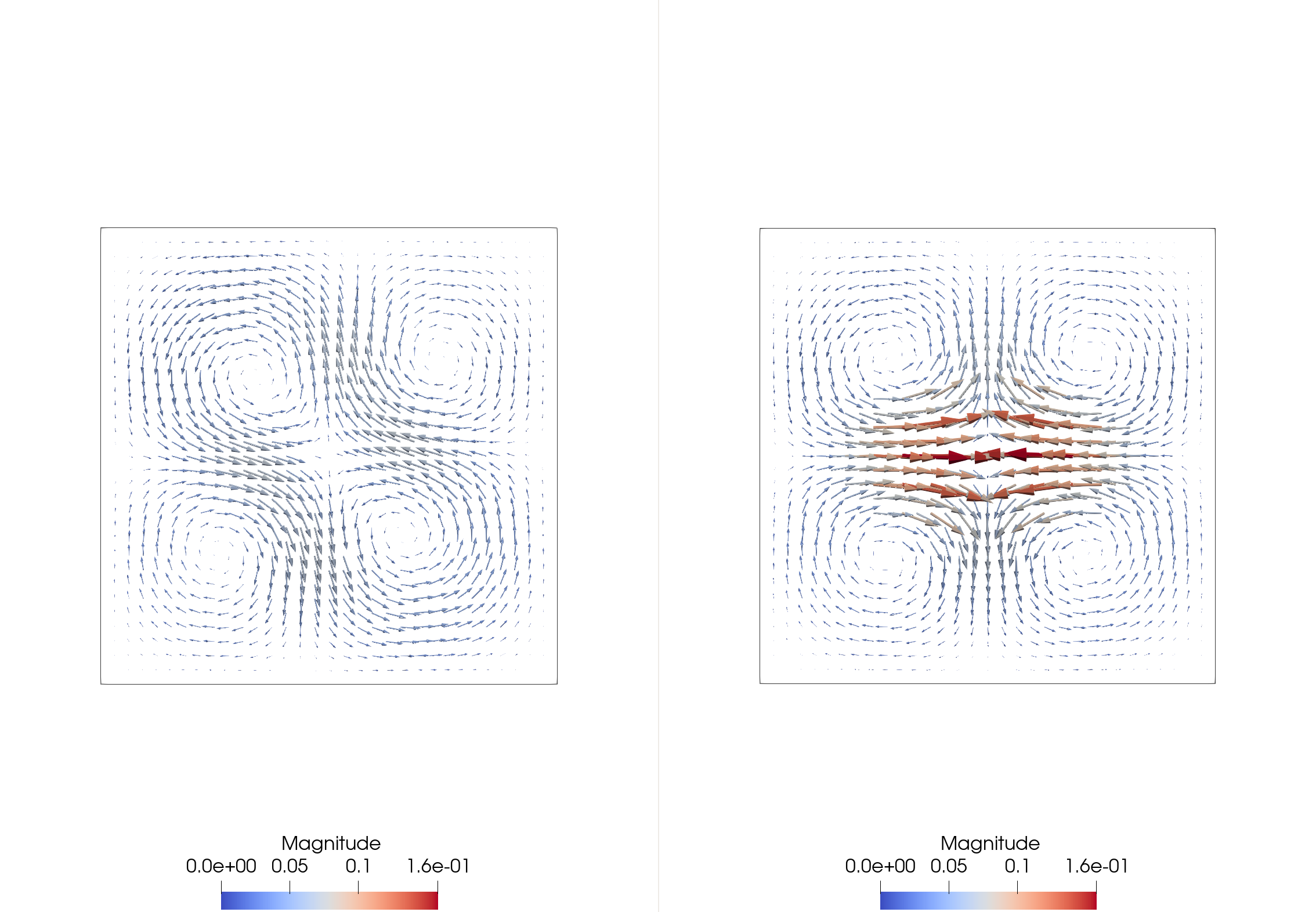}
        \caption{Velocity field, $t=0.03$. The color marks the magnitude.}
        \label{fig:defects-velocity}
    \end{subfigure}
    \caption{Experiment \ref{sec:defects}: Evolution of the director and velocity fields of the approximate solutions of Algorithm \ref{algo:dg} (left) and Algorithm \ref{algo:cg} (right) in the plane $z=0$ for $k=0.001, h=1/16$. The arrows show only the first two components of the vector fields.}
    \label{fig:defects}
\end{figure}
\section{Acknowledgements}
In the first place, I would like to thank Dr.~Robert Lasarzik for his guidance, support, and the fruitful discussions throughout the research process of this article.
Further, I gratefully acknowledge the financial support received in the form of a Ph.D. scholarship from the \textit{Friedrich-Naumann-Foundation for Freedom} (dt.: Friedrich-Naumann-Stiftung für die Freiheit) with funds from the Federal Ministry of Education and Research (BMBF) and funding by the Deutsche Forschungsgemeinschaft (DFG, German Research Foundation) under Germany's Excellence Strategy - \textit{The Berlin Mathematics Research Center MATH+} (EXC-2046/1, project ID: 390685689).
Lastly, I would like to thank the \textit{Weierstrass Institute for Applied Analysis and Stochastics} for being given the opportunity to use their high-performance computers for the numerical experiments considered in this work.
\printbibliography
\end{document}